\documentclass{amsart}
\usepackage{amsfonts,amssymb,amscd,amsmath,latexsym,amsbsy,mathrsfs,amsxtra}
\usepackage{mathtools,tikz-cd}
\usetikzlibrary{cd}
\usetikzlibrary{arrows}
\usetikzlibrary{matrix}
\usepackage{graphicx,ctable,booktabs}
\usepackage[all]{xy}
\usetikzlibrary{arrows.meta}
\numberwithin{equation}{section}            
\theoremstyle{plain}
\newtheorem{thm}{Theorem}[section]

\newtheorem{prop}[thm]{Proposition}

\newtheorem{defi}[thm]{Definition}
\newtheorem{lem}[thm]{Lemma}
\newtheorem{cor}[thm]{Corollary}
\newtheorem{eg}[thm]{{Example}}

\theoremstyle{remark}
\newtheorem{rema}[thm]{Remark}

\newcommand{\Ad}{{\mbox{\upshape{Ad}}}}

\newcommand{\Aut}{\mathrm{Aut}}

\newcommand{\bc}{{\mathbf{c}}}

\newcommand{\Btil}{{\widetilde{B}}}
\newcommand{\bfrak}{{\mathfrak b}}
\newcommand{\Bfrak}{{\mathfrak B}}

\newcommand{\N}{{\mathbb N}}

\newcommand{\cF}{{\mathcal F}}

\newcommand{\cI}{{\mathcal I}}

\newcommand{\cJ}{{\mathcal J}}
\newcommand{\cK}{{\mathcal K}}

\newcommand{\cR}{{\mathcal R}}
\newcommand{\cS}{{\mathcal S}}
\newcommand{\cU}{{\mathcal U}}
\newcommand{\cYD}{{\mathcal YD}}
\newcommand{\cop}{\mathrm{cop}}

\newcommand{\End}{\mathrm{End}}
\newcommand{\Etil}{\widetilde{E}}

\newcommand{\field}{{k}}

\newcommand{\gfrak}{{\mathfrak g}}
\newcommand{\ufofrak}{{\mathfrak{ufo}}}

\newcommand{\gr}{{\mathrm{gr}}}
\newcommand{\Heis}{\mathrm{Heis}}

\newcommand{\Hom}{{\mathrm{Hom}}}

\newcommand{\Ifrak}{\mathfrak{I}}
\newcommand{\id}{{\mathrm{id}}}

\newcommand{\kfrak}{{\mathfrak k}}

\newcommand{\kow}{{\varDelta}}
\newcommand{\lact}{{\triangleright}}

\newcommand{\mx}{\mathrm{max}}

\newcommand{\okow}{{\overline{\varDelta}}}

\newcommand{\oB}{{\overline{B}}}
\newcommand{\oF}{{\overline{F}}}

\newcommand{\oP}{{\overline{P}}}

\newcommand{\opi}{\overline{\pi}}

\newcommand{\op}{\mathrm{op}}
\newcommand{\ot}{\otimes}
\newcommand{\osigma}{\overline{\sigma}}

\newcommand{\poly}{{\mathrm{poly}}}

\newcommand{\pfrak}{{\mathfrak p}}

\newcommand{\rel}{{\mathrm{rel}}}

\newcommand{\ract}{\triangleleft}

\newcommand{\slfrak}{{\mathfrak{sl}}}

\newcommand{\uqg}{{U_q(\mathfrak{g})}}

\newcommand{\Util}{\widetilde{U}}
\newcommand{\Hmin}{{\mathrm{Heis}}(\chi)\spcheck}
\newcommand{\Upoly}{U(\chi)^{\mathrm{poly}}}
\newcommand{\uB}{\underline{B}}
\newcommand{\uF}{\underline{F}}
\newcommand{\uE}{\underline{E}}

\newcommand{\vep}{\varepsilon}

\newcommand{\Z}{{\mathbb Z}}
\newcommand{\wt}{\widetilde}
\newcommand{\wh}{\widehat}
\begin{document}
\title[Symmetric pairs for Nichols algebras of diagonal type]
{Symmetric pairs for Nichols algebras of diagonal type via star products}
\author[Stefan Kolb]{Stefan Kolb}
\address{School of Mathematics, Statistics and Physics,
Newcastle University, Newcastle upon Tyne NE1 7RU, United Kingdom}
\email{stefan.kolb@newcastle.ac.uk}
\author[Milen Yakimov]{Milen Yakimov}
\address{
Department of Mathematics \\
Louisiana State University \\
Baton Rouge, LA 70803 \\
U.S.A.
}
\email{yakimov@math.lsu.edu}
%\urladdr{https://www.math.lsu.edu/~yakimov}

%\thanks{}
%\date{}


\begin{abstract}
  We construct symmetric pairs for Drinfeld doubles of pre-Nichols algebras of diagonal type and determine when they possess an Iwasawa decomposition. This extends G.~Letzter's theory of quantum symmetric pairs. Our results can be uniformly applied to Kac--Moody quantum groups for a generic quantum parameter, for roots of unity in respect to both big and small quantum groups, to quantum supergroups and to exotic quantum groups of ufo type. We give a second construction of symmetric pairs for Heisenberg doubles in the above generality and prove that they always admit an Iwasawa decomposition.

  For symmetric pair coideal subalgebras with Iwasawa decomposition in the above generality we then address two problems which are fundamental already in the setting of quantum groups. Firstly, we show that the symmetric pair coideal subalgebras are isomorphic to intrinsically defined deformations of partial bosonizations of the corresponding pre-Nichols algebras.
  To this end we develop a general notion of star products on $\N$-graded connected algebras which provides an efficient tool to prove that two deformations of the partial bosonization are isomorphic. The new perspective also provides an effective algorithm for determining the defining relations of the coideal subalgebras. 

  Secondly, for Nichols algebras of diagonal type, we use the linear isomorphism between the coideal subalgebra and the partial bosonization to give an explicit construction of quasi $K$-matrices as sums over dual bases. We show that the resulting quasi $K$-matrices give rise to weakly universal $K$-matrices in the above generality.
\end{abstract}
\maketitle

%%%%%%%%%%%%%%%%%%%%%%%%%%%%%%%%%%%%
\section{Introduction}
%%%%%%%%%%%%%%%%%%%%%%%%%%%%%%%%%%%%
\subsection{(Pre-)Nichols algebras of diagonal type}
%%%%%%%%%%%%%%%%%%%%%%%%%%%%%%%%%%%%
Since their inception in the 1980s quantum groups have become an integral part of representation theory with many deep applications. Quantum groups in particular reinvigorated the general investigation of Hopf algebras as they provided many new noncommutative, noncocommutative examples. In the late 1990s N.~Andruskiewitsch and H.-J.~Schneider proposed an approach to the classification of finite dimensional, pointed Hopf algebras \cite{MSRI-AS02}. In this approach a central role is played by Nichols algebras which are Hopf algebras in a braided category of Yetter-Drinfeld modules. Important examples of Nichols algebras include the positive part $U^+$ of quantized enveloping algebras $\uqg$ for $q$ not a root of unity, and the positive part of the small quantum group $\mathfrak{u}_q(\gfrak)$ if $q$ is a root of unity. Other examples come from quantum Lie superalgebras, but there are also large example classes which had not been studied previously.

The starting point for the construction of a Nichols algebra $\Bfrak(V)$ is a Hopf algebra $H$ and a Yetter-Drinfeld module $V$ over $H$. If $H$ is the group algebra of an abelian group and $V$ is semisimple and finite rank, the $\Bfrak(V)$ is called a Nichols algebra of diagonal type. Nichols algebras of diagonal type are determined by a bicharacter $\chi:\Z^n\times \Z^n \rightarrow \field$ into the base field $\field$. The finite dimensional such Nichols algebras were classified by I.~Heckenberger in \cite{a-Heck09}. The Nichols algebra $\Bfrak(V)$ is a quotient of the tensor algebra $T(V)$ by the uniquely determined maximal proper biideal $\cI_\mx\subset \oplus_{m=2}^\infty V^{\ot m}$. If instead one considers any
$H$-stable biideal $\cI$ with $\{0\}\subseteq \cI \subseteq \cI_\mx$ then $T(V)/\cI$ is a pre-Nichols algebra as introduced by Angiono in \cite{a-Ang16}. Prominent examples of pre-Nichols algebras which are not Nichols algebras are the positive parts of quantized enveloping (super) algebras at roots of unity. 
%%%%%%%%%%%%%%%%%%%%%%%%%%%%%%%%%%%%
\subsection{Quantum symmetric pairs}\label{sec:QSP-intro}
%%%%%%%%%%%%%%%%%%%%%%%%%%%%%%%%%%%%
 Let $\gfrak$ be a semisimple complex Lie algebra and let $\theta:\gfrak\rightarrow \gfrak$ be an involutive Lie algebra automorphism with pointwise fixed Lie subalgebra $\kfrak=\{x\in \gfrak\,|\,\theta(x)=x\}$. The theory of quantum symmetric pairs provides quantum group analogs $B_\bc=U'_q(\kfrak)\subset \uqg$ of the universal enveloping algebra $U(\kfrak)$. Crucially, $B_\bc\subset\uqg$ is not a Hopf subalgebra but satisfies the weaker coideal property
\begin{align*}
  \kow(B_\bc)\subset B_\bc\ot \uqg
\end{align*}
for the coproduct $\kow$ of $\uqg$. Quantum symmetric pairs for classical $\gfrak$ were originally introduced by M.~Noumi, M.~Dijkhuizen and T.~Sugitani, case by case, to perform harmonic analysis on quantum group analogs of symmetric spaces, see \cite{a-Noumi96}, \cite{a-Dijk96}, \cite{a-NS95}.
Independently, G.~Letzter developed a comprehensive theory of quantum symmetric pairs based on the classification of involutive automorphisms of $\gfrak$ in terms of Satake diagrams \cite{a-Letzter99a}, \cite{MSRI-Letzter}. A Satake diagram $(X,\tau)$ consists of a subset $X$ of the nodes of the Dynkin diagram for $\gfrak$ and a diagram automorphism $\tau$ satisfying certain compatibility conditions, see \cite{a-Araki62}. Letzter's construction was extended to the Kac-Moody case in \cite{a-Kolb14}.

Much is known about the structure of the algebras $B_\bc$. Generators and relations for $B_\bc$ were determined in \cite[Section 7]{a-Letzter03}, see also \cite[Section 7]{a-Kolb14}. Let $\pfrak_X$ be the standard parabolic subalgebra corresponding the $X$. The algebra $B_\bc$ has a natural filtration such that the associated graded algebra is isomorphic to a subalgebra $U_q'(\pfrak_X)$ of the quantized enveloping algebra $U_q(\pfrak_X)$. This suggests that it is possible to interpret $B_\bc$ as a deformation of $U_q'(\pfrak_X)$.

\medskip

\noindent{\bf Problem I.} Explicitly define an associative product $\star$ on 
$U_q'(\pfrak_X)$ such that the algebra $(U_q'(\pfrak_X),\star)$ is canonically isomorphic to $B_\bc$.

\medskip

In the quasi-split case $X=\emptyset$, the algebras $B_\bc$ were already introduced in \cite{a-Letzter97}. In this case the involution $\theta$ can be given in terms of the Chevalley generators $\{e_i, f_i, h_i\,|\,i\in I\}$ of $\gfrak$ by
\begin{align*}
  \theta(e_i)=-f_{\tau(i)}, \qquad \theta(f_i)=-e_{\tau(i)}, \qquad \theta(h_i)=-h_{\tau(i)}.
\end{align*}
Let $E_i$, $F_i$, $K_i^{\pm 1}$ for $i\in I$ denote the standard generators of $\uqg$. Then the quantum symmetric pair coideal subalgebra $B_\bc$ corresponding to $(\emptyset, \tau)$ is generated by the elements
\begin{align}\label{eq:Bcgens}
  B_i=F_i + c_i E_{\tau(i)}K_i^{-1}, \qquad K_iK_{\tau(i)}^{-1} \qquad \mbox{for all $i\in I$}
\end{align}
where $\bc=(c_i)_{i\in I}\in \field^I$ are fixed parameters. The parameters $c_i$ need to satisfy certain compatibility conditions which assure that $\gr(B_\bc)$ is canonically isomorphic to the subalgebra $U'_q(\bfrak)=\field\langle F_i, K_i K_{\tau(i)}^{-1}\,|\,i\in I \rangle$ of $\uqg$. This conditions is equivalent to the fact that the pair $(\uqg,B_\bc)$ satisfies a quantum Iwasawa decomposition. In the quasi-split case this means that the multiplication map
\begin{align}\label{eq:q-Iwasawa-Uqg}
  U^+\ot U_0'\ot B_\bc \rightarrow \uqg
\end{align}  
is a linear isomorphism. Here $U_0'$ is the subalgebra of $\uqg$ generated by $\{K_i^{\pm 1}\,|\,i\in I_\tau\}$ where $I_\tau\subseteq I$ is a set of representatives of the $\tau$-orbits in $I$, and $U^+\subset \uqg$ is the subalgebra generated by $\{E_i\,|\,i\in I\}$. The central role of the quantum Iwasawa decomposition was first highlighted in \cite{a-Letzter97}. More general versions appeared in \cite{a-Letzter99a}, \cite{a-Letzter04}, \cite{a-Kolb14}.
In the general case \cite{MSRI-Letzter}, \cite{a-Kolb14}, the generators $B_i$ may come with a second parameter $s_i$. Here we suppress this parameter for simplicity, but we note that the theory can be extended by twisting by a character, see for example \cite[Section 3.5]{a-DK18p}.

The theory of quantum symmetric pairs received a big push in 2013 when the preprint versions of \cite{a-BaoWang18a} and \cite{a-ES18} introduced the notion of a bar involution for quantum symmetric pairs. H.~Bao and W.~Wang showed that much of G.~Lusztig's theory of canonical bases allows analogs for quantum symmetric pairs \cite{a-BaoWang18a}, \cite{a-BaoWang18b}. Of pivotal importance in Lusztig's theory is the quasi $R$-matrix $\Theta$ which lives in a completion of $U^-\ot U^+$ and intertwines two bar involution on $\kow(\uqg)$, see \cite[Theorem 4.1.2]{b-Lusztig94}. For the symmetric pair of type AIII with $X=\emptyset$, Bao and Wang showed in particular that there exists an intertwiner $\Theta^\theta$ in a completion of $B_\bc\ot U^+$ which plays a similar role as the quasi $R$-matrix $\Theta$. The existence of the intertwiner $\Theta^\theta$ was established in full generality in \cite{a-Kolb17p}. Following the program outlined in \cite{a-BaoWang18a}, the intertwiner $\Theta^\theta$ was used in \cite{a-BalaKolb15p}, \cite{a-Kolb17p} to construct a universal $K$-matrix for quantum symmetric pairs. The universal $K$-matrix is an analog of the universal $R$-matrix for $\uqg$. For this reason we call the intertwiner $\Theta^\theta$ the \textit{quasi $K$-matrix} for $B_\bc$.

The construction of the quasi $K$-matrix in \cite{a-BaoWang18a}, \cite{a-BalaKolb15p} is recursive and based on the intertwiner property for the bar involutions on $\uqg$ and $B_\bc$. 
This differs from the situation with (quasi) $R$-matrices.
Drinfeld constructed universal $R$-matrices for the doubles of all Hopf algebras as sums of dual bases \cite{inp-Drinfeld1}.
In this direction, the quasi $R$-matrix $\Theta$ has a second description in terms of dual bases of $U^-$ and $U^+$ with respect to a non-degenerate pairing, see \cite[Theorem 4.1.2]{b-Lusztig94}. 
It is an open question to give a similar description of the quasi $K$-matrix $\Theta^\theta$.
\medskip

\noindent{\bf Problem II.} Give a conceptual, non-recursive description of the quasi $K$-matrix $\Theta^\theta$ for quantum symmetric pairs in terms of dual bases of $U^-$ and $U^+$. This description should be parallel to the Drinfeld--Lusztig construction of the quasi $R$-matrices $\Theta$ as sums of dual bases, and should not involve the bar-involutions which are not applicable in closely related situations, such as roots of unity.

\medskip
For large classes of examples there exist explicit formulas for the quasi $K$-matrix, see \cite{a-DK18p}. However, these formulas do not come from dual bases on $U^-$ and $U^+$.
%%%%%%%%%%%%%%%%%%%%%%%%%%%%%%%%%%%%
\subsection{Goal of this paper}
%%%%%%%%%%%%%%%%%%%%%%%%%%%%%%%%%%%%
In the present paper we propose a construction of symmetric pairs for pre-Nichols algebras which extends Letzter's construction of quantum symmetric pairs. To keep things manageable, we restrict to pre-Nichols algebras of diagonal type. For quantum symmetric pairs this means that we restrict to the case $X=\emptyset$. The theory developed in the present paper includes examples of symmetric pairs for quantized enveloping algebras at roots of unity, quantum Lie superalgebras, and the more exotic examples which arose from Heckenberger's classification \cite{a-Heck09} of Nichols algebras of diagonal type. We do not place any restrictions on the Gelfand-Kirillov dimension of the
pre-Nichols algebra.
The case of general $X$ will involve Nichols algebras for Yetter-Drinfeld modules over more general Hopf algebras. We intend to address this more general case in the future.
One of the upshots of this is an intrinsic construction of quantum symmetric pairs in terms of a base Hopf algebra $H$, an involutive Hopf algebra 
automorphism of $H$, and an isomorphism between two Yetter-Drinfeld modules for $H$.

For pre-Nichols algebras of diagonal type we develop a general theory in full analogy to Letzter's theory \cite{a-Letzter99a}, \cite{MSRI-Letzter}, \cite{a-Kolb14}. For a symmetric 
bicharacter $\chi:\Z^n\times \Z^n\rightarrow \field$, we consider a Hopf algebra $U(\chi)$ with triangular decomposition $U(\chi)\cong U^+\rtimes H\ltimes U^-$ where $H=\field[\Z^n]$ is the group algebra of $\Z^n$ and $U^+, U^-$ are pre-Nichols algebra associated to $\chi$. We call $U(\chi)$ the Drinfeld double of $U^+$, see Remark \ref{rem:DD}. We define a coideal subalgebra $B_\bc\subset U(\chi)$ which depends on parameters $\bc=(c_i)\in \field^n$ and is generated by elements analogous to those given in \eqref{eq:Bcgens}. The coideal subalgebra
$B_\bc$ has a natural filtration and we determine the set of parameters $\bc$ for which $\gr(B_\bc)$ is isomorphic to a partial bosonization $H_\theta\ltimes U^-$.

In this setting we answer Problems I and II from Section \ref{sec:QSP-intro}. Lusztig's quasi $R$-matrix $\Theta$ also exists in the general setting of the present paper. To answer Problem I, we define two associative products on $H_\theta\ltimes U^-$. First, by a twisting construction,
we define a product $\star$ by a closed formula which only involves the quasi $R$-matrix $\Theta$ 
and an explicitly given algebra homomorphism $\osigma:U^-\hookrightarrow U^+\rtimes H$. Secondly, we use a linear isomorphism 
\begin{align*}
  \psi: B_\bc\rightarrow H_\theta\ltimes U^-,
\end{align*}
coming from a triangular decomposition of $U(\chi)$,
to push forward the algebra structure on $B_\bc$. We develop a general theory of star products on $\N$-graded algebras generated in degree $0$ and $1$ to show that the two algebra structures on $H_\theta\ltimes U^-$ coincide. Hence the map $\psi:B_\bc\rightarrow (H_\theta\ltimes U^-,\star)$ is an algebra isomorphism.

To resolve Problem II we need to restrict to the case where $U^+$, $U^-$ are Nichols algebras. We show that the element
\begin{align}\label{eq:Thetatheta}
  \Theta^\theta=(\psi^{-1}\ot \id)(\Theta)
\end{align}
which lives in a completion of $B_\bc\ot U^+$, has all the desired properties of a quasi $K$-matrix, and indeed coincides with the quasi $K$-matrix in the case of quantum symmetric pairs. We then use $\Theta^\theta$ to essentially construct a universal $K$-matrix for $B_\bc$ in the setting of Nichols algebras of diagonal type. We do not discuss the representation theory of $U(\chi)$, but follow an approach proposed by N.~Reshetikhin and T.~Tanisaki for universal $R$-matrices in \cite{a-Tanisaki92}, \cite{a-Reshetikhin95}. We obtain a weak notion of a universal $K$-matrix, which consists of an automorphism of a completion of $B_\bc\ot U(\chi)$ which satisfies the properties of conjugation by a universal $K$-matrix. In the following we discuss the results of the present paper in more detail. All through this paper the symbol $\N$ denotes the natural numbers including $0$, that is $\N=\{0,1,2,\dots\}$.
%%%%%%%%%%%%%%%%%%%%%%%%%%%%%%%%%%%%
\subsection{Symmetric pairs for pre-Nichols algebras}
%%%%%%%%%%%%%%%%%%%%%%%%%%%%%%%%%%%%
For the construction of the Hopf algebra $U(\chi)$ we mostly follow \cite{a-Heck10},  which extended Lusztig's braid group action to Nichols algebras of diagonal type, but we allow pre-Nichols algebras as introduced in \cite{a-Ang16}. Associated to the bicharacter $\chi$ is a Yetter-Drinfeld module $V^+(\chi)$ with basis $\{E_1,\dots, E_n\}$. We consider the corresponding pre-Nichols algebra $U^+=T(V^+(\chi))/\cI$ where $\cI$ is a $\Z^n$-graded biideal of the tensor algebra $T(V^+(\chi))$. We then form the bosonization $U^+\rtimes H$ and consider a quotient $U(\chi)$ of the quantum double of $U^+\rtimes H$ obtained by identifying the two copies of $H$. The Hopf algebra $U(\chi)$ is a natural generalisation of $\uqg$. In particular, it is generated by elements $E_i, F_i, K_i^{\pm 1}$ for $i\in I=\{1,\dots, n\}$, has a triangular decomposition $U(\chi)=U^+\rtimes H \ltimes U^-$, and satisfies relations similar to those for $\uqg$, see Section \ref{sec:setting}. Let $\{\alpha_i\,|\,i\in I\}$ be the standard basis of $\Z^n$, and let $\tau:I\rightarrow I$ be an involutive bijection such that $\chi(\alpha_{\tau(i)},\alpha_{\tau(j)})=\chi(\alpha_i,\alpha_j)$ for all $i,j\in I$. We define $B_\bc$ to be the subalgebra of $U(\chi)$ generated by the elements given in \eqref{eq:Bcgens} where $\bc=(c_1,\dots,c_n)\in \field^n$ are fixed parameters. Moreover, we let $H_\theta$ denote the subalgebra of $H$ generated by the elements $K_i K_{\tau(i)}^{-1}$ for all $i\in I$. The algebra $B_i$ has a natural filtration given by the degree function $\deg(B_i)=1$, $\deg(K_i K_{\tau(i)}^{-1})=0$. There is always a surjective algebra homomorphism
\begin{align}\label{eq:varphi-intro}
  \varphi:\gr(B_\bc) \rightarrow H_\theta \ltimes U^-.
\end{align}
We use linear projection maps $\pi_{0,0}:U(\chi)\rightarrow H$ and $P_{\mu}:U(\chi) \rightarrow U(\chi)$ for $\mu\in \Z^n$, which were first defined in \cite{MSRI-Letzter}, to show the following result.

\medskip

\noindent{\bf Theorem A.} (Theorem \ref{thm:c-cond}) \textit{The map $\varphi$ is an algebra isomorphism if and only if the following condition holds:
\begin{enumerate} 
  \item[({\bf c})] The ideal $\cI\subset T(V^+(\chi))$ is generated by homogeneous, noncommutative polynomials $p_j(E_1, \dots, E_n)$ for $j=1,\dots, k$ of degree $\lambda_j\in \N^n$, respectively, for which $\pi_{0,0}\circ P_{-\lambda_j}(p_j(B_1, \dots, B_n))=0$.
\end{enumerate}  
}
\medskip

As in the quantum case, the map $\varphi$ is an isomorphism if and only if the pair $(U(\chi), B_\bc)$ admits an Iwasawa decomposition analogous to \eqref{eq:q-Iwasawa-Uqg}, see Remark \ref{rem:Iwasawa}. 

Let $\Upoly$ be the subalgebra of $U(\chi)$ generated by the elements $E_iK_i^{-1}$, $F_i$, $K_i^{-1}$, $K_i K_{\tau(i)}^{-1}$ for all $i\in I$. The algebra $\Upoly$ contains $B_\bc$ and has a natural surjection $\kappa:\Upoly\rightarrow \Heis(\chi)$ onto a Heisenberg double $\Heis(\chi)$ associated to the bicharacter $\chi$. By construction, the kernel of $\kappa$ is the ideal generated by $K_i^{-1}$ for all $i\in I$. We can consider the image $\overline{B}_\bc=\kappa(B_\bc)$ inside $\Heis(\chi)$. Again we have a natural filtration given by a degree function and a surjection
\begin{align*}
  \overline{\varphi}: \gr ( \overline{B}_\bc ) \rightarrow H_\theta\ltimes U^-.
\end{align*}
It turns out that map $\overline{\varphi}$ is an algebra isomorphism irrespective of the choice of parameters $\bc$. Let $\overline{G}^+$ be the subalgebra of $\Heis(\chi)$ generated by the elements $\kappa(E_iK_i^{-1})$ for all $i\in I$.
\medskip

\noindent{\bf Theorem B.} (Theorem \ref{thm:Heis-no-cond}) \textit{The map $\overline{\varphi}$ is an isomorphism, that is, 
the pair $(\Heis(\chi), B_\bc)$ always admits an Iwasawa decomposition $\Heis(\chi)\cong \overline{G}^+\ot \oB_\bc$.}

\medskip

The algebra $\Upoly$ has an $\N$-filtration given by the degree function defined by
\begin{align*}
  \deg(E_iK_i^{-1})=\deg(F_i)=\deg(K_i^{-1})=1, \qquad \deg(K_iK_{\tau(i)}^{-1})=0 
\end{align*}
for all $i\in I$. We call the associated graded algebra $\Heis(\chi)^\vee=\gr(\Upoly)$ the \textit{negative Heisenberg double} associated to $U^+$. We observe that condition ({\bf c}) in Theorem A can be 
verified in the negative Heisenberg double. Indeed, the projection map $\pi_{0,0}$ has an analog $\pi_{0,0}^\vee:\Heis(\chi)^\vee\rightarrow H$. For all $i\in I$ set $B_i^\vee=F_i+c_i E_{\tau(i)}K_i^{-1}\in \Heis(\chi)^\vee$. 

\medskip

\noindent{\bf Theorem C.} (Theorem \ref{thm:Hmin}) \textit{For any homogeneous, noncommutative polynomial $p(x_1,\dots, x_n)$ of degree $\lambda\in \N^n$ we have}
\begin{align*}
  \pi_{0,0}\circ P_{\lambda}(p(B_1, \dots, B_n))=\pi_{0,0}^\vee(p(B_1^\vee,\dots,B_n^\vee)).
\end{align*}

The point of Theorem C is that calculations in $\Heis(\chi)^\vee$ are easier than calculations in $U(\chi)$. We can summarize the situation in the following diagram:
%%%%%%%%%%%%%%%%%%%%%%%%%%%%%%%%%%%%%%%%%%%%
\[
\begin{tikzpicture}[scale=1.4]
\node (A) at (0,0) {$\oB_\bc$}; 
\node (B) at (2,0) {$\Heis(\chi)$};
\node (C) at (4,0) {$\Hmin$};
\node (D) at (1,1) {$B_\bc$};
\node (E) at (3,1) {$\Upoly$};
\node (E') at (5,1) {$U(\chi)$};
\node (F) at (2,2) {$\Btil_\bc$};
\node (G) at (4,2) {$\Util(\chi)^\poly$}; 
\node (G') at (6,2) {$\Util(\chi)$};
%\path[->>,font=\scriptsize,>=angle 90]
\draw
(A) edge[right hook ->,font=\scriptsize,>=angle 90] (B)
(D) edge[right hook ->,font=\scriptsize,>=angle 90] (E)
(E) edge[right hook ->,font=\scriptsize,>=angle 90] (E')
(F) edge[right hook ->,font=\scriptsize,>=angle 90] (G)
(G) edge[right hook ->,font=\scriptsize,>=angle 90] (G')
%(B) edge[->>,font=\scriptsize,>=angle 90] node[above]{$\wt{\psi}$} (C)
%(D) edge[right hook ->,font=\scriptsize,>=angle 90] (E)
(E) edge[->>,font=\scriptsize,>=angle 90] node[above]{$\kappa$} (B)
(G) edge[->>,font=\scriptsize,>=angle 90] (E)
(G') edge[->>,font=\scriptsize,>=angle 90] node[above]{$\eta$} (E')
(F) edge[->>,font=\scriptsize,>=angle 90] (D)
(D) edge[->>,font=\scriptsize,>=angle 90] (A)
(E) edge[->,font=\scriptsize,>=angle 90] node[above]{$\gr$} (C);
%(A) edge[->>,font=\scriptsize,>=angle 90] node[right]{$\eta$}(D)
%(B) edge[->>,font=\scriptsize,>=angle 90] node[right]{$\eta$} (E)
%(C) edge[->>,font=\scriptsize,>=angle 90] node[right]{$\eta$} (F);
\end{tikzpicture}
\]
%%%%%%%%%%%%%%%%%%%%%%%%%%%%%%%%%%%%%%%%%%
Here the tildes $\sim$ denote the versions of $U(\chi)$, $B_\bc$, $\Upoly$ in the case where the biideal $\cI$ is trivial, that is $\cI=\{0\}$. In this case $\Util(\chi)=T(V^+(\chi))$ is just the tensor algebra. The map $\eta$ denotes the canonical projections.

In Section \ref{sec:Examples} we apply Theorems A and C to various classes of examples. For each example class we determine the parameters $\bc\in \field^n$ for which the maps $\varphi$ in \eqref{eq:varphi-intro} is an algebra isomorphism. In each case the calculation simplifies significantly because Theorem C allows us to calculate in the negative Heisenberg double. We first consider quantized enveloping algebras in Section \ref{sec:large} extending known results from \cite{MSRI-Letzter}, \cite{a-Kolb14} to the root of unity case. In Section \ref{sec:small} we consider the small quantum groups $\mathfrak{u}_\zeta(\slfrak_3)$ where $\zeta$ is an arbitrary root of unity. The calculations for this example naturally lead us to consider the Al-Salam-Carlitz I discrete orthogonal polynomials $U_n^{(a)}(x;q)$ originally defined in \cite{a-AlCa65}, see also \cite{b-KoLeSw00}. As further examples we consider quantized enveloping algebras of Lie superalgebras of type $\slfrak(m|k)$ and the distinguished pre-Nichols algebra of type $\mathfrak{ufo}(8)$ in Sections \ref{sec:super} and \ref{sec:ufo}, respectively.
%%%%%%%%%%%%%%%%%%%%%%%%%%%%%%%%%%%%%%%%%%
\subsection{Star products on partial bosonizations}
%%%%%%%%%%%%%%%%%%%%%%%%%%%%%%%%%%%%%%%%%%
In Section \ref{sec:star} we introduce star products and apply them to solve Problem I from Section \ref{sec:QSP-intro}. We define a star product on an $\N$-graded $\field$-algebra $A=\bigoplus_{j\in \N} A_j$
to be an associative bilinear operation 
\[
* : A \times A \to A, \quad (a,b)\mapsto a\ast b
\]
such that
\[
  a * b - ab \in A_{< m+ n} \qquad \mbox{for all $a \in A_m$, $b \in A_n$.}
\]
A star product  will be called $0$-equivariant
if
\[
a * h = ah \quad \mbox{and} \quad h * a= h a \quad \mbox{for all} \quad h \in A_0, a \in A. 
\]
Star products provide us with an efficient way to prove that two filtered deformations of $A$ are isomorphic.
Namely, if $A$ is generated in degrees $0$ and $1$, and  $A_1 = \sum_i A_0 F_i A_0$ for a subset 
$\{ F_i \} \subset A_1$, then every 0-equivariant star product on $A$ is uniquely determined by the collection of 
$\field$-linear maps 
\begin{align} \label{eq:star-linear-maps}
u \in A \mapsto F_i * u -  F_i u \in A,
\end{align}
see Lemma \ref{lem:sprod-isom}. The above conditions are satisfied for the algebra $A=H_\theta\ltimes U^-$ which is graded with $A_0=H_\theta$ and $A_1=H_\theta \, \mathrm{span}_\field\{F_i\,|\,i\in I\}$.

We have the decomposition $\Upoly \cong (H_\theta\ltimes U^-) \otimes \field \langle K_i^{-1}, E_i K_i^{-1} \mid i \in I \rangle$.
Consider the $\field$-linear map $\psi : \Upoly \twoheadrightarrow H_\theta\ltimes U^-$ 
which is the identity map on $H_\theta\ltimes U^-$ and is the algebra homomorphism given by $K_i^{-1} \mapsto 0$, $E_i K_i^{-1} \mapsto 0$ on the second factor. By restriction to $B_\bc$ we obtain the following commutative diagram:
\[
\begin{tikzpicture}[scale=1.4]
\node (A) at (0,1) {$B_\bc$}; 
\node (B) at (3.5,1) {$(H_\theta\ltimes U^-) \otimes \field \langle K_i^{-1}, E_i K_i^{-1} \mid i \in I \rangle$};
\node (C) at (7,1) {$\Upoly$};
\node (D) at (3.5,0) {$H_\theta\ltimes U^-$};
\draw
(A) edge[->,font=\scriptsize,>=angle 90] (D)
(B) edge[->,font=\scriptsize,>=angle 90] node[above]{$\cong$} (C)
(B) edge[->>,font=\scriptsize,>=angle 90] node[right]{$\psi$} (D)
(A) edge[right hook ->,font=\scriptsize,>=angle 90] (B);
%(D) edge[right hook ->,font=\scriptsize,>=angle 90] (E)
%(E) edge[right hook ->,font=\scriptsize,>=angle 90] (E')
%(F) edge[right hook ->,font=\scriptsize,>=angle 90] (G)
%(G) edge[right hook ->,font=\scriptsize,>=angle 90] (G')
%(B) edge[->>,font=\scriptsize,>=angle 90] node[above]{$\wt{\psi}$} (C)
%(D) edge[right hook ->,font=\scriptsize,>=angle 90] (E)
%(E) edge[->>,font=\scriptsize,>=angle 90] node[above]{$\kappa$} (B)
%(G) edge[->>,font=\scriptsize,>=angle 90] (E)
%(G') edge[->>,font=\scriptsize,>=angle 90] node[above]{$\eta$} (E')
%(F) edge[->>,font=\scriptsize,>=angle 90] (D)
%(D) edge[->>,font=\scriptsize,>=angle 90] (A)
%(E) edge[->,font=\scriptsize,>=angle 90] node[above]{$\gr$} (C);
%(A) edge[->>,font=\scriptsize,>=angle 90] node[right]{$\eta$}(D)
%(B) edge[->>,font=\scriptsize,>=angle 90] node[right]{$\eta$} (E)
%(C) edge[->>,font=\scriptsize,>=angle 90] node[right]{$\eta$} (F);
\end{tikzpicture}
\]
In the setting of quantized enveloping algebras the map $\psi$ recently appeared in \cite[Corollary 4.4]{a-Letzter17p}.
It turns out that the restriction of $\psi$ to $B_\bc$ is a linear isomorphism if and only if the map $\varphi$ given by \eqref{eq:varphi-intro} is an algebra isomorphism, see Remark \ref{rem:phi-psi}. We may hence use the map $\psi$ to push forward the algebra structure from $B_\bc$ to $H_\theta\ltimes U^-$.

\medskip

\noindent{\bf Theorem D.} (Theorem \ref{thm:first-star}) \textit{If the map $\varphi$ is an algebra isomorphism {\em{(}}i.e.~if 
$(U(\chi), B_\bc)$ admits an Iwasawa decomposition{\em{)}}, then the the restriction $\psi : B_\bc \to H_\theta\ltimes U^-$ 
is an algebra isomorphism to the uniquely determined 0-equivariant star product on $H_\theta\ltimes U^-$ such that
\[
F_i * u  = F_i u + c_i q_{i\tau(i)}(K_{\tau(i)} K_i^{-1}) \partial_{\tau(i)}^L(u) \quad \mbox{for all} \quad i \in I, u\in U^-
\]
where $\partial_i^L$ are the frequently used skew derivations of $U^-$ given by \eqref{eq:partialFj}--\eqref{eq:skew-properties}.
}

\medskip

In addition to determining the algebraic structure of $B_\bc$, Theorem D also gives an effective way for the explicit description of the 
relations among the generators of $B_\bc$. In Proposition \ref{prop:def-rels} we prove that 
the relations among the generators $F_i$ and $K_i K_{\tau(i)}^{-1}$ 
of the star product algebra $(H_\theta\ltimes U^-, *)$ are the relations with respect to the usual product on $H_\theta\ltimes U^-$ 
but re-expressed in terms of the star product, see Section \ref{sec:RelsRev} for details and examples.

In Section \ref{sec:twist-def-ass} we define a second associative binary operation $\star$ on $H_\theta\ltimes U^-$. 
Denote by $U^\pm_\mx$ the Nichols algebras that are factors of $U^\pm$ and by $U(\chi)_\mx$ the corresponding Drinfeld double. By \cite[Theorem 5.8]{a-Heck10} there exists a pairing of Hopf algebras
\begin{align}\label{eq:pair-intro}
  \langle\, ,\, \rangle_\mx: (H \ltimes U^-_\mx)\ot (U^+_\mx\rtimes H) \rightarrow \field
\end{align}
which is nondegenerate when restricted to $U^-_\mx\ot U^+_\mx$. The pairing induces a left action $\lact$ and a right action $\ract$ of $U^+_\mx\rtimes H$ on $H\ltimes U^-$, see Section \ref{sec:quasiR}. The pairing \eqref{eq:pair-intro} allows us to define quasi $R$-matrix for $U(\chi)_\mx$ as a sum of tensor products of dual bases of $U^-_\mx$ and $U^+_\mx$. We write formally 
\begin{equation}
\label{eq:quasiRintro}
\Theta=\sum_\mu (-1)^{|\mu|} F_\mu\ot E_\mu.
\end{equation}
In Sections \ref{sec:quasiR} and \ref{sec:skew-derivation} we show that this quasi $R$-matrix retains essential properties of the quasi $R$-matrix for quantum groups in \cite{b-Lusztig94}.
There exists an algebra homomorphism $\osigma : U^-_\mx \hookrightarrow U^+_\mx \rtimes H$ such that $\osigma(F_i)=c_{\tau(i)}K_i E_{\tau(i)}$ for all $i\in I$, see Section \ref{sec:osigma1}. The associative binary operation $\star$ on $H_\theta\ltimes U^-$ is defined solely in terms of the quasi $R$-matrix $\Theta$ and the algebra homomorphism $\osigma$ and exists irrespective of the choice of parameters $\bc$. Let $S$ denote the antipode of $U(\chi)_\mx$.

\medskip

\noindent{\bf Theorem E.} (Theorem \ref{thm:star-twist}, Proposition \ref{prop:second-star} and Corollary \ref{cor:star=ast}) \textit{For any $\bc\in \field^n$ the operation
\begin{align}\label{eq:fstarg-intro}
  f \star g = \sum_{\rho} (-1)^{|\rho|}(\osigma(F_\rho)\lact f) K_\rho [g \ract (S^{-1}(E_\rho)K_\rho)] 
  \quad \mbox{for all $f, g \in U^-$}
\end{align}
defines a $0$-equivariant star product $\star$ on $H_\theta \ltimes U^-$. The star product $\star$ coincides with the star product $\ast$ from Theorem D
when the latter is defined.}

\medskip

Theorem E provides the desired explicit formula for the star product on $H_\theta\ltimes U^-$ and hence solves Problem I. The main step in the proof of the first part of Theorem E is to show that the bilinear operation $\star$ defined by \eqref{eq:fstarg-intro} is associative. The second statement then follows by comparison of the linear maps \eqref{eq:star-linear-maps} for the two star products $\star$ and $\ast$.

In the situation of Theorem D, the algebra isomorphism $\psi$ turns the algebra $(H_\theta\ltimes U^-,\star)$ into a $U(\chi)_\mx$-comodule algebra. In Section \ref{sec:comod} we give an explicit formula for the corresponding coaction $\kow_\star$. This formula again only involves the quasi $R$-matrix $\Theta$ and the homomorphism $\osigma$. The $U(\chi)_\mx$-comodule algebra structure on $(H_\theta\ltimes U^-,\star)$ again exists irrespective of the choice of parameters $\bc\in \field^n$.  

%%%%%%%%%%%%%%%%%%%%%%%%%%%%%%%%%%%%%%%%%
\subsection{Quasi $K$-matrices versus quasi $R$-matrices}
%%%%%%%%%%%%%%%%%%%%%%%%%%%%%%%%%%%%%%%%%
In Section \ref{sec:quasiK} we address Problem II from Section \ref{sec:QSP-intro}. We need to restrict to the case that $U^\pm=U^\pm_\mx$ are Nichols algebras and we assume that the conditions of Theorem D are satisfied.
Under these assumptions the map $\psi$ is an isomorphism and we may define an element $\Theta^\theta$ in a completion of $B_\bc\ot U^+$ by \eqref{eq:Thetatheta}. In Proposition \ref{prop:kowidThetatheta} we give explicit formulas for $(\kow\ot \id)(\Theta^\theta)$ and $(\id \ot \kow)(\Theta^\theta)$ which are analogs of the formulas for  $(\kow\ot \id)(\Theta)$ and $(\id \ot \kow)(\Theta)$ in \cite[4.2]{b-Lusztig94}. We then show in Proposition \ref{prop:quasiK-intertwiner} that $\Theta^\theta$ satisfies an intertwiner property which reproduces the intertwiner property for bar involutions of $B_\bc$ and $\uqg$ from \cite[Proposition 3.2]{a-BaoWang18a} in the case of quantized enveloping algebras. For this reason we call $\Theta^\theta$ the quasi $K$-matrix for the pair $(U(\chi)_\mx,B_\bc)$.

The following diagram illustrates our double construction for quasi $K$-matrices versus the Drinfeld--Lusztig construction for quasi $R$-matrices:
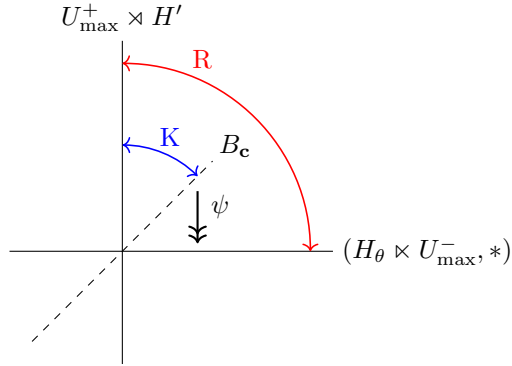
\[
\begin{tikzpicture}
\draw (-1.5,0) -- (2.8,0) node[right] {$(H_\theta \ltimes U^-_\mx, *)$};
\draw [dashed] (-1.2,-1.2) -- (1.2,1.2);
\node at (1.5,1.4) {$B_\bc$};
\node at (1.3,0.6) {$\psi$};
\draw (0,-1.5) -- (0,2.8) node[above] {$U^+_\mx \rtimes H'$};
\draw [->>, thick] (1,0.8) -- (1cm,3pt);
\draw[<->, red, semithick] (2.5,0) arc (0:90:2.5);
\node at (1.064,2.56) {\textcolor{red}{R}};
\draw[<->, blue, semithick] (1,1) arc (45:90:1.41);
\node at (0.635,1.517) {\textcolor{blue}{K}};
\end{tikzpicture}
\]
The two axes represent the decomposition $U(\chi)_\mx \cong  (U^+_\mx \rtimes H') \otimes (H_\theta \ltimes U^-_\mx)$ for a Hopf subalgebra 
$H'$ of $H$, and the corresponding quasi $R$-matrix is a sum of dual bases of $U^-_\mx$ and $U^+_\mx$. The diagonal 
represents the coideal subalgebra $B_\bc$ which is isomorphic via the projection $\psi$ to a star product on the horizontal axes,
and the corresponding quasi $K$-matrix is the pull back under $\psi^{-1} \otimes \id$ of the quasi $R$-matrix.

In Section \ref{sec:w-quasi-Hopf} we review the theory of weakly quasitriangular Hopf algebras from \cite{a-Tanisaki92}, \cite{a-Reshetikhin95}, see also \cite{a-Gav97}. This theory is extended to comodule algebras in Section \ref{sec:wq-coid}. The notion of a weakly quasitriangular comodule algebra captures the existence of a universal $K$-matrix. Using the coproduct identities and the intertwiner property for $\Theta^\theta$ we show the following result. 

\medskip

\noindent{\bf Theorem F.} (Theorem \ref{thm:wq-tri2})
\textit{Under the above assumptions the coideal subalgebra $B_\bc$ of $U(\chi)_\mx$ is weakly quasitriangular up to completion.}

\medskip

\noindent
{\bf Acknowledgements.} We would like to thank Nicol\'as Andruskiewitsch, Iv\'an Angiono and Istv\'an Heckenberger for valuable discussions and correspondence. Part of the work on this paper was carried out while S.K. visited Louisiana State University in March 2017. This visit was supported by a 
Scheme 4 grant from the London Mathematical Society. The research of M.Y. was supported by NSF grant DMS-1601862 and Bulgarian Science Fund grant H02/15.
%%%%%%%%%%%%%%%%%%%%%%%%%%%%%%%%%%%%%%%%%%%%%%%%%%%%%%%%%%%%%%%%%%%%%%%%%%%%%%%%%%%%
\section{The size of coideal subalgebras of Heisenberg doubles and Drinfeld doubles}\label{sec:size}
%%%%%%%%%%%%%%%%%%%%%%%%%%%%%%%%%%%%%%%%%%%%%%%%%%%%%%%%%%%%%%%%%%%%%%%%%%%%%%%%%%%%
In this first section we describe the general setting and introduce the coideal subalgebras $B_\bc$ which are the main objects of investigation in the present paper. The algebras $B_\bc$ have a natural filtration. We determine the parameters $\bc$ for which $\mathrm{gr}(B_\bc)$ is of the right size. To this end we use methods first employed for quantized universal enveloping algebras by G.~Letzter in \cite[Section 7]{MSRI-Letzter}.
%%%%%%%%%%%%%%%%%%%%%%%%%%
\subsection{The setting}\label{sec:setting}
%%%%%%%%%%%%%%%%%%%%%%%%%%
We review the Drinfeld double $\Util(\chi)$ of the tensor algebra of a braided vector space of diagonal type, following \cite[Section 4]{a-Heck10}. We will need in particular the description of ideals of $\Util(\chi)$ which preserve the triangular decomposition from \cite[Proposition 4.17]{a-Heck10}. This allows us to consider quotients of $\Util(\chi)$ which are generalizations of Drinfeld-Jimbo quantized enveloping algebras for deformation parameters including roots of unity.

Let $\field$ be a field and set $\field^\times=\field \setminus \{0\}$. Let $I=\{1,\dots,n\}$ and let $\{\alpha_i\,|\,i\in I\}$ denote the standard basis of $\Z^n$. Let $H=\field[K_i, K_i^{-1}\,|\,i\in I]$ denote the group algebra of $\Z^n$. Let $\chi:\Z^n\times \Z^n\rightarrow \field^\times$ be a bicharacter and set $q_{ij}=\chi(\alpha_i,\alpha_j)$ for all $i,j\in I$. In this paper we always assume that the matrix $(q_{ij})$ is symmetric, that is $q_{ij}=q_{ji}$ for all $i,j\in I$. Recall that every bicharacter is twist-equivalent to a symmetric bicharacter, and that the corresponding Nichols algebras are linearly isomorphic, see \cite[Proposition 3.9]{MSRI-AS02}. Let
\begin{align*}
  V^+(\chi)\in {}^{H}_{H} \cYD, \qquad V^-(\chi)\in {}^{H}_{H} \cYD
\end{align*}  
be the Yetter-Drinfeld modules with linear basis $\{E_i\,|\,i\in I\}$ and  $\{F_i\,|\,i\in I\}$, respectively, such that the left action $\cdot$ and the left coaction $\delta$ of $H$ on $V^+(\chi)$ and on $V^-(\chi)$ are given by
\begin{equation}\label{eq:YD+-}
\begin{aligned}
  K_i\cdot E_j&=q_{ij} E_j, & \delta(E_i)&=K_i\ot E_i,\\
  K_i\cdot F_j&=q_{ji}^{-1} F_j, & \delta(F_i)&=K_i^{-1}\ot F_i,
\end{aligned}
\end{equation}
respectively.
Let $T(V^+(\chi))$ and $T(V^-(\chi))$ denote the tensor algebras of $V^+(\chi)$ and $V^-(\chi)$, respectively. Recall that $T(V^+(\chi))$ and $T(V^-(\chi))$ are braided Hopf algebras in the category ${}^{H}_{H} \cYD$. Let $T(V^+(\chi))\rtimes H$ and $T(V^-(\chi))\rtimes H$ denote the bosonization of $T(V^+(\chi))$ and $T(V^-(\chi))$, respectively, which are Hopf algebras, see \cite{a-Radford85}, \cite{a-Majid94}, \cite[(4.5)]{a-Heck10}. We write $(T(V^-(\chi))\rtimes H)^\cop$ to denote the Hopf algebra structure on $T(V^-(\chi))\rtimes H$ with the opposite coproduct. There exists a skew Hopf-pairing between $T(V^+(\chi))\rtimes H$ and $(T(V^-(\chi))\rtimes H)^\cop$, see \cite[Proposition 4.3]{a-Heck10}. We consider the quotient of the corresponding Drinfeld double by the ideal identifying the two copies of $H$
\begin{align*}
  \Util(\chi)=\Big(\big(T(V^+(\chi))\rtimes H\big) \ot \big(T(V^-(\chi))\rtimes H \big)^{\mathrm{cop}}\Big)/(K_i L_i-1\,|\,i\in I)
\end{align*}
where $L_i$ denotes the inverse of $K_i$ in the second copy of $H$, see \cite[Definition 4.5, Remark 5.7]{a-Heck10}. More explicitly, $\Util(\chi)$ is a Hopf algebra generated by the elements $E_i, F_i, K_i, K_i^{-1}$ with coproducts
\begin{align}
  \kow(E_i)&=E_i\ot 1 + K_i\ot E_i,\nonumber\\
  \kow(F_i)&=F_i\ot K_i^{-1} + 1\ot F_i, \label{eq:kowEFK}\\
  \kow(K_i)&=K_i\ot K_i\nonumber
\end{align}
for all $i\in I$. Defining algebra relations for $\Util(\chi)$ are given by
\begin{align}
  K_i K_i^{-1}&=1,\nonumber\\
  K_i E_j = q_{ij} E_j K_i,& \qquad K_i F_j = q_{ij}^{-1} F_j K_i, \label{eq:def-rel}\\
  E_i F_j - F_j E_i&= \delta_{i,j}(K_i-K^{-1}_i), \nonumber
\end{align}
for all $i,j\in I$, see \cite[Proposition 4.6]{a-Heck10}.
In view of the defining relations \eqref{eq:def-rel} of $\Util(\chi)$ it follows that $\omega$ extends to an isomorphism of Hopf algebras
  $\omega:\Util(\chi)\rightarrow \Util(\chi)^{\mathrm{cop}}$ such that
  \begin{align*}
    \omega(K_i)=L_i, \quad \omega(L_i)=K_i, \quad \omega(E_i)=F_i, \quad \omega(F_i)=E_i,
  \end{align*}
  for all $i\in I$. The automorphism $\omega$ is denoted by $\phi_3$ in \cite[Proposition 4.9.(6)]{a-Heck10}.

The algebra $\Util(\chi)$ has a triangular decomposition in the sense that the multiplication map
\begin{align}\label{eq:BHB}
  T(V^+(\chi))\ot H \ot T(V^-(\chi))\rightarrow \Util(\chi)
\end{align}
is a linear isomorphism, see \cite[Proposition 4.14]{a-Heck10}. We write this as 
\[
\Util(\chi)\cong T(V^+(\chi))\rtimes H \ltimes T(V^-(\chi))
\]
to indicate that the bosonizations $T(V^+(\chi))\rtimes H$ and  $H\ltimes T(V^-(\chi))=(T(V^-(\chi)\rtimes H)^\cop$ are subalgebras of $U(\chi)$. We will use similar notation for other triangular decompositions later in the paper.
We are interested in ideals of $\Util(\chi)$ which are compatible with the triangular decomposition. Let
\begin{align*}
  \cI\subseteq \bigoplus_{m=2}^\infty T(V^+(\chi))_m
\end{align*}
be a $\Z^n$-graded biideal of $T(V^+(\chi))$. By \cite[Corollary 4.24]{a-Heck10} the subspace $\cI H T(V^-(\chi))$ is a Hopf ideal of $\Util(\chi)$. Similarly one shows that the subspace $T(V^+(\chi))H \omega(\cI)$ is a Hopf ideal of $\Util(\chi)$. Let $(\cI,\omega(\cI))$ denote the ideal of $\Util(\chi)$ generated by $\cI$ and $\omega(\cI)$. We define
\begin{align*}
  U(\chi)=\Util(\chi)/(\cI,\omega(\cI)), \qquad U^+=T(V^+(\chi))/\cI, \qquad U^-=T(V^-(\chi))/\omega(\cI).
\end{align*}
By \cite[Proposition 4.17]{a-Heck10} the Hopf algebra $U(\chi)$ has a triangular decomposition
\begin{align}\label{eq:U+H-}
  U(\chi)\cong U^+ \rtimes H \ltimes U^-.
\end{align}  
The subalgebras $U^+$ and $U^-$ are pre-Nichols algebras as defined in \cite{a-Ang16}. Recall from \cite{a-Ang16} that a pre-Nichols algebra of a braided vector space $V$ is any graded braided Hopf algebra of the form $T(V)/\cI$ where $\cI$ is a graded biideal. In particular, if we choose $\cI=\cI_{\mathrm{max}}(\chi)\subseteq  T(V^+(\chi))$ to be the maximal $\Z^n$-graded biideal in $\bigoplus_{m=2}^\infty T(V^+(\chi))_m$, then $U^+$ is the Nichols algebra of $V^+(\chi)$. We allow more general graded biideals $\cI$ to cover non-restricted 
specializations of quantized universal enveloping algebras at roots of unity \cite{a-DKP}.

%%%%%%%%%%%%%%%%%%%%%%%%%
\begin{rema}\label{rem:DD}
  The algebra $U(\chi)$ is a factor of the Drinfeld double of the bosonization of $U^+$. For the sake of brevity, we will refer to $U(\chi)$ as the {\em{Drinfeld double of the pre-Nichols algebra $U^+$}}. 
\end{rema}
%%%%%%%%%%%%%%%%%%%%%%%%%%
We end this introductory section by recalling two projection maps which play an important role in Letzter's theory of quantum symmetric pairs, see \cite[Section 4, Lemma 7.3]{MSRI-Letzter}.
 Let $G^-$ be the subalgebra of $U(\chi)$ generated by the elements $F_iK_i$ for all $i\in I$. We can rewrite the triangular decomposition \eqref{eq:U+H-} as
  \begin{align*}
     U(\chi) \cong U^+\rtimes H \ltimes G^-.
  \end{align*}  
  As a vector space $U(\chi)$ has a direct sum decomposition
  \begin{align}\label{eq:Uchi-UKG}
     U(\chi)=\bigoplus_{\lambda\in \Z^n} U^+ K_\lambda G^-.
  \end{align}
  Here we write $K_\lambda=K_1^{\lambda_1} \cdot \dots \cdot K_n^{\lambda_n}$ for any $\lambda=(\lambda_1, \dots, \lambda_n)\in \Z^n$.
  For $\lambda\in \Z^n$ let
  \begin{align}\label{eq:Plambda-def}
    P_\lambda:U(\chi)\rightarrow U^+ K_\lambda G^-
  \end{align}
  be the canonical projection with respect to the direct sum decomposition \eqref{eq:Uchi-UKG}. It follows from the definition of the coproduct \eqref{eq:kowEFK} that $P_\lambda$ is a homomorphism of left $U(\chi)$-comodules, that is
  \begin{align}\label{eq:kowPlambda}
     \kow(P_\lambda(x)) = (\id \ot P_\lambda)(\kow(x))
  \end{align}
  for all $x\in U(\chi)$. The algebras $U^+$ and $U^-$ are $\Z^n$-graded with $\deg(E_i)=\alpha_i$ and $\deg(F_i)=-\alpha_i$ for all $i\in I$. Degrees of homogeneous elements in $U^+$ and $U^-$ lie in $\N^n$ and $-\N^n$, respectively. %Define
%  \begin{align*}
%     \Z^n_+=\left\{\sum_{i\in I}n_i\alpha_i\,|\,n_i\ge 0 \mbox{ for all $i\in %I$}\right\}.
%  \end{align*}
 Hence we obtain a second direct sum decomposition
  \begin{align}\label{eq:Uchi-UHU}
     U(\chi)=\bigoplus_{\alpha,\beta\in \N^n} U^+_\alpha H U^-_{-\beta}.
  \end{align}
  For $\alpha, \beta \in \N^n$ let
  \begin{align}
     \pi_{\alpha,\beta}: U(\chi)\rightarrow  U^+_\alpha H U^-_{-\beta}
  \end{align}
  denote the canonical projection with respect to the direct sum decomposition \eqref{eq:Uchi-UHU}.
%%%%%%%%%%%%%%%%%%%%%%%%%%%%%%%%%
\subsection{The partial bosonization $H_\theta\ltimes U^-$ and the coideal subalgebra $B_\bc$}\label{sec:Bc}
%%%%%%%%%%%%%%%%%%%%%%%%%%%%%%%%%
Let $\tau:I\rightarrow I$ be a bijection such that $\tau^2=\id$ and $q_{ij}=q_{\tau(i)\tau(j)}$ for all $i,j\in I$. We may consider $\tau$ as an automorphism of the braided bialgebra $T(V^+(\chi))$. We always assume that the ideal $\cI$ used to define $U(\chi)$ satisfies the relation $\tau(\cI)=\cI$. We also consider $\tau$ as a group automorphism of $\Z^n$ given by $\tau(\alpha_i)=\alpha_{\tau(i)}$ for all $i\in I$. Let $\theta:\Z^n\rightarrow \Z^n$ be the involutive group automorphism given by
\begin{align*}
  \theta(\lambda)=-\tau(\lambda)\qquad \mbox{for all $\lambda\in \Z^n$}
\end{align*}
and set 
\begin{equation}
\label{Ztheta}
\Z^n_\theta=\{\lambda\in \Z^n\,|\,\theta(\lambda)=\lambda\}.
\end{equation}
Define $H_\theta$ to be the subalgebra of $H$ generated by the elements $K_iK_{\tau(i)}^{-1}$ for all $i\in I$. By construction, $H_\theta$ is the group algebra of $\Z^n_\theta$. We call the subalgebra $H_\theta \ltimes U^-$ of $U(\chi)$ generated by $H_\theta$ and $U^-$ the partial bosonization of $U^-$. As a vector space we have $H_\theta\ltimes U^-=H_\theta\ot U^-$.

For $\bc=(c_1,\dots, c_n)\in \field^n$ we define $B_\bc$ to be the subalgebra of $U(\chi)$ generated by $H_\theta$ and the elements
  \begin{align*}
     B_i=F_i + c_i E_{\tau(i)} K_i^{-1} \qquad \mbox{for all $i\in I$.}
  \end{align*}
  The definition of the coproduct $\kow$ on $U(\chi)$ implies that
  \begin{align}\label{eq:kowB}
     \kow(B_i)=B_i\ot K_i^{-1} + 1\ot F_i + c_i K_{\tau(i)}K_i^{-1}\ot E_{\tau(i)}K_i^{-1} \qquad \mbox{for all $i\in I$}
  \end{align}
and hence $B_\bc\subset U(\chi)$ is a right coideal subalgebra, that is
\begin{align*}
    \kow(B_\bc)\subset B_\bc \ot U(\chi).
\end{align*}
The algebra $B_\bc$ has a filtration $\cF$ defined by the degree function given by
\begin{equation}
  \label{filt-B}
  \begin{aligned}
  \deg(B_i)&=1 \qquad \mbox{for all $i\in I$,}\\
  \deg(h)&=0 \qquad \mbox{for all $h\in H_\theta$.}
  \end{aligned} 
\end{equation}  
In the following we want to compare the associated graded algebra $\gr(B_\bc)$ with the algebra $H_\theta\ltimes U^-$. To this end, we first introduce some more notation.
For any multi-index $J=(j_1, \dots,j_m)\in I^m$ we write $|J|=m$,
and we write $F_J=F_{j_1}\cdot \dots \cdot F_{j_m}$ and $B_J=B_{j_1}\cdot \dots \cdot B_{j_m}$. The commutation relations \eqref{eq:def-rel} imply that $K_i K_{\tau(i)}^{-1} B_j = q_{ij}^{-1} q_{\tau(i)j} B_j K_i K_{\tau(i)}^{-1}$ for all $i,j\in I$ and hence $B_\bc=\sum_J H_\theta B_J$.
  Let $p=p(x_1, \dots, x_n)$ be a noncommutative polynomial in variables $x_i$ for $i\in I$. To shorten notation we write $p(\uF)=p(F_1, \dots, F_n)$, $p(\uE)=p(E_1,\dots,E_n)$, $p(\uB)=p(B_1, \dots, B_n)$, $p(\underline{E_\tau K^{-1}})=p(E_{\tau(1)}K_1^{-1}, \dots, E_{\tau(n)}K_n^{-1})$, $p(\underline{K E_\tau})=p(K_1E_{\tau(1)}, \dots, K_nE_{\tau(n)})$ and $p(\underline{F_\tau K^{-1}}) =p(F_{\tau(1)}K_1^{-1}, \dots, F_{\tau(n)} K_n^{-1})$.
For any $m\in \N$ define
  \begin{align*}
    U^-_{\le m}=\mathrm{span}_\field\{F_{i_1}\dots F_{i_j}\,|\,j\le m\}.
  \end{align*}
By definition of the generators $B_i$ and the defining relations \eqref{eq:def-rel} of $U(\chi)$ we have
  \begin{align}\label{eq:cFBm-1}
    \cF_{m-1}(B_\bc)\subset U^+HU^-_{\le m-1} \qquad \mbox{for any $m\in \N$.}
  \end{align}
  Hence, if $p$ be a non-commutative, homogeneous polynomial of degree $m$ then
  \begin{align*}
     p(\uB)\in \cF_{m-1}(B_\bc)  \quad \Longrightarrow \quad p(\uF)=0.
  \end{align*}
 Hence we obtain a surjective homomorphism of graded algebras 
 \begin{align}\label{eq:varphi}
   \varphi: \gr(B_\bc)\rightarrow H_\theta \ltimes U^-
 \end{align}
 such that $\varphi(B_i)=F_i$ and $\varphi(h)=h$ for all $i\in I$, $h\in H_\theta$. We want to know under which conditions the map $\varphi$ is an isomorphism. To this end, for any homogeneous noncommutative polynomial $p$ of degree $m$ we consider the following property
  \begin{align}\label{eq:Bc-rel}
    p(\uF)=0 \quad \Longrightarrow \quad p(\uB)\in \cF_{m-1}(B_\bc).\tag{$B_\bc$-rel}
  \end{align}
  We consider the set $\N_\rel$ of all degrees for which homogeneous relations in $U^-$ lead to relations in $B_\bc$, that is
  \begin{align}\label{eq:Nrel-def}
    \N_\rel=\{k\in \N\,|\, \mbox{any polynomial $p$ of degree $m\le k$ satisfies \eqref{eq:Bc-rel} }\}.
  \end{align}
By definition of $\N_\rel$ the map $\varphi$ is injective if and only if $\N_\rel=\N$. 
%%%%%%%%%%%%%%%%%%%%%%%%%%%%%%%%%%%%%%%%%%%%%%%%%%%%%%%%%%
\begin{prop}
\label{prop:Bc-assoc-grad}
  The map $\varphi$ is an isomorphism of graded algebras if and only if $\N_\rel=\N$. 
\end{prop}  
%%%%%%%%%%%%%%%%%%%%%%%%%%%%%%%%%%%%%%%%%%%%%%%%%%%%%%%
In Section \ref{sec:prop-c} we will formulate necessary and sufficient conditions on the parameters $\bc$ which imply that $\N_\rel=\N$. First, however, we show in Section \ref{sec:HeisNrel=N} that a quotient of $B_\bc$ inside a Heisenberg double satisfies the relation $\N_\rel=\N$ irrespective of the parameters $\bc$.
For later reference we note the following technical lemma.
  %%%%%%%%%%%%%%%%%%%%%%%%%%%%%%%%%%%%%%%%%%  
\begin{lem}\label{lem:EKF}
  Let $p$ be a homogeneous polynomial. The following are equivalent:
  \begin{enumerate}
    \item $p(\uF)=0$,
    \item $p(\uE)=0$,
    \item $p(\underline{E_{\tau}K^{-1}})=0$,
    \item $p(\underline{K E_{\tau}})=0$.
    \item $p(\underline{F_\tau K^{-1}})=0$.  
  \end{enumerate}  
\end{lem}  
%%%%%%%%%%%%%%%%%%%%%%%%%%%%%%%%%%%%%%%%%5  
\begin{proof}
The equivalence between (1) and (2) follows from $p(\uE)=\omega(p(\uF))$.
As $\tau(\cI)=\cI$, the latter is equivalent to $p(E_{\tau(1)},\dots,E_{\tau(n)})=0$. By the triangular decomposition \eqref{eq:U+H-}, this is equivalent to the relation $p(E_{\tau(1)}K_1^{-1},\dots,E_{\tau(n)}K_n^{-1})=0$ in $U(\chi)$. Indeed, the factor which is obtained by commuting all negative $K_i$-powers to the very right of any monomial of weight $\lambda\in \Z^n$ depends only on $\lambda$ and not on the monomial because $(q_{ij})$ is symmetric. This shows that (2) and (3) are equivalent. The equivalence between (2) and (4) is verified analogously, and so is the equivalence between (1) and (5).
\end{proof}
%%%%%%%%%%%%%%%%%%%%%%%%%%%%%%%%%%%%%%%%%%
%%%%%%%%%%%%%%%%%%%%%%%%%%%%%%%%%%%%%%%%%%%%%%%%%%%%%%%%%  
  \subsection{The Heisenberg double}\label{sec:Heis}
%%%%%%%%%%%%%%%%%%%%%%%%%%%%%%%%%%%%%%%%%%%%%%%%%%%%%%%%%
Let $\Upoly$ be the subalgebra of $U(\chi)$ generated by the elements $F_i$, $E_i K_i^{-1}$, $K_i^{-1}$, $K_i K_{\tau(i)}^{-1}$ for all $i\in I$. Let $G^+$ 
denote the subalgebra of $U(\chi)$ generated by the elements $\Etil_i=E_iK_i^{-1}$ for all $i\in I$. The following description of $\Upoly$ 
in terms of generators and relations follows from the corresponding description of $U(\chi)$.
%%%%%%%%%%%%%%%%%%%%%%%%%%%%%%%%
\begin{lem}\label{lem:Upoly-rels}
  The algebra $\Upoly$ is the factor of the free product of the algebras
\[
G^+, \quad U^-, \quad \field[K_\lambda\,|\, \lambda \in - \N^n + \Z^n_\theta]
\]
by the relations
\begin{equation}
\label{KEF}
K_\lambda \Etil_i = \chi(\lambda, \alpha_i) \Etil_i K_\lambda, \quad K_\lambda F_i = \chi(\lambda, \alpha_i)^{-1} F_i K_\lambda
\end{equation}
for $i \in I$, $\lambda \in - \N^n + \Z^n_\theta$ and the cross relations
\begin{equation}
\label{cross-D}
q^{-1}_{ij}\Etil_i F_j - F_j \Etil_i = \delta_{ij} (1 - K_i^{-2})
\end{equation}
for $i,j \in I$.
\end{lem}  
%%%%%%%%%%%%%%%%%%%%%%%%%%%%%%%%%%
It follows from the above Lemma and from the triangular decomposition \eqref{eq:U+H-} of $U(\chi)$ that $\Upoly$ has a triangular decomposition
\begin{align}\label{eq:Upoly-triang}
  \Upoly \cong G^+ \ot \field[K_\lambda\,|\, \lambda \in - \N^n + \Z^n_\theta] \ot U^-.
\end{align}  
  Observe that $\Upoly$ is a sub-bialgebra of $U(\chi)$ but not a sub-Hopf algebra. By construction we have $B_\bc\subset \Upoly$.
  By definition of the coproduct \eqref{eq:kowEFK} the two sided ideal $\Ifrak^-=\langle K_i^{-1}\,|\,i\in I\rangle$ in $\Upoly$ is a right coideal, that is
  \begin{align*}
     \kow(\Ifrak^-)\subseteq \Ifrak^- \ot \Upoly.
  \end{align*}  
  Hence the quotient $\Heis(\chi)=\Upoly/\Ifrak^-$ is a right $\Upoly$-comodule algebra with generators $\Etil_i=E_iK_i^{-1}, F_i, K_i K_{\tau(i)}^{-1}$ for $i\in I$. 
  We call $\Heis(\chi)$ the {\em{Heisenberg double}} associated to bicharacter $\chi$ and the pre-Nichols algebra $U^+$.  We write
\begin{align*}
  \okow:\Heis(\chi) \rightarrow \Heis(\chi)\ot \Upoly
\end{align*}
to denote the coaction. 
  Let $\kappa: \Upoly \rightarrow \Heis(\chi)$
  be the projection map and observe that
  \begin{align}\label{eq:okowkappa}
     \okow(\kappa(x)) = (\kappa \ot \id)\kow(x) \qquad \mbox{for all $x\in \Upoly$.}
  \end{align}
  Lemma \ref{lem:Upoly-rels} implies that $\Heis(\chi)$ is the factor of the free product of $G^+$, $U^-$ and $\field[K_\lambda\,|\,\lambda\in Z^n_\theta]$ by the relations \eqref{KEF} for $i\in I$, $\lambda\in \Z^n_\theta$ and the cross relations
  \begin{align*}
     q^{-1}_{ij}\Etil_i F_j - F_j \Etil_i = \delta_{ij}
  \end{align*}
for $i,j \in I$.
  This implies that $\Heis(\chi)$ has a triangular decomposition
  \begin{align}\label{eq:Heis-triang}
       \Heis(\chi) \cong G^+ \rtimes H_\theta \ltimes U^-
  \end{align}
  where $G^+=\field\langle \Etil_i\,|\,i\in I\rangle$.
%%%%%%%%%%%%%%%%%%%%%%%%  
\begin{rema} In the special case of the quantized universal enveloping algebra $U_q(\gfrak)$ of a symmetrizable Kac--Moody algebra at 
a non-root of unity $q$ and $\tau=\id$, the algebra $\Heis(\chi)$ is isomorphic to Kashiwara's bosonic algebra ${\mathscr{B}}_q(\gfrak)$ \cite[Section 3.3]{a-Kash91}.
When $\gfrak$ is finite dimensional, in \cite[Theorem 6.2]{a-GoodYak16p} it was proved that it has the structure of a quantum cluster algebra; the algebra
was denoted by $\cU^-_\op \bowtie \cU^+$ in \cite[Theorem 4.7, Remark 4.8]{a-GoodYak16p}.
\end{rema}
%%%%%%%%%%%%%%%%%%%%%%%%%%%%%%%%%%%%%%%%%%%%%%%%%%%%%%%
The projection maps $P_\lambda$ for $\lambda\in \Z^n$ and $\pi_{\alpha,\beta}$ for $\alpha,\beta\in \Z^n_+$ from the end of Section \ref{sec:setting} have analogs for the Heisenberg double. For $\alpha \in \Z^n_+$ we write $G^+_\alpha=U^+_\alpha H\cap G^+$. In view of the triangular decomposition \eqref{eq:Heis-triang} of the Heisenberg double we get a direct sum decomposition
  \begin{align}\label{eq:HG+HU-}
    \Heis(\chi)=\bigoplus_{\alpha,\beta\in \N^n,\mu\in \Z^n_\theta} G_\alpha^+K_\mu U^-_{-\beta}.
  \end{align}
  Now the projection $P_\lambda$ from \eqref{eq:Plambda-def} induces a projection
\begin{align}\label{eq:oPlambda}
  \oP_\lambda:\Heis(\chi)\rightarrow \bigoplus_{\mu-\alpha-\beta=\lambda} G^+_\alpha K_\mu U^-_{-\beta}.
\end{align}
By \eqref{eq:kowPlambda} we obtain
\begin{align}\label{eq:kowoPlambda}
  \okow\circ \oP_\lambda(x)=(\id\ot P_\lambda)\okow(x) \qquad \mbox{for all $x\in \Heis(\chi)$.}
\end{align}
Moreover, for $\alpha,\beta\in \Z_+^n$ let
\begin{align}\label{eq:opi-alphabeta}
  \opi_{\alpha,\beta}: \Heis(\chi)\rightarrow G^+_\alpha H_\theta U^-_{-\beta}
\end{align}
be the projection with respect to the decomposition \eqref{eq:HG+HU-}. We consider the partial order on $\Z^n$ given by
  \begin{align*}
    (j_1, \dots,j_n)\le (j'_1, \dots,j'_n) \quad \Longleftrightarrow \quad j_i\le j'_i \quad \mbox{for $i=1, \dots, n$.}
  \end{align*}
  For later use we note the following property of the projection map \eqref{eq:opi-alphabeta}.
%%%%%%%%%%%%%%%%%%%%%%%%%%%%%%%%%%%%%%%
\begin{lem}\label{lem:neq0}
  Let $u\in \Heis(\chi)$ and let $\alpha\in \N^n$ be maximal with respect to the partial order such that $\opi_{\alpha,\beta}(u)\neq 0$ for some $\beta\in \N^n$. Then
  \begin{align*}
     0\neq (\id\ot \pi_{\alpha,0})\okow(u).
  \end{align*}
\end{lem}
%%%%%%%%%%%%%%%%%%%%%%%%%%%%%%%%%%%%%%%
\begin{proof}
  Using the direct sum decomposition \eqref{eq:HG+HU-} we write
  \begin{align*}
     u=\sum_{\gamma,\beta\in \N^n, i} x_{\gamma,i} u^0_{\gamma,\beta,i} y_{-\beta,i}
  \end{align*}
  where $x_{\gamma,i}\in G^+_\gamma$ are linearly independent, and $u^0_{\gamma,\beta,i}\in H_\theta$ and $y_{-\beta,i}\in U_{-\beta}^-$. Let now $\alpha$ be as in the assumption and set
  \begin{align*}
    u_\alpha=\sum_{\beta\in \N^n, i} x_{\alpha,i} u^0_{\alpha,\beta,i} y_{-\beta,i}
  \end{align*}
  and $u_{\neq\alpha}=u-u_\alpha$. By the maximality of $\alpha$ we have
  \begin{align*}
     (\id \ot \pi_{\alpha,0})\okow(u_{\neq \alpha})=0.
  \end{align*}
  Hence, using Sweedler notation for the coaction $\okow$ in the form
  $\okow(u)=u_{(0)}\ot u_{(1)}$ we obtain
    \begin{align*}
      (\id \ot \pi_{\alpha,0})\okow(u) &=
      (\id \ot \pi_{\alpha,0})\okow(u_{\alpha})\\
      &= \sum_{\beta\in \N^n, i} u^0_{\alpha,\beta,i(1)} y_{-\beta,i} \ot x_{\alpha,i}  u^0_{\alpha,\beta,i(2)} K_{-\beta}
    \end{align*}
  and the latter expression is non-zero by the linear independence of the $x_{\alpha,i}$.  
\end{proof}  
%%%%%%%%%%%%%%%%%%%%%%%%%%%%%%%%%%%%%%%
\subsection{Relations in $\oB_\bc$}\label{sec:HeisNrel=N}
%%%%%%%%%%%%%%%%%%%%%%%%%%%%%%%%%%%%%%%
Recall the projection map $\kappa:\Upoly\rightarrow \Heis(\chi)$ and define $\oB_\bc=\kappa(B_\bc)$.  We also use the notation $\overline{x}=\kappa(x)$ for $x\in \Upoly$ and write in particular $\oB_i=\kappa(B_i)$ for all $i\in I$. We proceed as in Section \ref{sec:Bc}. The algebra $\oB_\bc$ is filtered by a degree function with $\deg(\oB_i)=1$ for all $i\in I$ and $\deg(h)=0$ for all $h\in H_\theta$. Let $\gr(\oB_\bc)$ denote the associated graded algebra. We obtain a surjective homomorphism of graded algebras
\begin{align*}
  \overline{\varphi}: \gr(\oB_\bc)\rightarrow H_\theta \ltimes U^- 
\end{align*}
such that $\overline{\varphi}(\oB_i)=\oF_i$ and $\overline{\varphi}(h)=h$ for all $i\in I$, $h\in H_\theta$. For any non-commutative polynomial $p$ in variables $x_1,\dots,x_n$ we write $p(\underline{\oB})=p(\oB_1,\dots,\oB_n)$.
Assume that the noncommutative polynomial $p$ has degree $m$. In analogy to property \eqref{eq:Bc-rel} we are interested in the following property
  \begin{align}\label{eq:oBc-rel}
    p(\uF)=0 \quad \Longrightarrow \quad p(\overline{\uB})\in \cF_{m-1}(\oB_\bc). \tag{$\oB_\bc$-rel}
  \end{align}
  We consider the set $\overline{\N}_\rel$ of all degrees for which homogeneous relations in $U^-$ lead to relations in $B_\bc$, that is
  \begin{align}
    \overline{\N}_\rel=\{k\in \N\,|\, \mbox{any polynomial $p$ of degree $l\le k$ satisfies \eqref{eq:oBc-rel} }\}.
  \end{align}
  We know that $1\in \overline{\N}_\rel$. The following lemma provides a main step in the proof that $\N=\overline{\N}_\rel$ below.
%%%%%%%%%%%%%%%%%%%%%%%%%%%%%%%%%%%%%%%%%%%%%%%
\begin{lem}\label{lem:Nrel}
  Let $\beta\in \Z^n$ with $\beta>0$ and $n\in \overline{\N}_\rel$. Then
  \begin{align*}
     \kappa(G^-K_{-\beta}) \cap \sum_{|J|\le n} H_\theta \oB_J=\{0\}.
  \end{align*}  
\end{lem}  
%%%%%%%%%%%%%%%%%%%%%%%%%%%%%%%%%%%%%%%%%%%%%%%
\begin{proof}
  Let $a\in  \kappa(G^-K_{-\beta}) \cap \sum_{J\le n} H_\theta \oB_J$. Choose $k\in \N_0$ minimal such that $a\in  \kappa(G^-K_{-\beta}) \cap \sum_{J\le k} H_\theta \oB_J$. We want to show that $k=0$. Assume on the contrary that $k\ge 1$ and write $a=a_0+a_k$ with $a_0\in \sum_{|J|\le k-1}H_\theta \oB_J$ and  $a_k\in \sum_{|J|= k}H_\theta \oB_J$. By the minimality of $k$ we have $a_k\neq 0$. Write $a_k=\sum_i h_i p_i(\oB)$ where $h_i\in H_\theta$ are linearly independent and $p_i=p_i(x_1,\dots,x_n)$ are homogeneous polynomials of degree $k$. The relation $a\in \kappa(G^-K_{-\beta})$ together with the definition of the generators $\oB_i$ of $\oB_\bc$ (and the linear independence of the $h_i$) imply that $p_i(\underline{E_{\tau}K^{-1}})=0$. Hence we have $p_i(\uF)=0$ by Lemma \ref{lem:EKF}. As $k\in \overline{\N}_\rel$ we obtain
$p_i(\overline{\uB})\in \sum_{|J|\le k-1}H_\theta \oB_J$. But then $a=a_0+a_k\in \sum_{|J|\le k-1}H_\theta \oB_J$ in contradiction to the minimality of $k$. Hence the assumption $k\ge 1$ was incorrect and we obtain $k=0$. Hence $a\in \kappa(G^- K_{-\beta}\cap H_\theta)=\{0\}$ which concludes the proof of the lemma. 
\end{proof}  
%%%%%%%%%%%%%%%%%%%%%%%%%%%%%%%%%%%%%%%%%%%%%%%
With these preparations we can show that $\oB_\bc$ is not too big.
%%%%%%%%%%%%%%%%%%%%%%%%%%%%%%%%%%%%%%%%%%%%%%%
\begin{prop}\label{prop:Nrel}
  $\overline{\N}_\rel=\N$.
\end{prop}
%%%%%%%%%%%%%%%%%%%%%%%%%%%%%%%%%%%%%%%%%%%%%%%
\begin{proof}
We proceed by induction. Let $k\in \N$ and assume that $\{1,\dots,k-1\}\subseteq \overline{\N}_\rel$.
  Let $p$ be a polynomial of degree $k$ such that $p(\uF)=0$. 
  Without loss of generality we may assume that $p$ is homogeneous of degree $\lambda\in \Z^n$ with $|\lambda|=k$. Write $Y=p(\overline{\uB})$ and $Z=\oP_{-\lambda}(Y)$ where $\oP_{-\lambda}$ is the projection operator from \eqref{eq:oPlambda}. Note that $Z\in \oB_\bc$ by \eqref{eq:kowoPlambda}. Relations \eqref{eq:kowB} and \eqref{eq:okowkappa} imply that
  \begin{align*}
    \okow(Y) \in Y\ot K_{-\lambda} + \sum_{|J|\le k-1} H_\theta \oB_J \ot \Upoly. 
  \end{align*}  
  Hence \eqref{eq:kowoPlambda} implies that the element $Z$ satisfies the relation
  \begin{align}\label{eq:okowZ}
     \okow(Z)\in Y\ot K_{-\lambda} + \sum_{|J|\le k-1} H_\theta \oB_J \ot P_{-\lambda}(\Upoly). 
  \end{align}
   We now prove $Z=0$ as in \cite[Proposition 5.16]{a-Kolb14}. Assume that $Z\neq 0$. Let $\alpha\in \Z^n_+$ be maximal with respect to the partial order such that $\opi_{\alpha,\beta}(Z)\neq 0$ for some $\beta\in \Z^n_+$. By Lemma \ref{lem:EKF} we know that $\alpha<\lambda$. Moreover, by Lemma \ref{lem:neq0} we have
  \begin{align}\label{eq:pialpha0}
      0\neq (\id \ot \pi_{\alpha,0})\okow(Z)\in \kappa(G^-K_{-\lambda+\alpha}) \ot U^+ K_{-\lambda}.
    \end{align}  
  If $\alpha\neq 0$ then \eqref{eq:okowZ} implies that
  \begin{align*}
     (\id \ot \pi_{\alpha,0})\okow(Z)\in \Big(\kappa(G^-K_{-\lambda+\alpha})\cap \sum_{|J|\le k-1}H_\theta \oB_J\Big)\ot U^+ K_{-\lambda}.
  \end{align*}
  However, the left hand side of the above expression is $\{0\}$ by Lemma \ref{lem:Nrel} in contradiction to \eqref{eq:pialpha0}. Hence $\alpha=0$ and $Z\in \kappa(G^-K_{-\lambda})$. But then the relation $p(\uF)=0$ implies that $Z\in \bigoplus_{\beta<\lambda}\kappa(G_{-\beta}^-K_{-\lambda})=\{0\}$. Hence $Z=0$. 

  Now we apply the counit $\vep$ to the second tensor factor in \eqref{eq:okowZ} to obtain
  \begin{align*}
    Y\in \sum_{|J|\le k-1} H_\theta B_J.
  \end{align*}
  Hence the polynomial $p$ satisfies property \eqref{eq:oBc-rel}. This proves that $k\in \overline{\N}_\rel$.
\end{proof}
%%%%%%%%%%%%%%%%%%%%%%%%%%%%%%%%%%%%%%%%%%%%%%%%%%%%%%%%%%
 We can now repeat the argument which led to Proposition \ref{prop:Bc-assoc-grad} to obtain the following result.
%%%%%%%%%%%%%%%%%%%%%%%%%%%%%%%%%%%%%%%%%%%%%%%%%%%%%%%%%%
\begin{thm}\label{thm:Heis-no-cond}
  For all pre-Nichols algebras $U^+$ of diagonal type and $\bc\in \field^n$, 
  the map $\overline{\varphi} : \gr(\oB_\bc)\rightarrow H_\theta \ltimes U^- $ is an isomorphism of graded algebras.
\end{thm}  
%%%%%%%%%%%%%%%%%%%%%%%%%%%%%%%%%%%%%%%%%%%%%%%%%%%%%%%%%%
%%%%%%%%%%%%%%%%%%%%%%%%%%%%%%%%%%%%%%%%%%%%%%%%%%%%%%%%%%
\subsection{Relations in $B_\bc$}
\label{sec:prop-c}
%%%%%%%%%%%%%%%%%%%%%%%%%%%%%%%%%%%%%%%%%%%%%%%%%%%%%%%%%%
We now want to see how much of the argument in the previous section translates from $\oB_\bc$ to the algebra $B_\bc$. Recall the definition of the subset $\N_\rel\subset \N$ from \eqref{eq:Nrel-def}. A word by word translation of the proof of Lemma \ref{lem:Nrel} gives the following result.
%%%%%%%%%%%%%%%%%%%%%%%%%%%%%%%%%%%%%%%%%%%%%%%%%%%%%%%%%%
 \begin{lem}\label{lem:Nrel2}
   Let $\beta\in \Z^n$ with $\beta>0$ and $n\in \N_\rel$. Then
   \begin{align*}
      G^-K_{-\beta} \cap \sum_{|J|\le n} H_\theta B_J=\{0\}.
   \end{align*}  
 \end{lem}  
%%%%%%%%%%%%%%%%%%%%%%%%%%%%%%%%%%%%%%%%%%%%%%%%%%%%%%%%%%
 Translating the initial steps of the proof of Proposition \ref{prop:Nrel} into the setting of $B_\bc$ we obtain the following result.
%%%%%%%%%%%%%%%%%%%%%%%%%%%%%%%%%%%%%%%%%%%%%%%%%%%%%%%%%%
 \begin{prop}\label{prop:K-laminB}
   Let $p$ be a homogeneous polynomial of degree $\lambda\in \Z^n$ with $|\lambda|$ minimal such that $p(\uF)=0$ but
   \begin{align*}
      p(\uB) \notin \sum_{|J|<|\lambda|} H_\theta B_J.
   \end{align*}
   Then $P_{-\lambda}(p(\uB))= \pi_{0,0}(P_{-\lambda}(p(\uB)))=a_pK_{-\lambda}$ for some $a_p\in \field^\times$ and hence $K_{-\lambda}\in B_\bc$.
 \end{prop}  
%%%%%%%%%%%%%%%%%%%%%%%%%%%%%%%%%%%%%%%%%%%%%%%%%%%%%%%%
 \begin{proof} 
   Write $Y=p(\uB)$ and $Z=P_{-\lambda}(Y)$. Equation \eqref{eq:kowB} for the coproducts of the generators $B_i$ implies that
   \begin{align*}
      \kow(Y) \in Y\ot K_{-\lambda} + \sum_{|J|\le |\lambda|-1} H_\theta B_J \ot \Upoly. 
  \end{align*}  
  Hence \eqref{eq:kowPlambda} implies that the element $Z$ satisfies the relation
  \begin{align}\label{eq:kowZ}
     \kow(Z)\in Y\ot K_{-\lambda} + \sum_{|J|\le |\lambda|-1} H_\theta B_J \ot P_{-\lambda}(\Upoly). 
  \end{align}
  If $Z=0$ then we can apply the counit $\vep$ to the second tensor leg of the above expression and obtain $Y\in \sum_{|J|\le |\lambda|-1} H_\theta B_J$ in contradiction to the assumption. Hence $Z\neq 0$.

  Let $\alpha\in \Z^n_+$ be maximal such that $\pi_{\alpha,\beta}(Z)\neq 0$ for some $\beta\in \Z^n_+$. By Lemma \ref{lem:EKF} we know that $\alpha<\lambda$. Moreover, in complete analogy to Lemma \ref{lem:neq0}, we obtain
  \begin{align}\label{eq:pialpha02}
      0\neq (\id \ot \pi_{\alpha,0})\kow(Z)\in G^-K_{-\lambda+\alpha} \ot U^+ K_{-\lambda}.
  \end{align}   
  If $\alpha\neq 0$ then \eqref{eq:kowZ} implies that
  \begin{align*}
     (\id \ot \pi_{\alpha,0})\kow(Z)\in \Big(G^-K_{-\lambda+\alpha}\cap \sum_{|J|\le |\lambda|-1}H_\theta B_J \Big) \ot U^+ K_{-\lambda}.
  \end{align*}
  However, the left hand side of the above expression is $\{0\}$ by Lemma \ref{lem:Nrel2} in contradiction to \eqref{eq:pialpha02}. Hence $\alpha=0$ and $Z\in G^-K_{-\lambda}$.

  Now choose $\beta\in \Z^n_+$ maximal such that $\pi_{0,\beta}(Z)\neq 0$. As $p(\uF)=0$ we have $\beta<\lambda$. In analogy to Lemma \ref{lem:neq0} we have
  \begin{align}\label{eq:K-l}
    0\neq (\id \ot \pi_{0,\beta})\kow(Z)\in K_{-\lambda+\beta} \ot G^-_{\beta} K_{-\lambda}
  \end{align}
  Comparison with \eqref{eq:kowZ} and application of Lemma \ref{lem:Nrel2} implies (as before for $\alpha$) that $\beta=0$. Hence $Z=\pi_{0,0}(Z)=a_p K_{-\lambda}$ for some $a_p\in \field$ and the claim follows from the relation $Z\in B_\bc\setminus \{0\}$. 
 \end{proof}
%%%%%%%%%%%%%%%%%%%%%%%%%%%%%%%%%%%%%%%%%%%%%%%%%%%%%%%%%
 Recall that $\cI\subset T(V^+(\chi))$ denotes the ideal in the tensor algebra such that
 \begin{align*}
   U^+=T(V^+(\chi))/\cI.
 \end{align*}
 The above proposition provides us with a method to check that condition \eqref{eq:Bc-rel} holds for all polynomials.
 %%%%%%%%%%%%%%%%%%%%%%%%%%%%%%%%%%%%%%%%%%%%%%%%%%%%%%%%
\begin{cor}\label{cor:NrelN}
  Let $p_j$ for $j=1,\dots,k$ be homogeneous, non-commutative polynomials of degree $\lambda_j\in \Z^n_+$, respectively, such that the set $\{p_j(\uE)\,|\,j=1,\dots,k\}$ generates the ideal $\cI$. Assume that
  \begin{align}\label{eq:p-assume}
      \pi_{0,0}\circ P_{-\lambda_j}(p_j(\uB))=0 \qquad \mbox{for $j=1, \dots,k$.}
  \end{align}
Then $\N_\rel=\N$.  
\end{cor}   
%%%%%%%%%%%%%%%%%%%%%%%%%%%%%%%%%%%%%%%%%%%%%%%%%%%%%%%%%%
\begin{proof}
  We prove this indirectly. Let $p$ be a homogeneous polynomial of minimal degree $\lambda\in \Z^n_+$ such that $p(\uE)\in \cI$ but
  \begin{align*}
     p(\uB)\notin \sum_{|J|<|\lambda|}H_\theta B_J.
  \end{align*}
  We can write
  \begin{align*}
    p=\sum_{j=1}^k \sum_\ell q_{j,\ell}\, p_j\, r_{j,\ell}
  \end{align*}
  where $q_{j,\ell}, r_{r,\ell} \in T(V^+(\chi))$ are homogeneous polynomials and
  \begin{align*}
    \deg(q_{j,\ell}) + \lambda_j + \deg(r_{j,\ell})=\lambda \qquad \mbox{for all $j,\ell$.}
  \end{align*}
  By the minimality assumption, any summand $s= q_{j,\ell} p_j r_{j,\ell}$ with $\deg(s)>\deg(p_j)$ satisfies $s(\uB)\in \sum_{|J|<|\lambda_j|}H_\theta B_J$ and hence may be omitted. Thus we may assume that
  \begin{align*}
      p=\sum_{j=1}^k a_j p_j \qquad \mbox{for some $a_j\in \field$.}
  \end{align*}
However, by Proposition \ref{prop:K-laminB} this is impossible, because of the assumption \eqref{eq:p-assume}.  
\end{proof}
%%%%%%%%%%%%%%%%%%%%%%%%%%%%%%%%%%%%%%%%%%%%%%%%%%%%%%%%%%
Corollary \ref{cor:NrelN} suggests the following assumption about the parameters $\bc$ in the definition of the coideal subalgebra $B_\bc$:
\begin{align}\label{assume:parameters}
  \mbox{\begin{minipage}[t]{11.5cm}
      The ideal $\cI\subset T(V^+(\chi))$ is generated by homogeneous, non-commutative polynomials $p_j(\uE)$ for $j=1,\dots,k$ of degree $\lambda_j\in \Z^n_+$, respectively, for which $\pi_{0,0}\circ P_{-\lambda_j}(p_j(\uB))=0$.
        \end{minipage}}\tag{$\bc$}
\end{align}
Condition \eqref{assume:parameters} provides a reformulation of the condition $\N_\rel=\N$ which can be verified in explicit examples.
%%%%%%%%%%%%%%%%%%%%%%%%%%%%%%%%%%%%%%%%%%%%%%%%%%%%%%%%%%%
\begin{thm}\label{thm:c-cond}
  For all pre-Nichols algebras $U^+$ of diagonal type the following statements are equivalent:
\begin{enumerate}
\item The map  $\varphi: \gr(B_\bc)\rightarrow H_\theta \ltimes U^-$ is an isomorphism.
\item $\N_\rel=\N$.
\item Condition \eqref{assume:parameters} holds.
\end{enumerate}  
Moreover, if $\N_\rel\neq \N$ then there exists $\lambda\in \N^n \setminus\{0\}$ such that $K_{-\lambda}\in B_\bc$.  
\end{thm}
%%%%%%%%%%%%%%%%%%%%%%%%%%%%%%%%%%%%%%%%%%%%%%%%%%%%%%%%%%%
\begin{proof}
The equivalence between (1) and (2) is the statement of Proposition \ref{prop:Bc-assoc-grad}. By Corollary \ref{cor:NrelN} we have that (3) implies (2). Conversely, if condition \eqref{assume:parameters} does not hold, then Proposition \ref{prop:K-laminB} implies that $P_{-\lambda}(p(\uB))=a_pK_{-\lambda}$ with $a_p\in \field^\times$ for some homogeneous polynomial $p$ of degree $\lambda$ for which $p(\uF)=0$. As $$P_{-\lambda}\left(\sum_{|J|<|\lambda|}H_\theta B_J\right)=0,$$
we see that the polynomial $p$ violates condition \eqref{eq:Bc-rel}. This proves that (2) implies (3) and the final statement of the theorem.
\end{proof}
%%%%%%%%%%%%%%%%%%%%%%%%%%%%%%%%%%%%%%%%%%%%%%%%%%%%%%%%%%%
If condition \eqref{assume:parameters} holds then the above theorem allows us to write down a basis of $B_\bc$ as a left $H_\theta$-module. Let $\displaystyle\cJ\subset \bigcup_{k=0}^\infty I^k$ be a subset of multiidices such that
$\{F_J\,|\,J\in \cJ\}$ is a linear basis of $U^-$. The following corollary is a version of \cite[Proposition 6.2]{a-Kolb14} in our setting. It is a consequence of the implication (3) $\Rightarrow$ (1) in the theorem. 
%%%%%%%%%%%%%%%%%%%%%%%%%%%%%%%%%%%%%%%%%%%%%%%%%%%%%%%%%%%
\begin{cor}\label{cor:basis-Bc}
  Let $U^+$ be a pre-Nichols algebra of diagonal type and assume that condition \eqref{assume:parameters} holds. Then $B_\bc$ is a free left $H_\theta$-module with basis $\{B_J\,|\,J\in \cJ\}$.
\end{cor}
\begin{rema}\label{rem:Iwasawa} One of the reasons for which the condition in Theorem \ref{thm:c-cond} is important is its relation to 
Iwasawa decompositions. The definition of the filtration \eqref{eq:cFBm-1} implies at once that for all $B_\bc$ the following statements are equivalent:
\begin{enumerate}
\item The map  $\varphi: \gr(B_\bc)\rightarrow H_\theta \ltimes U^-$ is an isomorphism.
\item The algebra $\Upoly$ admits the Iwasawa decomposition 
$$\Upoly \cong G^+ \otimes \field[K_i^{-1} \mid i \in I] \otimes B_\bc.$$
\item The algebra $U(\chi)$ admits the Iwasawa decomposition 
$$U(\chi) \cong U^+ \ot \field [K_i^{\pm 1} \mid i \in I_\tau] \ot B_\bc,$$
where $I_\tau \subset I$ is a set of representatives of the $\tau$-orbits in $I$.
\end{enumerate}   
\end{rema}
%%%%%%%%%%%%%%%%%%%%%%%%%%%%%%%%%%%%%%%%%%%%%%%%%%%%%%%%%%%

%%%%%%%%%%%%%%%%%%%%%%%%%%%%%%%%%%%%%%%%%%%%%%%%%%%%%%%%%%
%%%%%%%%%%%%%%%%%%%%%%%%%%%%%%%%%%%%%%%%%%%%%%%%%%%%%%%%%%  
 \subsection{The negative Heisenberg double}\label{sec:Hmin}
 %%%%%%%%%%%%%%%%%%%%%%%%%%%%%%%%%%%%%%%%%%%%%%%%%%%%%%%%%%
 Recall the algebra $\Upoly$ defined at the beginning of Section \ref{sec:Heis}.
In this section we show that condition \eqref{assume:parameters} for $B_\bc$ to be of the right size can be verified in a simpler algebra which is closely related to quantum Weyl algebras. This fact will be applied extensively in Section \ref{sec:Examples}.
 
 The algebra $\Upoly$ has an $\N$-filtration $\cF$ defined by the following degree function on the generators
 \begin{align}\label{eq:deg-Hmin}
    \deg(\Etil_i)=\deg(F_i)=\deg(K_i^{-1})=1, \qquad \deg(K_i K_{\tau(i)}^{-1})=0
 \end{align}
 for all $i\in I$. It follows from the triangular decomposition \eqref{eq:Upoly-triang} that the multiplication map
 \begin{align}\label{eq:Fm-iso}
    \bigoplus_{{\alpha,\beta,\gamma\in \N^n}\atop{|\alpha+\beta+\gamma|\le m} } H_\theta\ot G_\alpha^+ \ot \field K_{-\beta} \ot U^-_{-\gamma} \rightarrow \cF_m(\Upoly) 
 \end{align}
 is a linear isomorphism for any $m\in \N$. With the notation
 \begin{align*}
    (G^+\ot \field[K_i^{-1}\,|\, i\in I]\ot U^-)_m= \bigoplus_{{\alpha,\beta,\gamma\in \N^n}\atop{|\alpha+\beta+\gamma|=m}} G_\alpha^+ \ot \field K_{-\beta} \ot U^-_{-\gamma}.
 \end{align*}
 the linear isomorphism \eqref{eq:Fm-iso} provides a direct sum decomposition
 \begin{align}\label{eq:Fm-sum}
    \cF_m(\Upoly)=\cF_{m-1}(\Upoly) \oplus \big(H_\theta \ot (G^+\ot \field[K_i^{-1}\,|\, i\in I]\ot U^-)_m\big).
 \end{align}
 We call the graded algebra
 \begin{align*}
    \Hmin=\mathrm{gr}_\cF(\Upoly)
 \end{align*}
 associated to the filtration $\cF$ of $\Upoly$ the {\em{negative Heisenberg double}} associated to the pre-Nichols algebra $U^+$. By \eqref{eq:Fm-sum} for any $m\in \N$ the graded component $\Hmin_m$ is a free $H_\theta$-module
 \begin{align*}
    \Hmin_m \cong H_\theta \ot (G^+\ot \field[K_i^{-1}\,|\, i\in I]\ot U^-)_m.
 \end{align*}  
 In particular $\Hmin_0\cong H_\theta$. The above also implies that $G^+$, $\field[K_\lambda\,|\, \lambda\in -\N^n+\Z_\theta^n]$ and $U^-$ are graded subalgebras of $\Hmin$ and that the multiplication map
 \begin{align}\label{eq:Heis-iso}
     G^+ \ot \field[K_\lambda\,|\, \lambda\in -\N^n+\Z_\theta^n]\ot U^- \rightarrow \Hmin
 \end{align}  
is a linear isomorphism.  The presentation of $\Upoly$ in Lemma \ref{lem:Upoly-rels} and the triangular decomposition \eqref{eq:Heis-iso} allow us to describe the negative Heisenberg double $\Hmin$ in terms of generators and relations.
 %%%%%%%%%%%%%%%%%%%%%%%%%%%%%%%%%%%%%%%%%%%%%%%%%%%%%%%%%%%%%%%%%%%
 \begin{lem}\label{lem:Hmin-rels}
   The negative Heisenberg double $\Hmin$ is canonically isomorphic to the quotient of the free product of the algebras $G^+$, $\field[K_\lambda\,|\, \lambda\in \-\N^n+\Z_\theta^n]$ and $U^-$ by the relations \eqref{KEF} and the cross relations
   \begin{align}\label{eq:cross-Hmin}
       q_{ij}^{-1} \Etil_i F_j - F_j \Etil_i=-\delta_{ij} K_i^{-2} \qquad \mbox{for all $i,j\in I$}.
   \end{align}
 \end{lem}  
 %%%%%%%%%%%%%%%%%%%%%%%%%%%%%%%%%%%%%%%%%%%%%%%%%%%%%%%%%%%%%%%%%%%
 \begin{proof}
   Let ${\Hmin}'$ be the algebra described in the lemma. The algebra ${\Hmin}'$ is graded by the degree function \eqref{eq:deg-Hmin} because the defining relations for ${\Hmin}'$ are homogeneous. It follows from Lemma \ref{lem:Upoly-rels} that there is a surjective homomorphism of graded algebras
   \begin{align*}
     \varphi: {\Hmin}' \rightarrow \Hmin
   \end{align*}
   which maps $\Etil_i, K_{-\lambda}, F_i\in {\Hmin}'$ to $\Etil_i, K_{-\lambda}, F_i\in \Hmin$, respectively, for all $i\in I, \lambda\in -\N^n+\Z^n_\theta$. The defining relations for ${\Hmin}'$ imply that the multiplication map
   \begin{align*}
      \mu':\Hmin\cong G^+ \ot \field[K_\lambda\,|\, \lambda\in \-\N^n+\Z_\theta^n]\ot U^- \rightarrow {\Hmin}'
   \end{align*}
   is surjective where we use the triangular decomposition \eqref{eq:Heis-iso}.
   With this identification the composition $\varphi\circ\mu':\Hmin \rightarrow \Hmin$ is the identity map which implies that $\varphi$ is also injective.
 \end{proof}  
%%%%%%%%%%%%%%%%%%%%%%%%%%%%%%%%%%%%%%%%%%%%%%%%%%
 We now show that condition \eqref{assume:parameters} can be verified in the negative Heisenberg double. Let $G_+^+$ and $U^-_+$ denote the augmentation ideals of $G^+$ and $U^-$, respectively. The triangular decomposition \eqref{eq:Heis-iso} of $\Hmin$ implies that
 \begin{align}\label{eq:Hmin00}
    \Hmin= \field[K_\lambda \mid \lambda \in - \N^n + \Z_\theta^n] \oplus \big( G^+_+ \Hmin + \Hmin U^-_+\big).
 \end{align}  
 Let $\pi_{0,0}^\vee:\Hmin\to \field[K_\lambda \mid \lambda \in - \N^n + \Z_\theta^n]$ denote the projection onto the first term in \eqref{eq:Hmin00}. For any $i\in I$ we define $B_i^\vee=F_i+c_i \Etil_{\tau(i)}(K_\tau(i)K_i^{-1})\in \Hmin$, and for any non-commutative polynomial $p(x_1,\dots,x_n)$ we write $p(\uB^\vee)=p(B_1^\vee,\dots,B_n^\vee)$.
%%%%%%%%%%%%%%%%%%%%%%%%%%%%%%%%%%%%%%%%%%%%%%%%5
 \begin{thm} \label{thm:Hmin}
Let $U^+$ be a pre-Nichols algebra of diagonal type corresponding to a bicharacter $\chi$. Let $p(x_1, \ldots, x_n)$ be a homogeneous, non-commutative polynomial of degree $\lambda \in \N^n$. Then
\begin{align}\label{eq:piPpi}
\pi_{0,0}\circ P_{- \lambda} (p(\uB)) = \pi^\vee_{0,0}(p(\uB^\vee)). 
\end{align}
Furthermore, if 
\begin{equation}
\label{no-good-deg}
\lambda \notin \oplus_{i \in I} \N ( \alpha_i + \alpha_{\tau(i)} ), 
\end{equation}
then $P_{- \lambda} \circ \pi_{0,0} (p(\uB)) = 0$ in $\Upoly$.
\end{thm}
%%%%%%%%%%%%%%%%%%%%%%%%%%%%%%%%%%%%%%%%%%%%%%%%%%%
\begin{proof}
   By Lemma \ref{lem:Hmin-rels} the negative Heisenberg double is $-\N+\Z^n_\theta$ graded by the degree function given by
   \begin{align*}
     \deg(\Etil_i)=\deg(F_i)=\deg(K_i^{-1})=-\alpha_i, \qquad \deg(K_iK_{\tau(i)}^{-1})=\alpha_i-\alpha_{\tau(i)}
   \end{align*}
   for all $i\in I$.  %By construction we have $B_i^\vee\in \Hmin_{-\alpha_i}$ and
%   \begin{align}
%      \Hmin_m=\bigoplus_{{\mu\in \N^n+\Z^n_\theta}\atop{|\mu|=m}} \Hmin_{-\mu}.
%   \end{align}  
   For any $\mu \in  \N^n+\Z^n_\theta$ let $P^\vee_{-\mu}:\Hmin \rightarrow \Hmin_{-\mu}$ be the projection onto the graded component $\Hmin_{-\mu}$.

   Let $\lambda=\sum_{i\in I} m_i\alpha_i\in \N^n$ and set $m=|\lambda|=\sum_{i\in I} m_i$. As $\cF_{m-1}(\Upoly)\subseteq \mathrm{Ker}(P_{-\lambda})$ we obtain a commutative diagram
 \begin{align}\label{eq:PP-diag}
     \xymatrix{
     \cF_{m-1}(\Upoly) \ar@{^{(}->}[r]& \cF_{m}(\Upoly) \ar[d]_{\pi_{0,0}^{\phantom{\vee}}\circ P_{-\lambda}} \ar[r]& \Hmin_m \ar[dl]^{\pi^\vee_{0,0}\circ P^\vee_{-\lambda}}\\
     & \field K_{-\lambda} & }
\end{align}
Let now $p(x_1,\dots,x_n)$ be a homogeneous non-commutative polynomial of degree $\lambda$. As $B_i^\vee\in \Hmin_{-\alpha_i}$ the element $p(\uB^\vee)\in \Hmin_m$ is homogeneous of degree $-\lambda$ and hence $\pi_{0,0}^\vee \circ P^\vee_{-\lambda} (p(\uB^\vee))=\pi_{0,0}^\vee (p(\uB^\vee))$. The relation \eqref{eq:piPpi} now follows from the commutativity of the diagram \eqref{eq:PP-diag}.

To prove the second statement in the theorem write $p(\uB)$ as a linear combination of noncommutative monomials in $F_i$ and $\Etil_{i} (K_i K_{\tau(i)}^{-1})$ for $i \in I$. Here we distribute parenthesis, but do not commute the $\Etil$ and $F$ generators. If \eqref{no-good-deg} holds, 
then there is no monomial of this kind that contains equal number of 
terms $F_i$ and $\Etil_{i} (K_{i} K_{\tau(i)}^{-1})$ for all $i \in I$. It follows from the cross relations \eqref{eq:cross-Hmin} that in this case 
\[
\pi^\vee_{0,0} (p(\uB^\vee))=0.
\]
Now the second statement of the theorem follows from the relation \eqref{eq:piPpi}.
\end{proof}
%%%%%%%%%%%%%%%%%%%%%%%%%%%%%%%%%%%%%%%%%%%%%%%%%%

%%%%%%%%%%%%%%%%%%%%%%%%%%%%%%%%%%%%%%%%%%%%%%%%%%%%%%%%%%%%%%%%%%%%%%%%%%%%%%%%%%%%
\section{Examples of coideal subalgebras}\label{sec:Examples}
%%%%%%%%%%%%%%%%%%%%%%%%%%%%%%%%%%%%%%%%%%%%%%%%%%%%%%%%%%%%%%%%%%%%%%%%%%%%%%%%%%%%
We now consider various classes of pre-Nichols algebras $U^+$ which fall into the setting of Section \ref{sec:size}. In each case, using Theorems \ref{thm:c-cond} and \ref{thm:Hmin}, we determine all parameters $\bc$ for which the map $\varphi:\gr(B_\bc)\rightarrow H_\theta\ltimes U^-$ given by \eqref{eq:varphi} is an isomorphism.
It is convenient to work with non-symmetric quantum integers. Given $\xi \in \field$, set $[n]_\xi = 1 + \xi + \cdots + \xi^{n-1}$ and  
\begin{align*}
[n]_q! &= [n]_\xi \ldots [1]_\xi,& [2n -1]_\xi !! &= [2n-1]_\xi [2n-3]_\xi \ldots [1]_\xi
\end{align*}
for $n \in \N$, and
\[
\begin{pmatrix}
n
\\
k
\end{pmatrix}_\xi =
\frac{[n]_\xi!}{[k]_\xi![n-k]_\xi!}
\] 
for $0 \leq k \leq n$. Note that the $\xi$-binomial coefficient is a polynomial in $\Z[\xi]$ and 
therefore defined even for roots of unity.
%%%%%%%%%%%%%%%%%%%%%%%%%%%%%%%%%%%%%%%%%%%%%%%%%%%%%%%%%%
\subsection{Quantized universal enveloping algebras and nonrestricted specializations}
\label{sec:large}
%%%%%%%%%%%%%%%%%%%%%%%%%%%%%%%%%%%%%%%%%%%%%%%%%%%%%%%%%%
Let $\gfrak$ be a symmetrizable Kac--Moody algebra with (generalized) Cartan matrix $(a_{ij})_{i,j \in I}$ 
where $I =\{1, \ldots, n\}$. Denote by 
$\{d_i \mid i \in I\}$ a set of relatively prime positive integers such that the matrix $(d_i a_{ij})$ is symmetric.
Let $\gfrak':= [\gfrak,\gfrak]$ be the derived subalgebra of $\gfrak$. 
Fix $\zeta \in \field^\times$, $\zeta \neq \pm 1$. Denote by $U_\zeta(\gfrak')$ the $\field$-algebra
with generators $E_i, F_i, K_i^{\pm 1}$, $i \in I$ and the following relations for $i, j \in I$:
\begin{align}
&K_i K_j = K_j K_i, \quad K_i E_j = \zeta^{d_i a_{ij}} E_j K_i, \quad K_i F_j = \zeta^{-d_i a_{ij}} F_j K_i,\nonumber 
\\
&E_i F_j - F_j E_i = \delta_{ij} (K_i - K_i^{-1}), \label{eq:uqg-rels}
\\
&p_{ij}(E_i, E_j) = p_{ij} (F_i, F_j) = 0, \; \; i \neq j,\nonumber
\end{align}
where $p_{ij}(x,y)$ are the noncommutative polynomials in two variables given by
\[
p_{ij}(x,y) = \sum_{k=0}^{1-a_{ij}} (-1)^k \zeta^{-d_i k (1- a_{ij} -k)}
\begin{pmatrix}
1 - a_{ij}
\\
k
\end{pmatrix}_{\zeta^{2 d_i}}
x^{1- a_{ij} -k} y x^k.
\]
In the case when $\zeta$ is not a root of unity, $U_\zeta(\gfrak')$ is the quantized universal enveloping algebra of $\gfrak'$ 
for the deformation parameter $\zeta$. If $\zeta$ is a root of unity, then $U_\zeta(\gfrak')$ is the big quantum group of $\gfrak'$ at $\zeta$, 
defined and studied by De Concini, Kac and Procesi \cite{a-DKP}. In either case
$U_\zeta(\gfrak')$ is a Hopf algebra with coproduct given by
\[
\Delta(K_i) = K_i \ot K_i, \quad \Delta(E_i) = E_i \ot 1 + K_i \ot E_i, 
\quad \Delta(F_i) = F_i \ot K_i^{-1} + 1 \ot F_i
\]
for $i \in I$. Denote by $U^\pm$ the unital $\field$-subalgebras of $U_\zeta(\gfrak')$
generated by $\{E_i \mid i \in I \}$ and $\{F_i \mid i \in I \}$, respectively.
Set $H = \field [K_i^{\pm 1} \mid i \in I ]$. Consider the symmetric bicharacter
\[
\chi : \Z^n \times \Z^n \to \field^\times \quad \mbox{defined by} \quad
\chi(\alpha_i, \alpha_j) = \zeta^{d_i a_{ij}}.
\]
If $\zeta \in \field^\times$ is not a root of unity, then $U^+$ is isomorphic to the Nichols algebra of the 
Yetter--Drinfeld module $V(\chi)$. 
If $\zeta \in \field^\times$ is a root of unity and $\gfrak$ is finite dimensional (and $\zeta^3 \neq 1$ if $\gfrak$ is of type $G_2$), then $U^+$ 
is isomorphic to the distinguished pre-Nichols algebra of $V(\chi)$ defined by Angiono \cite[Definition 1]{a-Ang16}.
For all $\zeta \in \field^\times\setminus \{\pm 1\}$ and symmetrizable Kac--Moody algebras $\gfrak$, the algebra
$U^+$ is a pre-Nichols algebra of $V(\chi)$ and $U_\zeta(\gfrak') \cong U(\chi)$ is the Drinfeld double of $U^+$ in the sense of Remark \ref{rem:DD}.
Thus the constructions from the previous section are applicable to $U_\zeta(\gfrak')$. 

Let $\tau : I \to I$ be a diagram automorphism, that is, it satisfies $a_{\tau(i) \tau(j)} = a_{ij}$ for all $i, j \in I$. Given $\bc=(c_1,\dots c_n)\in \field^n$, consider the coideal subalgebra $B_\bc$ of $U_\zeta(\gfrak')$
generated by the elements
\[
     B_i=F_i + c_i E_{\tau(i)} K_i^{-1} = F_i + c_i \Etil_{\tau(i)} (K_{\tau(i)} K_i^{-1}), 
     \quad K_i K^{-1}_{\tau(i)} \qquad \mbox{for all $i\in I$.}
\]
In the case when $\zeta$ is not a root of unity, the following result is contained in \cite[Lemma 5.4, Proposition 5.16 and Theorem 7.3]{a-Kolb14}, see also \cite[Section 7]{MSRI-Letzter} for a similar discussion for $\gfrak$ of finite type.
%%%%%%%%%%%%%%%%%%%%%%%%%%%%%%%%%%%%%%%%%%%%%%%%%%%%%5
\begin{prop} \label{prop:non-restrict}
  Let $\gfrak$ be a symmetrizable Kac--Moody algebra, $\zeta \in \field^\times\setminus \{\pm 1\}$, and let $\tau : I \to I$ be a diagram automorphism.
\begin{enumerate}
\item[(i)] If $a_{ij} \neq 0$ or $\tau(i) \neq j$, then $\pi_{0,0}\circ P_{-\lambda}(p_{ij}(B_i,B_j))=0$ for $\lambda = (1- a_{ij}) \alpha_i + \alpha_j$.
If $a_{ij} = 0$ and $\tau(i) = j$, then 
\begin{align}\label{eq:Palphalph}
P_{- \alpha_i - \alpha_j }\circ \pi_{0,0}(p_{ij}(B_i,B_j))= (c_j - c_i) K_i^{-1} K_j^{-1}.
\end{align}
\item[(ii)] For the coideal subalgebra $B_\bc$ of $U_\zeta(\gfrak')$ the map  
$\varphi: \gr(B_\bc)\rightarrow H_\theta \ltimes U^-$ is an algebra isomorphism if and only if $c_i = c_{\tau(i)}$ for all $i\in I$ with  $a_{i \tau(i)}=0$. 
\end{enumerate}
\end{prop}
%%%%%%%%%%%%%%%%%%%%%%%%%%%%%%%%%%%%%%%%%%%%%%%%%
\begin{proof} (i) We work in the corresponding negative Heisenberg double, which we denote by $\mathrm{Heis}_\zeta(\gfrak')\spcheck$,
and apply Theorem \ref{thm:Hmin} to get the statement in $U_\zeta(\gfrak')$.

Let $i \neq j \in I$. If $a_{ij} \neq 0$ or $\tau(i) \neq j$, then $\lambda = (1- a_{ij}) \alpha_i + \alpha_j$ satisfies \eqref{no-good-deg}, 
and by the second part of Theorem \ref{thm:Hmin} we have 
\[
\pi_{0,0}\circ P_{-\lambda}(p_{ij}(B_i,B_j))=0
\]
in this case.

Now assume that $a_{ij} = 0$ and $\tau(i) = j$. Then in the notation of Section \ref{sec:Hmin} we have
\begin{align*}
&\pi_{0,0}^\vee(p_{ij}(B_i^\vee,B_j^\vee)) = \pi_{0,0}^\vee\Big( (F_i + c_i \Etil_j (K_i K_j^{-1})^{-1})
(F_j + c_j  \Etil_i (K_j K_i^{-1})^{-1}) 
\\
&- (F_j + c_j  \Etil_i (K_j K_i^{-1})^{-1}) (F_i + c_i  \Etil_j (K_i K_j^{-1})^{-1}) \Big) = (c_j - c_i) K_i^{-1} K_j^{-1}
\end{align*}
in  $\mathrm{Heis}_\zeta(\gfrak')\spcheck$. Hence Theorem \ref{thm:Hmin} implies \eqref{eq:Palphalph}.
Part (ii) follows from the first part and Theorem \ref{thm:c-cond}.
\end{proof} 
%%%%%%%%%%%%%%%%%%%%%%%%%%%%%%%%%%%%%%%%%%%%%%%%%%%%%%%%%%
\subsection{The small quantum group $\mathfrak{u}_\zeta(\slfrak_3)$}\label{sec:small}
%%%%%%%%%%%%%%%%%%%%%%%%%%%%%%%%%%%%%%%%%%%%%%%%%%%%%%%%%%
Consider the Nichols algebra of type $A_2$ at a root of unity. For this we fix an integer $N >2$ and set
\begin{equation}
\label{MN}
M:= \frac{N}{\gcd(N,2)} \cdot
\end{equation}
Let $\zeta$ be a primitive $N$-th root of unity and $\chi : \Z^2 \times \Z^2 \to \field^\times$ be the 
symmetric bicharacter defined by 
\begin{align*}
  q_{11}=q_{22}=\zeta^2, \qquad q_{12}=q_{21}=\zeta^{-1}.
\end{align*}
The Nichols algebra $B(V^+(\chi))$ is an algebra in ${}^{H}_{H} \cYD$ with braiding $c$, and it is generated by elements $x_1,x_2$. Recall that the braided commutator is defined by $[x,y]_c=\mu\circ(\mathrm{id}-c)(x\ot y)$ for all $x,y\in B(V^+(\chi))$ where $\mu$ denotes multiplication. Set $x_{12}=[x_1,x_2]_c=x_1x_2-\zeta^{-1}x_2x_1$. With this notation defining  relations for $B(V^+(\chi))$ are given by \cite[Equation (4.5)]{a-AAReview}
\begin{align*}
  x_1^M=x_2^M=x_{12}^M=0, \qquad [x_1,[x_1,x_2]_c]_c=0=[x_2,[x_2,x_1]_c]_c.
\end{align*}
Denote by $u_\zeta(\slfrak_3)$ the Drinfeld double of $B(V^+(\chi))$. Its factor by the ideal generated by $K_i^N-1$ for $i=1,2$ is isomorphic to the small quantum group $\mathfrak{u}_\zeta(\slfrak_3)$ of type $A_2$. 

Consider the diagram automorphism $\tau$ given by $\tau(1)=2$, $\tau(2)=1$. It follows from Theorem \ref{thm:Hmin} that the only relation which gives a condition on the parameters $c_1, c_2$ of the coideal subalgebra is the relation $x_{12}^M=0$ because the other four relations are homogeneous of a degree $\lambda$ which satisfies \eqref{no-good-deg}.
This relation gives a condition for any integer $N$ (even or odd!). Recall that
\begin{equation}
  \label{B12}
  \begin{aligned}
  B_1&=  F_1+c_1 E_2 K_1^{-1} = F_1+c_1 \Etil_2 (K_1 K_2^{-1})^{-1},
  \\
  B_2&=F_2+c_2 E_1 K_2^{-1} = F_2+c_2 \Etil_1 (K_2 K_1^{-1})^{-1}.
  \end{aligned}
\end{equation}
We define a non-commutative polynomial $p(x_1,x_2)$ by
\begin{align*}
  p(x_1,x_2)=(x_1 x_2 - \zeta^{-1} x_2 x_1)^M.
\end{align*}
Note that  $p(x_1,x_2)$ is homogeneous of degree $\lambda=(M,M)\in \Z^2$.
%%%%%%%%%%%%%%%%%%%%%%%%%%%%%%%%%%%%%%%%%%%
\begin{prop} 
\label{prop:A2} Let $N\in \N$ with $N\ge 2$ and let $\zeta \in \field$ be a primitive $N$-th root of unity. Let $M$ be given by \eqref{MN}.
\begin{enumerate}
\item[(i)] In the quantum double $u_\zeta(\slfrak_3)$ of the Nichols algebra of type $A_2$ corresponding 
to the root of unity $\zeta$, we have
\[
\big[ \pi_{0,0}\circ P_{-\lambda}(p(B_1,B_2)) \big] K_\lambda = 
\begin{cases}
c_2^M + c_1^M, & \mbox{if} \; \; N \equiv 2 \mod 4
\\
c_2^M - c_1^M, &\mbox{otherwise}.
\end{cases}
\]

\item[(ii)] For the coideal subalgebra $B_\bc$ of $u_\zeta(\slfrak_3)$ the map  
$\varphi: \gr(B_\bc)\rightarrow H_\theta \ltimes U^-$ is an algebra isomorphism
if and only if
\[
c_1 = \upsilon c_2
\]
where $\upsilon \in \field$ is such that $\upsilon^M =-1$ if $N \equiv 2 \mod 4$, 
and $\upsilon^M=1$, otherwise.
\end{enumerate}
\end{prop}
%%%%%%%%%%%%%%%%%%%%%%%%%%%%%%%%%%%%%%%%%%%%%%%%%%%%%%%%
For example, when $N=4$ we have $\zeta = \sqrt{-1}$. Then  
\[
\pi_{0,0}\circ P_{-\lambda}(p(B_1,B_2)) = ( c_2^2 - c_1^2) K_\lambda 
\]
and $B_\bc\subset u_{\sqrt{-1}}(\slfrak_3)$ is of the right size if and only if $c_2 = \pm c_1$.

In the proof of the proposition we will use the Al-Salam-Carlitz I discrete orthogonal polynomials $U_n^{(a)}(x; q)$, see \cite{a-AlCa65} and \cite[pp. 534-537]{b-KoLeSw00}. 
They have been used in the related setting of the $q$-harmonic oscillator in \cite{a-ASu93}. 
From an algebraic point of view $U_n^{(a)}(x; q) \in \Z[a,q,x]$ is given by 
\[
U_n^{(a)} (x;q) = \sum_{k=0}^n 
\begin{pmatrix} 
n \\k
\end{pmatrix}_q (-a)^k q^{k(k-1)/2} (x-1) \ldots (x - q^{n-k-1}). 
\]
The Al-Salam-Carlitz I polynomials satisfy the backward shift recursion 
\begin{align}\label{eq:ASC-backward}
- q^{-n+1}x U_n^{(a)}(x; q) = a  U_{n-1}^{(a)}(x; q) - (x-1)(x-a) U_{n-1}^{(a)}(q^{-1}x; q)
\end{align}
for all $n>0$, see \cite[Eq. (14.24.8)]{b-KoLeSw00}.
Consider the $q$-derivative ${\mathscr{D}}_q f (x) = (f(qx) - f(x))/((q-1)x)$ for $f(x) \in \field[x]$. 
The recursion \eqref{eq:ASC-backward} implies the following lemma. The proof is left to the reader.
%%%%%%%%%%%%%%%%%%%%%%%%%%%%%%%%%%%%%%%%%%%%%%%%%%%%%%%
\begin{lem} \label{lem:A2}
  Consider the polynomials $p_n(x;t,q)\in \Z[t,q^{\pm 1}, x]$ defined recursively by
  \begin{align}\label{p-recursion}
    p_0(x;t,q)&=1, & p_n(x;t,q) = (x + t q^{-2n} {\mathscr{D}}_q + q^{-n}) p_{n-1}(x; t, q), \; \; \forall n >0.
  \end{align}
  Consider $\Z[t,q^{\pm 1},x]$ as a subring of $\Z[t_1^{\pm 1}, q^{\pm 1},x]$ via the map $t\mapsto (q-1)t_1(t_1+1)$. Then in $\Z[t_1^{\pm 1}, q^{\pm 1},x]$ we have
  \begin{align*}
p_n(x; t, q)  = t_1^n q^{- n^2} U_n^{(-t_1^{-1}-1)} (q ^n t_1^{-1} x; q) 
  \end{align*}
for all $n\in\N$. 
\end{lem}
%%%%%%%%%%%%%%%%%%%%%%%%%%%%%%%%%%%%%%%%%%%
Define the quantum Weyl algebra $A_1^q$ as the $\field[q^{\pm 1}]$-algebra with generators  $X, Y, Z_1, Z_2$
and relations
\begin{align*}
YX - q^{-1} XY = Z_2, \; \;  Z_i Y = q^i Y Z_i, \; \; Z_i X = q^{-i} X Z_i,\; \;  Z_1 Z_2 = Z_2 Z_1.
\end{align*}
Inside the localization $A_1^q[Z_1^{-1}]$ we have a copy of the first quantized Weyl algebra ${\mathcal{A}}_q^1$, which is the $\field[q^{\pm 1}]$-algebra with generators 
$y = Y Z_1^{-1}$, $x= X Z_1^{-1}$, $z = qZ_2 Z_1^{-2}$ and relations
\[
yx - q xy = z, \; \; z x = x z, \; \; z y = y z.
\]
The algebra ${\mathcal{A}}_q^1$ acts on $\field[q^{\pm1}, t, x]$ by $x \mapsto (x \cdot)$, $y \mapsto t {\mathscr{D}}_q$, $z \mapsto (t \cdot)$. 
Iterating the recursion \eqref{p-recursion} gives that the polynomials $p_n(x; z,q) \in \Z[q^{\pm 1}, t, x]$ satisfy
\[
(x + q^{- 2n} y + q^{-n}) \ldots (x + q^{-2} y + q^{-1})  \cdot 1 = p_n(x; z, q). 
\]
Since $\field[q^{\pm1}, t, x] \cong {\mathcal{A}}_q^1/ ({\mathcal{A}}_q^1y)$ as left ${\mathcal{A}}_q^1$-modules, 
we have 
\begin{align}
\label{power-n}
(X+Y+Z_1)^n &= q^{n(n+1)/2} Z_1^n (x + q^{- 2n} y + q^{-n}) \ldots (x + q^{-2} y + q^{-1}) 
\\
\nonumber
&\equiv q^{n (n+1)/2} Z_1^n p_n(x; z, q) \mod A_1^q Y.
\end{align}

For $\xi \in \field^\times$, let $A^\xi_1 = A^q_1/(q-\xi) A^q_1$ denote the specialization of $A_1^q$ at $\xi$. 
\begin{proof}[Proof of Proposition \ref{prop:A2}]

(i) We work in the negative Heisenberg double ${\mathrm{Heis}}_\zeta(\slfrak_3)\spcheck$ 
corresponding to $u_\zeta(\slfrak_3)$ and apply Theorem \ref{thm:Hmin} to get the statement in $u_\zeta(\slfrak_3)$.
Set 
\[
\overline{E}_1:= \Etil_1 (K_1 K_2^{-1}) = E_1 K_2^{-1}, \quad \overline{E}_2:=  \Etil_2 (K_2 K_1^{-1}) = E_2 K_1^{-1}, \quad K_{12} = K_1 K_2,
\]
so $B_1^\vee = F_1 + c_1 \overline{E}_2$ and $B_2^\vee = F_2 + c_2 \overline{E}_1$. Denote also
\[
F_{12} = F_1 F_2 - \zeta^{-1} F_2 F_1, \quad 
\overline{E}_{21} = \overline{E}_2 \overline{E}_1 - \zeta^{-1} \overline{E}_1 \overline{E}_2.
\]
One verifies that 
\[
F_{12} \overline{E}_j = \zeta^{-1} \overline{E}_j F_{12} + \delta_{j2} F_1 K_{12}^{-1} \qquad \mbox{for $j=1,2$}
\]
from which it follows that 
\[
F_{12} \overline{E}_{21} = \zeta^{-2} \overline{E}_{21} F_{12} + (1 - \zeta^{-2})^2 \overline{E}_1 F_1 K_{12}^{-1} + \zeta^{-1} (1 - \zeta^{-2}) K_{12}^{-2}.
\]
In a similar fashion one shows that 
\[
F_{12} (\overline{E}_1 F_1) = \zeta^{-2} (\overline{E}_1 F_1) F_{12}, \quad 
 \overline{E}_{21} (\overline{E}_1 F_1) = \zeta^2 (\overline{E}_1 F_1)  \overline{E}_{21}.
\]
From the last three identities one derives that we have a homomorphism $\rho : A_1^{\zeta^2} \to {\mathrm{Heis}}_\zeta(\slfrak_3)\spcheck$
given by 
\begin{align*}
\rho(X) &= c_1 c_2 \overline{E}_{21}, \; \; \rho(Y)= F_{12}, 
\\
\rho(Z_1) &= (\zeta - \zeta^{-1}) c_2 \overline{E}_1 F_1 +(c_2 - \zeta^{-1} c_1) K_{12}^{-1},
\\
\rho(Z_2)&=  (1 - \zeta^{-2})^2 c_1 c_2 \overline{E}_1 F_1 K_{12}^{-1} +  \zeta^{-1} (1 - \zeta^{-2}) c_1 c_2 K_{12}^{-2}.
\end{align*}
Equation \eqref{power-n} implies that 
\begin{align}
&p(B_1^\vee,B_2^\vee) = \rho(X+Y+Z_1)^M
\label{first-step}
\\
&\equiv \zeta^{M(M+1)} \rho(Z_1)^M p_M(\rho(X Z_1^{-1}); \zeta^2 \rho(Z_2 Z_1^{-1}), \zeta^2) 
\mod {\mathrm{Heis}}_\zeta(\slfrak_3)\spcheck F_{12}.
\nonumber
\end{align}
There are no terms with $Z_1$-denominators in the right hand side because 
$\deg p_n(x; t, q) =n$ when $p_n(x; t, q)$ is considered as a polynomial in $x$ and $t$ and the degree is computed
with respect to the grading $\deg x =1$, $\deg t =2$. 
This follows from the recursion \eqref{p-recursion} and the fact that the operator ${\mathscr{D}}_{\zeta^2}$
lowers the degree by $1$. 

Every pair of the six terms of $\rho(X)$, $\rho(Z_1)$ and $\rho(Z_2)$ quasi-commute. Therefore
\[
\pi_{0,0}^\vee( \rho(X^i Z_1^j Z_2^k)) = 
\delta_{i,0} \big( (c_2 - \zeta^{-1} c_1) K_{12}^{-1} \big)^j \big( \zeta^{-1} (1 - \zeta^{-2}) c_1 c_2 K_{12}^{-2} \big)^k
\]
for all $i, j, k \in \N$. 
Combining this with \eqref{first-step} gives that
\begin{equation}
\pi_{0,0}^\vee (p(B_1^\vee,B_2^\vee))   =  \zeta^{M (M+1)} (c_2 - \zeta^{-1} c_1)^M p_M(0; t, \zeta^2) K_{12}^{-M} \label{eq:pi00veep}
\end{equation}
where
\[
t = \frac{(\zeta - \zeta^{-1}) c_1 c_2}{(c_2 - \zeta^{-1} c_1)^2} \cdot
\]
As $t=(\zeta^2-1)t_1(t_1+1)$ for $t_1 = -c_2 /(c_2 - \zeta^{-1} c_1)$
%,  \quad
%t_2 = \frac{ \zeta^{-1} c_1}{(c_2 - \zeta^{-1} c_1)} 
we can apply Lemma \ref{lem:A2} and obtain
\begin{align}\label{eq:pU}
  p_M(0;t,\zeta^2)=t_1^M\zeta^{-2M^2} U_M^{(-\zeta^{-1}c_1/c_2)}(0;\zeta^2).
\end{align}
Since $\zeta^2$ is a primitive $M$-th root of unity, 
$\small{\begin{pmatrix} M \\ k \end{pmatrix}_{\zeta^2}} =0$ for all $0<k <M$. 
For the corresponding Al-Salam-Carlitz I polynomials we hence have 
\begin{align}\label{eq:U0}
U_M^{(a)} (0; \zeta^2) = (-1)^M \zeta^{M(M-1)} (1+ a^M).
\end{align}
Inserting \eqref{eq:pU} and \eqref{eq:U0} into \eqref{eq:pi00veep} we obtain
\begin{align*}
\pi_{0,0}^\vee (p(B_1^\vee,B_2^\vee))  K_{12}^M &= \zeta^{M (M+1)} (c_2 - \zeta^{-1} c_1)^M t_1^M \zeta^{- 2 M^2} U_M^{(- \zeta^{-1} c_1/c_2)}(0; \zeta^2) 
\\
&=  c_2^M +  (-1)^M \zeta^M c_1^M, 
\end{align*}
which proves part (i).
Part (ii) follows directly from the first part.
\end{proof}
%%%%%%%%%%%%%%%%%%%%%%%%%%%%%%%%%%%%%%%%%%%%%%%%%%%%%%%%%%
\subsection{The quantum supergroups of type $\slfrak(m|k)$}\label{sec:super}
%%%%%%%%%%%%%%%%%%%%%%%%%%%%%%%%%%%%%%%%%%%%%%%%%%%%%%%%%%
Let $m, k$ be positive integers such that $(m,k) \neq (1,1)$. Denote $n= m+k -1$. The (super) Dynkin diagrams of the Lie superalgebra 
$\slfrak(m|k)$ associated to different choices of Borel subalgebras 
are the Dynkin diagrams of type $A_n$ where each vertex is denoted in two different 
ways: by $\bigotimes$ if the vertex is odd and by $\bigcirc$ if it is even, cf. \cite[Sections 2.5.5-2.5.6]{a-Kac77}. 
(There is a dependence between the number of odd vertices, $m$ and $k$ which will not play a role below.)
All odd simple roots are necessarily isotropic. We label the vertices in an increasing way from left to right by 
the elements of $I= \{1, \ldots, n\}$. Define the parity function $p : I \to \{0,1\}$ by letting $p(i) =0$ for even 
vertices and $p(i) =1$ for odd vertices. The corresponding (super) Cartan matrix is given by
\[
a_{ii} = 
\begin{cases} 
2, & p(i)=0
\\
0, & p(i)=-1
\end{cases}
\quad \mbox{and} \quad \mbox{for $i\neq j$} 
\quad
a_{ij} = 
\begin{cases}
-1, &p(i)=0 \; \mbox{and} \; j = i \pm 1,
\\
\pm 1, & p(i)=1 \; \mbox{and} \; \; j = i \pm 1, 
\\
0, & |i-j| >1,
\end{cases}
\]
cf. \cite[Section 5.1.5]{a-AAReview}. 

Fix $\zeta \in \field^\times$, $\zeta \neq \pm 1$ and consider the bicharacter $\chi : \Z^n \times \Z^n \to \field^\times$ 
given by
\[
\chi(\alpha_i, \alpha_j) = (-1)^{p(i) p(j)} \zeta^{a_{ij}}
\]
for $i,j \in I$. Denote by $U^+$ the $\field$-algebra with generators $x_i$, $i \in I$ and relations
\begin{align*}
&[x_i, [x_i, x_{i\pm1}]_c]_c = 0, \; \; p(i)=0, \qquad &[x_i, x_j]_c =0, \; \; i < j-i,
\\
&[[x_{i-1}, [x_i, x_{i+1}]_c]_c, x_i]_c =0, \; \; p(i)=1, \quad &x_i^2 =0, \; \; p(i)=1.
\end{align*}
In the case $\zeta = \sqrt{-1}$, following \cite[Definition 1]{a-Ang16},  
we also add the relations 
\[
[x_i, [x_i, x_{i\pm1}]_c]_c = 0 \quad \mbox{for} \quad p(i)=1.
\]
If $\zeta$ is not a root of unity, then $U^+$ is isomorphic to the Nichols algebra of $V^+(\chi)$, see \cite[Eq. (5.10)]{a-AAReview}.
If $\zeta$ is a root of unity, then $U^+$ is isomorphic to the distinguished pre-Nichols algebra of $V^+(\chi)$, 
see  \cite[Definition 1]{a-Ang16} and \cite[Eq. (5.10)]{a-AAReview}. 

Denote the set of odd vertices $J = \{ i \in I \mid p(i) = 1\}$. Denote the Drinfeld double of $U^+$ 
by $U(\chi)$ and form the smash product
\[
U_\zeta(\slfrak(m|k))_J = U(\chi) \rtimes \field \Z_2 
\]
where the generator $\sigma$ of $\Z_2$ acts on $U(\chi)$ by 
\[
\sigma(E_i) = (-1)^{p(i)} E_i, \quad  
\sigma(F_i) = (-1)^{p(i)} F_i, \quad
\sigma(K_i^{\pm 1}) = K_i^{\pm 1}
\]
for all $i \in I$. Our generators differ from those in \cite{a-Yam94,a-BKK00}. In terms of the generators $e_i, f_i, t_i$ 
of \cite{a-BKK00}, our generators are given by
\[
E_i = \sigma^{p(i)} e_i, \quad F_i = f_i, \quad K_i^{\pm1} = \sigma^{p(i)} t_i^{\pm1}
\]
for all $i \in I$. The coproduct convention of \cite{a-BKK00} is also slightly different from ours. 
By \cite[Theorem 6.11]{a-Heck10}, for different choices of $J$, the Hopf algebras $U_\zeta(\slfrak(m|k))_J$ 
are isomorphic to each other as algebras with isomorphisms provided by generalized Lusztig isomorphisms (these isomorphisms 
descend from the actual Drinfeld double to its quotient $U(\chi)$). However, the 
Lusztig isomorphisms are not Hopf algebra isomorphisms, and as a consequence,  
$U_\zeta(\slfrak(m|k))_J$ are not isomorphic to each other as Hopf algebras for different choices of $J$. The Hopf algebra 
$U_\zeta({\mathfrak{gl}}(m|k))$ in \cite{a-BKK00} is our $U_\zeta(\slfrak(m|k))_{ \{ m \} }$ up a slightly different convention for the coproduct.

If $\zeta \in \field^\times$ is not a root of unity, then $U_\zeta(\slfrak(m|k))_J$ exhaust all 
different quantum supergroups of type $\slfrak(m|k)$. If $\zeta \in \field^\times$ is a root of unity, then 
$U_\zeta(\slfrak(m|k))_J$ are the corresponding nonrestricted specializations at roots of unity.

Let $\tau : I \to I $ be the identity or the involution $\tau(i) = n-i +1$ (for $i \in n-i$) in the case when the vertices $i$ and $n-i+1$ have the same parity for all $i \in I$.   
For $\bc=(c_1,\dots, c_n)\in \field^n$, let $B_\bc$ to be the coideal subalgebra of $U_\zeta(\slfrak(m|k))_J$ 
generated by $H_\theta$ and the elements
\[
     B_i=F_i + c_i E_{\tau(i)} K_i^{-1} \qquad \mbox{for all $i\in I$.}
\]
\begin{prop}
\label{prop:super}
For all choices of odd roots $J \subseteq I$ and $\zeta \in \field^\times$, $\zeta \neq \pm1$, for the coideal 
subalgebra $B_\bc$ of the the quantum linear supergroup  $U_\zeta(\slfrak(m|k))_J$, the map  
$\varphi: \gr(B_\bc)\rightarrow H_\theta \ltimes U^-$ is an algebra isomorphism if and only if
\begin{equation}
\label{c1-super}
c_{\tau(i)} = c_i \; \; \mbox{for $i \in I$ such that $|\tau(i)-i| >1$}
\end{equation}
and
\begin{equation}
\label{c2-super}
c_i =0  \; \; \mbox{for all odd vertices $i$ fixed by $\tau$}.
\end{equation}
\end{prop}
More precisely, the conditions on $\bc$ in the proposition are as follows:
\begin{enumerate}
\item If $\tau = \id$, then $c_i =0$ for all odd vertices $i \in J$ (there is only one such vertex for the standard choice of simple roots 
corresponding to $J = \{m \}$); 
\item If $\tau$ is the flip $\tau(i) = n-i +1$, then $c_i = c_{n-i+1}$ for $i \in \{1, \ldots, \lceil n/2 \rceil -1 \}$ and 
$c_{(n+1)/2}=0 $ if $n$ is odd and $(n+1)/2$ is an odd vertex.
\end{enumerate}
For the proof of Proposition \ref{prop:super} we will need the following lemma.
\begin{lem} 
\label{lem:no-cond}
Let $p(x_1, \ldots, x_n)$ be a homogeneous noncommutative polynomial in $x_1, \ldots, x_n$ 
of degree $\sum_j m_j \alpha_j$, and $i \in I$ be such that $m_i >0$. For all 
bicharacters $\chi : \Z^n \times \Z^n \to \field^\times$, $\tau : I \to I$ and 
$\bc = (c_1, \ldots, c_n) \in \field^n$ such that $\tau(i) =i$ and $c_i =0$, we have
\[
\pi_{0,0} (p (B_1, \ldots, B_n)) =0
\]
in $U(\chi)$.
\end{lem}
\begin{proof} After distributing the parenthesis in $p (B_1, \ldots, B_n)$, we 
get an expression for $p (B_1, \ldots, B_n)$ as a sum of monomials $f \in U(\chi)$ in $F_j$, $E_{\tau(j)} K_j^{-1}$ with $j \in I$. 
If such a monomial $f$ contains the factor $E_i^{-1} K_i$, then its coefficient equals 0 because $c_i =0$. 
For all other monomials $\pi_{0,0}(f) =0$ which is obtained by directly commuting the factors of the monomials.
\end{proof}
\begin{proof}[Proof of Proposition \ref{prop:super}] We apply Theorem \ref{thm:c-cond} and explicitly compute
condition \eqref{assume:parameters} in Section \ref{sec:prop-c}. 
As in the proof of Proposition \ref{prop:non-restrict}(i), the first set of relations of $U^+$ and the extra relations in the case 
$\zeta = \sqrt{-1}$ give no condition of $\bc$, while the 
second set of relations of $U^+$ gives condition \eqref{c1-super}. If $\tau(i) =i$ for some $i \in I$, 
then in the negative Heisenberg double $\Hmin$ we have 
\[
\pi_{0,0}^\vee \big( (B_i^\vee)^2 \big) = \pi_{0,0}^\vee \big( (F_i + c_i \Etil_i)^2 \big)  = c_i K_i^{-2}.
\] 
It follows from Theorem \ref{thm:Hmin}(ii) that the fourth set of relations of $U^+$ gives condition 
\eqref{c2-super} on $\bc$. Finally, we consider the third set of relations of $U^+$. If the third relation of $U^+$ 
for a given odd vertex $i$ gives a condition on $\bc$, then by Theorem \ref{thm:Hmin}(ii), 
\[
\tau (\alpha_{i-1} + 2 \alpha_i + \alpha_{i+1}) = - (\alpha_{i-1} + 2 \alpha_i + \alpha_{i+1}).
\]
This implies that $\tau(i) = i$. If \eqref{c2-super} is satisfied, then we also have $c_i =0$.
Now it follows that in the presence of condition \eqref{c2-super}, the third set of relations of $U^+$ do not 
give any new condition on $\bc$ because of Lemma \ref{lem:no-cond}.
\end{proof}
The techniques of this proof can be used to classify the coideal subalgebras $B_\bc$ of the quantized enveloping algebras
of all finite dimensional and affine contragredient Lie superalgebras $\gfrak$ with the property that $\varphi: \gr(B_\bc)\rightarrow H_\theta \ltimes U^-$ 
is an algebra isomorphism. This is more technical and will appear in a subsequent publication.
%%%%%%%%%%%%%%%%%%%%%%%%%%%%%%%%%%%%%%%%%%%%%%%%%%%%%%%%%%  
 \subsection{The Drinfeld double of the distinguished pre-Nichols algebra of type $\ufofrak(8)$}\label{sec:ufo}
%%%%%%%%%%%%%%%%%%%%%%%%%%%%%%%%%%%%%%%%%%%%%%%%%%%%%%%%%%
Let $\zeta$ be a primitive 12-th root of unity and $\zeta^{1/2}$ be a primitive 24-th root of unity that squares to $\zeta$.
Consider the symmetric bicharachter $\chi$ given by
\begin{align*}
  q_{11}=q_{22}=-\zeta^2, \qquad q_{12}=q_{21}=\zeta^{1/2}.
\end{align*}
It is associated to the first of the three generalized Dynkin diagrams on row 8 of Table 1 in \cite{a-Heck09}. 
The corresponding Nichols algebra is one of three such algebras of type $\ufofrak(8)$. It is one of the non-Cartan type 
examples that appeared in Heckenberger's classification of arithmetic root systems \cite{a-Heck09}.

The generalized Cartan matrix of the bicharacter $\chi$ is 
\[
C^\chi = 
\begin{pmatrix}
2 &-2
\\
-2 & 2
\end{pmatrix}
\cdot
\]
The generalized root system of $\chi$ is finite and has three Cartan matrices 
corresponding to the generalized Dynkin diagrams on row 8 of Table 1 in \cite{a-Heck09}.
We refer the reader to \cite[Sections 3 and 5]{a-Heck06}, 
\cite[Section 2.7]{a-AAReview} and \cite[Section 4]{a-HeYa08} for details
on this topic and Weyl groupoids.

The relations of the Nichols algebra of $\chi$ are given in \cite[Section 10.8.6]{a-AAReview}. 
Let $U^+$ denote the distinguished pre-Nichols algebra of $\chi$ 
defined by Angiono in \cite[Definition 1]{a-Ang16}
as the factor of $T(V^+(\chi))$ by removing from the Nichols ideal 
the power relations for Cartan roots and adding certain quantum Serre relations. 
There are none of the latter in this case and 
the algebra $U^+$ has two generators $x_1, x_2$ with relations
\[
x_1^3=x_2^3=0 \quad \mbox{and} \quad [x_1, x_{\alpha_1 + 2 \alpha_2}]_c = -(1 + \zeta^{-1} + \zeta^{-2}) \zeta^{1/2} x_{12}^2,
\]
the third of which is the last relation in \cite[Eq. (10.55)]{a-AAReview}. Here
\[
x_{12} = [x_1, x_2]_c \quad \mbox{and} \quad x_{\alpha_1 + 2 \alpha_2} = [x_{12}, x_2]_c 
\]
in the free algebra in $x_1$, $x_2$.

Consider the diagram automorphism $\tau(1)=2$, $\tau(2)=1$ and the coideal subalgebra generators $B_1$, $B_2$ given by 
\eqref{B12}.
\begin{prop}
\label{prop:ufo8} 
The following hold for the quantum double $U(\chi)$ of the distinguished pre-Nichols algebra of type $\ufofrak(8)$:
\begin{enumerate}
\item[(i)] For $p(x_1, x_2)=[x_1, x_{\alpha_1 + 2 \alpha_2}]_c + (1 + \zeta^{-1} + \zeta^{-2}) \zeta^{1/2} x_{12}^2$ and
$\lambda = 2 \alpha_1 + 2 \alpha_2$, 
\[
P_{-\lambda}\circ \pi_{0,0}(p(B_1,B_2)) = (1+ \zeta) \zeta^{1/2} \big(c_1^2 - 2\zeta^{-1/2}c_1c_2 + c_2^2\big) K_{-\lambda}.
\]

\item[(ii)] For the coideal subalgebra $B_\bc$ of $U(\chi)$ the map  
$\varphi: \gr(B_\bc)\rightarrow H_\theta \ltimes U^-$ is an algebra isomorphism if and only if
\[
c_1 = (1\pm \sqrt{1-\zeta})\zeta^{-1/2} c_2.
\]
\end{enumerate}
\end{prop}
\begin{proof} (i) We have
\begin{equation}
\label{p-poly}
p(x_1,x_2) = (x_1^2 x_2^2 + x_2^2 x_1^2)+ a (x_1 x_2 x_1 x_2 + x_2 x_1 x_2 x_1) + b (x_1 x_2^2 x_1 + x_2 x_1^2 x_2)
\end{equation}
in the free algebra in $x_1$, $x_2$, where
\[
a = (1 + \zeta^{-1})\zeta^{1/2}, \quad b= -(1+ \zeta^{-1} + \zeta^{-2}) \zeta.
\]
From this one directly computes $\pi_{0,0}^\vee(p(B_1^\vee,B_2^\vee))$ in the negative Heisenberg double $\mathrm{Heis}_\zeta(\chi)\spcheck$ corresponding to $U(\chi)$.
Now part (i) follows from Theorem \ref{thm:Hmin}.

(ii) It follows from the second statement in Theorem \ref{thm:Hmin} that $\pi^\vee_{0,0}((B^\vee_1)^3)= \pi_{0,0}((B^\vee_2)^3)=0$ in $\mathrm{Heis}_\zeta(\chi)\spcheck$.
Theorem \ref{thm:c-cond} implies the validity of part (ii).
\end{proof}
%%%%%%%%%%%%%%%%%%%%%%%%%%%%%%%%%%%%%%%%%%%%%%%%%%%%%%%%%%%%%%%%%%%%%%%%%%%%%%%%%%%%
\section{A twist product on partial bosonizations}\label{twist}
%%%%%%%%%%%%%%%%%%%%%%%%%%%%%%%%%%%%%%%%%%%%%%%%%%%%%%%%%%%%%%%%%
Assume that condition \eqref{assume:parameters} from Section \ref{sec:prop-c} holds. By Theorem \ref{thm:c-cond} the algebra $B_\bc$ has a filtration such that the associated graded algebra is isomorphic to the partial bosonization $H_\theta\ltimes U^-$. In the present section we use the quasi $R$-matrix for $U(\chi)$ to define a twisted algebra structure $\star$ on $H_\theta\ltimes U^-.$ We will see in Section \ref{sec:star} that $(H_\theta\ltimes U^-, \star)$ is canonically isomorphic to $B_\bc$.
%%%%%%%%%%%%%%%%%%%%%%%%%%%%%%%%%%%%%%%%%%%%%%%%%%%%%%%%%%%%%%%%%
\subsection{The quasi $R$-matrix for $U(\chi)_\mx$}\label{sec:quasiR}
%%%%%%%%%%%%%%%%%%%%%%%%%%%%%%%%%%%%%%%%%%%%%%%%%%%%%%%%%%%%%%%%%
Recall that $\cI_\mx(\chi)\subset T(V^+(\chi))$ denotes the maximal $\Z^n$-graded biideal in the braided Hopf algebra $T(V^+(\chi))$.
In the following we use the subscript ${}_\mx$ to indicate constructions involving $\cI_\mx$. In particular, we use the notation $U^+_\mx$, $U^-_\mx$ for the Nichols algebras corresponding to $\cI_\mx$ and we write $U(\chi)_\mx$ for the corresponding Drinfeld double as defined in Section \ref{sec:setting}. By \cite[Theorem 5.8]{a-Heck10} there exists a uniquely determined skew-Hopf pairing
\begin{align}\label{eq:pairing}
  \langle \,\, ,\,\rangle_\mx: (U^-_\mx \rtimes H)^\cop \ot (U^+_\mx \rtimes H) \rightarrow \field
\end{align}
such that
\begin{equation}\label{eq:pair-def}
\begin{aligned}
  \langle F_i,E_j\rangle_\mx&=\delta_{ij}, & \langle K_i, K_j\rangle_\mx&=q_{ij}^{-1},\\
  \langle F_i, K_j\rangle_\mx &=0, & \langle K_i, E_j\rangle_\mx&=0
\end{aligned}
\end{equation}
for all $i,j\in I$. Recall that by definition of a skew-Hopf pairing we have
\begin{equation}\label{eq:skew-hopf}
\begin{aligned}
  \langle y, x x'\rangle_\mx&=\langle y_{(1)},x' \rangle_\mx \langle y_{(2)},x\rangle_\mx,\\
   \langle y y', x \rangle_\mx&=\langle y,x_{(1)} \rangle_\mx \langle y', x_{(2)}\rangle_\mx
\end{aligned}
\end{equation}
for all $y,y' \in (U^-_{\mx} \rtimes H)^\cop$ and $x,x'\in U^+_{\mx} \rtimes H$.
Let
\begin{align*}
  \pi_\mx: U^-\rtimes H \rightarrow U^-_\mx\rtimes H
\end{align*}  
denote the canonical projection. By construction $\pi_\mx$ is a surjective Hopf algebra homomorphism. The pairing \eqref{eq:pairing} allows us to define a right and a left $U^+_\mx\rtimes H$ module structure on $H\ltimes U^-=(U^-\rtimes H)^\cop$ by
\begin{align}\label{eq:lactract}
  e\lact f&= \langle \pi_\mx(f_{(1)}), e\rangle_\mx f_{(2)}, &
  f\ract e&= \langle \pi_\mx(f_{(2)}), e\rangle_\mx f_{(1)}
\end{align}
for all $e\in U^+_\mx\rtimes H$, $f\in (U^-\rtimes H)^\cop$. The properties in \eqref{eq:skew-hopf} imply that $(U^-\rtimes H)^\cop$ is a right and a left $U^+_\mx\rtimes H$-module algebra.

The pairing $\langle\,\, ,\,\rangle_\mx$ respects the $\Z^n$-grading of $U^-_\mx \rtimes H$ and  $U^+_\mx \rtimes H$.  Moreover, by \cite[Theorem 5.8]{a-Heck10} the restriction of $\langle \,\, ,\,\rangle_\mx$ to $U^-_\mx\ot U^+_\mx$ is nondegenerate. This allows us to formulate the notion of a quasi $R$-matrix for $U(\chi)_\mx$ in complete analogy to \cite[Chapter 4]{b-Lusztig94}.
Let $U(\chi)_\mx\widehat{\ot} U(\chi)_\mx$ denote the completion of $U(\chi)_\mx\ot U(\chi)_\mx$ with respect to the descending sequence of subspaces
\begin{align*}
  \mathcal{H}_N = \big(U^+_\mx H \sum_{|\nu|\ge N} (U^-_\mx)_{-\mu}\big) \ot U(\chi)_\mx +
  U(\chi)_\mx \ot \big(U^-_\mx H \sum_{|\nu|\ge N} (U^+_\mx)_{\mu}\big).
\end{align*}  
The $\field$-algebra structure on  $U(\chi)_\mx\ot U(\chi)_\mx$ extends to a $\field$-algebra structure on  $U(\chi)_\mx\widehat{\ot} U(\chi)_\mx$.

  For any $\mu\in \N^n$ let $\{F_{\mu,j}\}\subset (U^-_\mx)_{-\mu}$ and $\{E_{\mu,j}\}\subset (U^+_\mx)_\mu$ be dual bases with respect to the nondegenerate pairing $\langle\,,\,\rangle_\mx$ and define $\Theta_\mu = \sum_j (-1)^{|\mu|} F_{\mu,j}\ot E_{\mu,j}$. For simplicity we usually suppress the summation and write formally
  \begin{align*}
    \Theta_\mu=(-1)^{|\mu|} F_\mu\ot E_\mu.  
  \end{align*}
Define an element $\Theta\in U(\chi)_\mx\widehat{\ot} U(\chi)_\mx$ by
\begin{align}
\label{eq:Theta}
  \Theta=\sum_{\mu\in \N^n} \Theta_\mu = \sum_{\mu\in \N^n} (-1)^{|\mu|} F_\mu\ot E_\mu.
\end{align}
For quantized enveloping algebras the element $\Theta$ coincides with the quasi $R$-matrix constructed in \cite[Chapter 4]{b-Lusztig94}. 
Analogously to \cite[Proposition 4.2.2]{b-Lusztig94} we have the following result.
%%%%%%%%%%%%%%%%%%%%%%%%%%%%%%%%%%%%%%%%%%%
\begin{lem}\label{lem:kowidTheta}
  The following relations hold
  \begin{align}
  (\kow\ot \id)(\Theta_\mu) &= (-1)^{|\mu|}\sum_{\lambda+\nu=\mu} F_\lambda \ot F_\nu K_\lambda^{-1}\ot E_\nu E_\lambda,\label{eq:kow-Theta1}\\
    (\id \ot\kow)(\Theta_\mu) &= (-1)^{|\mu|}\sum_{\lambda+\nu=\mu} F_\lambda  F_\nu \ot E_\lambda K_\nu\ot E_\nu.\label{eq:kow-Theta2}
  \end{align}  
\end{lem}
%%%%%%%%%%%%%%%%%%%%%%%%%%%%%%%%%%%%%%%%%%%
\begin{proof}
  By definition of the coproduct of $U(\chi)_\mx$ in \eqref{eq:kowEFK} we have
  \begin{align*}
     (\kow\ot \id)(\Theta_\mu)\in \sum_{\lambda+\nu=\mu} (U^-_\mx)_{-\lambda}\ot (U^-_\mx)_{-\nu} K_\lambda^{-1}\ot (U^+_\mx)_\mu.
  \end{align*}
  For $e, e'\in U^+$ the definition of $\Theta$ and the properties of a skew pairing \eqref{eq:skew-hopf} imply that
  \begin{align}\label{eq:ee'1}
    ee'&=\sum_\mu \langle F_\mu, ee'\rangle_\mx E_\mu
    =\sum_\mu \langle F_{\mu (1)}, e'\rangle_\mx \langle F_{\mu (2)}, e\rangle_\mx E_\mu.
  \end{align}
  On the other hand
  \begin{align}\label{eq:ee'2}
    e e' = \sum_{\lambda, \nu} \langle F_\nu, e\rangle_\mx E_\nu
    \langle F_\lambda, e'\rangle_\mx E_\lambda
         =  \sum_{\lambda, \nu} \langle F_\lambda, e'\rangle_\mx \langle F_\nu K_\lambda^{-1}, e\rangle_\mx E_\nu
     E_\lambda.
  \end{align} 
  Comparison of \eqref{eq:ee'1} and \eqref{eq:ee'2} implies \eqref{eq:kow-Theta1}, as the componentwise pairing between $\bigoplus_\lambda (U_\mx^-)_{-\lambda}\ot U^-_\mx K_\lambda^{-1}$ and  $U_\mx^+\ot U_\mx^+$ is nondegenerate. Equation \eqref{eq:kow-Theta2} is verified analogously.
\end{proof}
%%%%%%%%%%%%%%%%%%%%%%%%%%%%%%%%%%%%%%%%%%
\subsection{The skew derivations $\partial_i^L$ and $\partial_i^R$ on $U^-$} \label{sec:skew-derivation}
%%%%%%%%%%%%%%%%%%%%%%%%%%%%%%%%%%%%%%%%%
For quantized universal enveloping algebras Kashiwara \cite[3.4]{a-Kash91} and Lusztig \cite[1.2.13, 3.1.6]{b-Lusztig94} consider skew-derivations on $U^+$ and $U^-$. As observed in \cite[Section 5]{a-Heck10}, these skew derivations allow a straightforward generalisation to the setting of (pre-)Nichols algebras of diagonal type. In the case of $U^-$,  for any $i\in I$, the skew derivations are the uniquely determined linear maps $\partial_i^R, \partial_i^L: U^- \rightarrow U^-$ such that
\begin{align}\label{eq:partial-def}
  [E_i,f]=K_i \partial_i^L(f)- \partial_i^R(f)K_i^{-1} \qquad \mbox{for all $f\in U^-$.}
\end{align}
For later reference we collect the main properties of the skew derivations $\partial_i^L$ and $\partial_i^R$ on $U^-$. It follows from the last relation in \eqref{eq:def-rel} that
\begin{align}\label{eq:partialFj}
  \partial_i^L(F_j)=\delta_{ij}=\partial_i^R(F_j) \qquad \mbox{for all $j\in I$.}
\end{align}
Moreover, Equation \eqref{eq:partial-def} implies that
\begin{equation}\label{eq:skew-properties}
\begin{aligned}
  \partial_i^L(f_\mu f_\nu) = \partial_i^L(f_\mu) f_\nu + \chi(\mu,\alpha_i) f_\mu \partial_i^L(f_\nu),\\
  \partial_i^R(f_\mu f_\nu) = \chi(\nu,\alpha_i)\partial_i^R(f_\mu) f_\nu + f_\mu \partial_i^L(f_\nu)
\end{aligned}
\end{equation}
for all $f_\mu\in U^-_{-\mu}$, $f_\nu\in U^-_{-\nu}$. In other words, $\partial_i^L$ is a left skew derivation on $U^-$ while $\partial_i^R$ is a right skew derivation. The skew derivations $\partial_i^L$ and $\partial_i^R$ are uniquely determined by the properties \eqref{eq:partialFj} and \eqref{eq:skew-properties}. They can also be read off from the coproduct on $U^-$. Indeed, for any $f_\mu\in U^-_{-\mu}$ one has
\begin{equation}\label{eq:kow-partial}
\begin{aligned}
  \kow(f_\mu)&=f_\mu\ot K_\mu^{-1} + \sum_i \partial_i^L(f_\mu)\ot F_i K^{-1}_{\mu-\alpha_i} + (\mathrm{rest})_1,\\
  \kow(f_\mu)&= 1\ot f_\mu + \sum_i F_i \ot \partial_i^R(f_\mu)K_i^{-1} +  (\mathrm{rest})_1
\end{aligned}
\end{equation}
where $(\mathrm{rest})_1\in \sum_{|\alpha|\ge 2} U^-_{\alpha-\mu}\ot U^-_{-\alpha} K^{-1}_{\mu-\alpha}$ and $(\mathrm{rest})_2\in \sum_{|\alpha|\ge 2} U^-_{-\alpha}\ot U^-_{\alpha-\mu} K^{-1}_{\alpha}$. The properties \eqref{eq:kow-partial} of the coproduct and the definition \eqref{eq:lactract} of the left and the right action of $U^+$ on $H\ltimes U^-$ imply that
\begin{align}\label{eq:lract-partial}
  E_i\lact f=\partial_i^R(f)K_i^{-1}, \qquad f\ract E_i=\partial_i^L(f) \qquad \mbox{for all $f\in U^-$, $i\in I$.}
\end{align}
Let $\langle \, , \,\rangle: U^-\ot U^+\rightarrow \field$ denote the pairing defined by
\begin{align*}
  \langle f,e \rangle = \langle \pi_\mx(f),\pi_\mx(e)\rangle_\mx \qquad\mbox{for all $f\in U^-, e\in U^+ $}
\end{align*}
where we use $\pi_\mx$ to denote both canonical projections $U^+\rightarrow U^+_\mx$ and $U^-\rightarrow U^-_\mx$. The relations \eqref{eq:kow-partial} and \eqref{eq:skew-hopf} imply that
\begin{align}\label{eq:pair-partial}
  \langle f, E_i e \rangle=\langle \partial^L_i(f),e\rangle, \qquad
  \langle f, e E_i  \rangle=\langle \partial^R_i(f),e\rangle
\end{align}
for all $f\in U^-$, $e\in U^+$ and $i\in I$. This tells us how the quasi $R$-matrix $\Theta$ behaves under the skew derivations.
%%%%%%%%%%%%%%%%%%%%%%%%%%%%%%%%%%%%%%%%%%%%
\begin{lem}\label{lem:partialTheta}
  For any $i\in I$ the following relations hold:
\begin{align}\label{eq:Theta-partial}
  \qquad (\partial^L_i\ot \id)(\Theta)=-(1\ot E_i)\Theta, \qquad (\partial^R_i\ot \id)(\Theta)=-\Theta(1\ot E_i). 
\end{align}
\end{lem}
%%%%%%%%%%%%%%%%%%%%%%%%%%%%%%%%%%%%%%%%%%%%
\begin{proof}
For any $f\in U^-_\mx$ the first relation in \eqref{eq:pair-partial} implies that 
\begin{align*}
  \partial^L_i(f)= \sum_{\mu}\langle \partial^L_i(f),E_\mu\rangle F_\mu=\sum_\mu \langle f,E_i E_\mu\rangle F_\mu 
\end{align*}  
and hence
\begin{align*}
  \sum_\mu (-1)^{|\nu|} \partial_i^L(F_\nu)\ot E_\nu&=\sum_{\mu,\nu}(-1)^{|\nu|} \langle F_\nu,E_i E_\mu\rangle F_\mu \ot E_\nu \\
  &= \sum_\mu(-1)^{|\mu|+1}F_\mu\ot E_i E_\mu 
\end{align*}
which proves the first relation in \eqref{eq:Theta-partial}. The second relation is verified similarly.
\end{proof}
%%%%%%%%%%%%%%%%%%%%%%%%%%%%%%%%%%%%%%%%%%%%%%%%%%%%%%%%%%%%%%%%%
\begin{cor}\label{cor:EFTheta}{\upshape (see \cite[Theorem 4.1.2]{b-Lusztig94})}
  The element $\Theta$ satisfies the relations
  \begin{align}
    (E_j \ot 1 + K_j \ot E_j) \Theta &= \Theta (E_j \ot 1 + K_j^{-1} \ot E_j)\label{eq:EFTheta1},\\
    (F_j \ot K_j^{-1} + 1\ot F_j) \Theta &= \Theta (F_j \ot K_j + 1 \ot F_j)\label{eq:EFTheta2}
  \end{align}
  for all $j\in I$. 
\end{cor}
%%%%%%%%%%%%%%%%%%%%%%%%%%%%%%%%%%%%%%%%%%%
\begin{proof}
  Relation \eqref{eq:EFTheta1} follows from Lemma \ref{lem:partialTheta} and the defining relation \eqref{eq:partial-def} of the skew derivations $\partial^L_i$ and $\partial^R_i$. The second relation is verified analogously using skew derivations on $U^+$.
\end{proof} 
%%%%%%%%%%%%%%%%%%%%%%%%%%%%%%%%%%%%%%%%%%5
\begin{rema}
Just as in \cite[Theorem 4.1.2]{b-Lusztig94} one can show that the element $\Theta \in $ is the unique element of the form $\Theta=\sum_{\mu\in \N^n}\Theta_\mu$ with $\Theta_\mu\in (U^-_\mx)_{-\mu}\ot (U^+_\mx)_\mu$ for which $\Theta_0=1\ot 1$ and the relations in Corollary \ref{cor:EFTheta} hold. 
\end{rema}  
%%%%%%%%%%%%%%%%%%%%%%%%%%%%%%%%%%%%%%%%%%%

%%%%%%%%%%%%%%%%%%%%%%%%%%%%%%%%%%%%%%%%%%%
\subsection{The algebra homomorphism $\osigma:U^-_\mx\rightarrow U^+_\mx\rtimes H $}\label{sec:osigma1}
%%%%%%%%%%%%%%%%%%%%%%%%%%%%%%%%%%%%%%%%%%%%%%%%%%%%%%%%%%%%%%%%%
By Lemma \ref{lem:EKF} there exists a well-defined algebra homomorphism
$\osigma:U^-_\mx\rightarrow U^+_\mx\rtimes H$ such that
\begin{align}\label{eq:sbFi}
  \osigma(F_i)= c_{\tau(i)}K_iE_{\tau(i)}=c_{\tau(i)}K_i\, \omega\circ \tau(F_i) \qquad \mbox{for all $i\in I$.}
\end{align}  
For any $f_\mu\in (U^-_\mx)_{-\mu}$ we can write
\begin{align}\label{eq:amu-def}
  \osigma(f_\mu)= a_\mu K_\mu\, \omega\circ \tau(f_\mu)
\end{align}
for some $a_\mu\in \field$.
%%%%%%%%%%%%%%%%%%%%%%%%%%%%%%%%%%%%%%%%%%%%%%%%%%
\begin{lem}
  The coefficients $a_\mu$ for $\mu\in \N^n$ are uniquely determined by $a_{\alpha_i}=c_{\tau(i)}$ for all $i\in I$ and by the recursion
  \begin{align}\label{eq:amu-rec}
     a_{\mu+\nu} = \chi(-\nu,\tau(\mu)) a_\mu a_\nu \qquad \mbox{for all $\mu,\nu\in \N^n$.}
  \end{align}
   In particular, the coefficient $a_\mu$ in \eqref{eq:amu-def} only depends on $\mu\in \N^n$ and not on the chosen element $f_\mu\in (U^-_\mx)_{-\mu}$. 
\end{lem}  
%%%%%%%%%%%%%%%%%%%%%%%%%%%%%%%%%%%%%%%%%%%%%%%%%%
\begin{proof}
  By \eqref{eq:sbFi} we have $a_{\alpha_i}=c_{\tau(i)}$ for all $i\in I$. Let  $f\in (U^-_\mx)_{-\mu}$ and  $g\in (U^-_\mx)_{-\nu}$. Then \eqref{eq:amu-def} implies that
  \begin{align*}
    a_{\mu+\nu}K_{\mu+\nu} \omega\circ \tau(fg)&=\osigma(fg)=\osigma(f) \osigma(g)\\
    &= a_\mu K_\mu  \omega\circ \tau(f) a_\nu K_\nu \omega \circ \tau(g)\\
    &=\chi(-\nu,\tau(\mu)) a_\mu a_\nu K_{\mu+\nu} \omega \circ \tau(fg).
  \end{align*}
  Hence we get the recursive formula \eqref{eq:amu-rec}.
\end{proof}
%%%%%%%%%%%%%%%%%%%%%%%%%%%%%%%%%%%%%%%%%%%%%%%%%%
We want to apply the algebra homomorphism $\kow \circ\osigma$ to the first tensor factor of the quasi $R$-matrix $\Theta=\sum_\mu F_\mu \ot E_\mu$. As $\omega \circ \tau$ is a coalgebra antiaumorphism, Lemma \ref{lem:kowidTheta} implies that
\begin{align*}
  \sum_\mu \kow \circ \osigma (F_\mu) \ot E_\mu &=\sum_\mu a_\mu K_\mu \omega\circ\tau(F_{\mu (2)}) \ot K_\mu \omega\circ \tau(F_{\mu (1)}) \ot E_\mu\\
  &= \sum_{\lambda,\nu} a_{\lambda+\nu} K_{\lambda+\nu} \omega\circ \tau(F_\nu K_\lambda^{-1})\ot K_{\lambda+\nu} \omega\circ \tau(F_\lambda)\ot E_\nu E_\lambda .
\end{align*}
Hence using the recursion \eqref{eq:amu-rec} we obtain
\begin{align}\label{eq:kow-osigma1}
  \sum_\mu \kow \circ \osigma (F_\mu) \ot E_\mu &=\sum_{\lambda,\nu} \osigma(F_\nu) K_\lambda K_{\tau(\lambda)}\ot K_\nu \osigma(F_\lambda) \ot E_\nu E_\lambda.
\end{align}
On the other hand, Equation \eqref{eq:kow-Theta2} implies that
\begin{align}\label{eq:kow-osigma2}
  \sum_{\mu} \osigma(F_\mu)& \ot \kow(S^{-1}(E_\mu)K_\mu)\\
  &=\sum_{\nu,\rho} \chi(-\nu,\rho) \osigma(F_\nu F_\rho)\ot S^{-1}(E_\rho) K_{\nu+\rho} \ot S^{-1}(E_\nu)K_\nu.\nonumber
\end{align}
Formulas \eqref{eq:kow-osigma1} and \eqref{eq:kow-osigma2} will be used to verify the associativity of the twist product in the next section.
%%%%%%%%%%%%%%%%%%%%%%%%%%%%%%%%%%%%%%%%%%%%%%%%%%%%%%%%%%%%%%%%%
\subsection{Definition and associativity of the twist product}\label{sec:twist-def-ass}
%%%%%%%%%%%%%%%%%%%%%%%%%%%%%%%%%%%%%%%%%%%%%%%%%%%%%%%%%%%%%%%%%
We now use the quasi $R$-matrix $\Theta$ and the algebra homomorphism $\osigma$ to define a twisted product on the partial bosonization $H_\theta\ltimes U^-$. Recall that we write formally $\Theta=\sum_{\rho} (-1)^{|\rho|} F_\rho\ot E_\rho$ and that we write $S$ to denote the antipode of $U(\chi)$. For any $f,g\in U^-$ we define
\begin{align}\label{eq:fstarg}
  f \star g = \sum_{\rho} (-1)^{|\rho|}(\osigma(F_\rho)\lact f) K_\rho [g \ract (S^{-1}(E_\rho)K_\rho)] .
\end{align}
Observe that $\osigma(F_\rho)\lact f \in U^-K^{-1}_{\tau(\rho)}$ and that $g \ract (S^{-1}(E_\rho)K_\rho)\in U^-$ and hence $f\star g\in H_\theta\ltimes U^-\subset U(\chi)$. For later reference it is convenient to spell out the formula for the twist product \eqref{eq:fstarg} explicitly in the case where one of the factors equals a generator $F_i$.
%%%%%%%%%%%%%%%%%%%%%%%%%%%%%%%%%%%%%%%%%%%%%%%%%%%%%%%
\begin{lem}\label{lem:mu-map-star}
  For any $f,g\in U^-$ and any $i\in I$ the relations
  \begin{align}
    F_i\star g &= F_ig + c_i q_{i\tau(i)} K_{\tau(i)}K_i^{-1} \partial_{\tau(i)}^L(g),\label{eq:Fistarg}\\
     f\star F_i &= fF_i + c_{\tau(i)} q_{i\tau(i)} \partial_{\tau(i)}^R(f)  K_i K_{\tau(i)}^{-1}\label{eq:fstarFi}
  \end{align}
  hold in $H_\theta\ltimes U^-$.
\end{lem}  
%%%%%%%%%%%%%%%%%%%%%%%%%%%%%%%%%%%%%%%%%%%%%%%%%%%%%%
\begin{proof}
  By \eqref{eq:fstarg} we have
  \begin{align*}
    F_i\star g &\stackrel{\phantom{\eqref{eq:lract-partial}}}{=} F_i g -(\osigma(F_{\tau(i)})\lact F_i)K_{\tau(i)}[g \ract(S^{-1}(E_{\tau(i)})K_{\tau(i)})]\\
    &\stackrel{\phantom{\eqref{eq:lract-partial}}}{=} F_i g + c_i \big((K_{\tau(i)}E_i)\lact F_i\big)K_{\tau(i)}[g\ract E_{\tau(i)}]\\
    &\stackrel{\eqref{eq:lract-partial}}{=} F_ig + c_i q_{i\tau(i)} K_i^{-1} K_{\tau(i)} \partial^L_i(g).
  \end{align*}
  This proves \eqref{eq:Fistarg}. Equation \eqref{eq:fstarFi} is verified by a similar calculation.
\end{proof}  
%%%%%%%%%%%%%%%%%%%%%%%%%%%%%%%%%%%%%%%%%%%%%%%%%%%%%%
We now want to extend the definition of the twist product to all of $H_\theta\ltimes U^-$. For simplicity we suppress tensor symbols and write elements $h\ot f\in H_\theta\ltimes U^-$ as $hf$. We define a bilinear binary 
operation $\star$ on $H_\theta \ltimes U^-$ by
\begin{align}\label{eq:KfstarKg}
   (K_\lambda f) \star (K_\mu g) = \chi(\alpha,\mu) K_{\lambda+\mu} (f\star g) 
\end{align}
for all $\lambda,\mu\in \Z^n_\theta$, $f\in U^-_{-\alpha}$, $g\in U^-$ and where $f\star g\in H_\theta\ltimes U^-$ is defined by \eqref{eq:fstarg}.
%%%%%%%%%%%%%%%%%%%%%%%%%%%%%%%%%%%%%%%%%%%%%%%%%%%%%
\begin{thm}
\label{thm:star-twist}
 For all pre-Nichols algebras of diagonal type $U^+$, the bilinear binary operation on $H_\theta \ltimes U^-$ defined by \eqref{eq:KfstarKg} is associative.
\end{thm}  
%%%%%%%%%%%%%%%%%%%%%%%%%%%%%%%%%%%%%%%%%%%%%%%%%%%%%
\begin{proof}
  Let $\lambda, \mu, \nu\in \Z^n_\theta$ and $f\in U^-_{-\alpha}$, $g\in U^-_{-\beta}$, $h\in U^-_{-\gamma}$. By the discussion following \eqref{eq:fstarg} we can write
  \begin{align}\label{eq:u-def}
     f\star g = \sum_{\rho} K_{\rho-\tau(\rho)} u^-_\rho(f,g)  
  \end{align}  
  where $u_\rho^-(f,g)\in U^-_{-(\alpha+\beta-\rho-\tau(\rho))}$. With this notation we calculate
  \begin{align*}
    \big((K_\lambda f)\star (K_\mu g)\big) \star (K_\gamma h)&=\chi(\alpha,\mu)\sum_{\rho} \big(K_{\lambda+\mu+ \rho-\tau(\rho)}u^-_\rho(f,g)\big)\star(K_\gamma h)\\
  &=\chi(\alpha,\mu) \chi(\alpha+\beta,\gamma) K_{\lambda+\mu+\gamma}\big((f\star g)\star h\big)
  \end{align*}
  where we used the fact that $\chi(\rho+\tau(\rho),\gamma)=1$ as $\tau(\gamma)=-\gamma$. Similarly one calculates
  \begin{align*}
    (K_\lambda f)\star \big((K_\mu g)\star (K_\gamma h)\big)=
  \chi(\beta,\gamma) \chi(\alpha,\mu+\gamma)K_{\lambda+\mu+\gamma}\big(f\star (g\star h)\big).
  \end{align*}
  Hence it suffices to show that $(f\star g)\star h = f \star (g \star h)$. Using \eqref{eq:u-def} we obtain
  \begin{align}\label{eq:(fg)h1}
    (f\star g)\star h = \sum_{\rho,\sigma}K_{\rho-\tau(\rho)}K_{\sigma-\tau(\sigma)} u^-_\sigma (u^-_\rho(f,g),h).
  \end{align}  
    By definition of $u^-_\rho(f,g)$ in \eqref{eq:u-def} we have
  \begin{align}\label{eq:um-formula}
    u^-_\rho(f,g)=(-1)^{|\rho|} \chi(\alpha{-}\tau(\rho),\rho{-}\tau(\rho))\big[(\osigma(F_\rho)\lact f)K_{\tau(\rho)}\big] \big[g\ract (S^{-1}(E_\rho)K_\rho)\big].
  \end{align}
  Inserting the above formula into \eqref{eq:(fg)h1} twice, we obtain
  \begin{align*}
    (f \star g)\star h &= \sum_{\rho,\sigma} K_{\rho+\sigma-\tau(\rho+\sigma)} (-1)^{|\rho|+|\sigma|}\chi(\alpha{-}\tau(\rho),\rho{-}\tau(\rho)) \chi(\alpha{+}\beta{-}\tau(\sigma),\sigma{-}\tau(\sigma))\cdot\\
      \cdot\Big[\osigma(F_\sigma)&\lact  \Big(\big[(\osigma(F_\rho)\lact f)K_{\tau(\rho)}\big] \big[g\ract (S^{-1}(E_\rho)K_\rho)\big]\Big)K_{\tau(\sigma)}\Big]\,
    \big[h \ract(S^{-1}(E_\sigma)K_\sigma)\big].\nonumber
  \end{align*}
  Using the fact that $U^-\rtimes H$ is a left module algebra over $U^+\rtimes H$ and formula \eqref{eq:kow-osigma1} we obtain
  \begin{align}
    (f\star g)\star h =& \sum_{\rho,\lambda, \nu} K_{\rho,\lambda,\nu}^\theta a_{\alpha,\beta,\rho,\lambda,\nu} \big(\osigma(F_\nu)K_{\lambda+\tau(\lambda)}\big)\lact\big[(\osigma(F_\rho)\lact f)K_{\tau(\rho)}\big]\cdot\label{eq:(fg)h2}\\
    \cdot \big( K_\nu \osigma(&F_\lambda)\big)\lact \big[g \ract(S^{-1}(E_\rho)K_\rho)\big] K_{\tau(\lambda+\nu)} \big[h \ract(S^{-1}(E_\nu E_\lambda)K_{\nu+\lambda})\big] \nonumber
  \end{align}
  where we use the abbreviations $K^\theta_{\rho,\lambda,\nu}=K_{\rho+\lambda+\nu-\tau(\rho+\lambda+\nu)}$ and
  \begin{align}\label{eq:a-def}
    a_{\alpha,\beta,\rho,\lambda,\nu}=(-1)^{|\rho|+|\lambda|+|\nu|}& \chi(\alpha{-}\tau(\rho),\rho{-}\tau(\rho))\cdot\\
    &\qquad \cdot\chi(\alpha{+}\beta{-}\tau(\lambda{+}\nu),\lambda{+}\nu {-}\tau(\lambda{+}\nu)).\nonumber
  \end{align}
  Formula \eqref{eq:(fg)h2} can be rewritten as
  \begin{align}
    (f\star g)\star h =& \sum_{\rho,\lambda, \nu} K_{\rho,\lambda,\nu}^\theta a_{\alpha,\beta,\rho,\lambda,\nu} \chi(-\nu,\tau(\rho))^2\, \chi(\nu,\tau(\lambda)) \, \chi(\beta{-}\tau(\lambda),\tau(\nu))\cdot \label{eq:(fg)h3}\\
    &\cdot\big[(\osigma(F_\nu F_\rho)\lact f)K_{\tau(\rho+\nu)}\big]
    \cdot \big[\big(\osigma(F_\lambda)\lact g \ract(S^{-1}(E_\rho)K_\rho)\big) K_{\tau(\lambda)} \big]\cdot\nonumber\\
    & \cdot\big[h \ract(S^{-1}(E_\nu E_\lambda)K_{\nu+\lambda}) \big]. \nonumber
  \end{align}
  Similarly, to obtain an explicit expression for $f\star(g\star h)$, we use \eqref{eq:u-def} to write
  \begin{align*}
     f\star (g\star h) = \sum_{\sigma, \lambda} K_{\sigma-\tau(\sigma)} K_{\lambda-\tau(\lambda)} \chi(\alpha,\lambda{-}\tau(\lambda)) u^-_\sigma(f, u^-_\lambda(g,h)).
  \end{align*}
  Using again \eqref{eq:um-formula} we obtain
  \begin{align*}
    f&\star (g\star h) {=} \sum_{\sigma, \lambda}  K_{\sigma+\lambda-\tau(\sigma+\lambda)} (-1)^{|\sigma|+|\lambda|} \chi(\alpha{-}\tau(\sigma),\sigma{-}\tau(\sigma)) \chi(\alpha{+}\beta{-}\tau(\lambda),\lambda{-}\tau(\lambda))\cdot\\
    &\cdot \big[(\osigma(F_\sigma)\lact f)K_{\tau(\sigma)}\big]
    \Big[\Big(\big[(\osigma(F_\lambda)\lact g)K_{\tau(\lambda)}\big]\,\big[h\ract(S^{-1}(E_\lambda)K_\lambda)\big] \Big)\ract (S^{-1}(E_\sigma)K_\sigma)\Big].
  \end{align*}
  Using Equation \eqref{eq:kow-osigma2} and the fact that $U^-\rtimes H$ is a right module algebra over $U^+\rtimes H$, we obtain
  \begin{align}
    &f\star (g\star h) = \sum_{\nu, \rho, \lambda} K^\theta_{\rho,\lambda,\nu} b_{\alpha,\beta,\rho,\lambda,\nu} \big[(\osigma(F_\nu F_\rho)\lact f)K_{\tau(\nu+\rho)}\big]\cdot \label{eq:f(gh)2}\\
    &\cdot \big[(\osigma(F_\lambda)\lact g)K_{\tau(\lambda)}\big] \ract (S^{-1}(E_\rho)K_{\nu+\rho}) \,\, \big[h\ract (S^{-1}(E_\lambda) K_\lambda)\big]\ract(S^{-1}(E_\nu) K_\nu)\nonumber
  \end{align}
  where as before $K^\theta_{\rho,\lambda,\nu}=K_{\rho+\lambda+\nu-\tau(\rho+\lambda+\nu)}$ and
  \begin{align}\label{eq:b-def}
    b_{\alpha,\beta,\rho,\lambda,\nu}=(-1)^{|\rho|+|\lambda|+|\nu|} &\chi(\alpha{-}\tau(\nu{+}\rho), \nu{+}\rho{-}\tau(\nu{+}\rho))\cdot\\
     &\qquad \cdot\chi(\alpha{+}\beta{-}\tau(\lambda),\lambda{-}\tau(\lambda))\, \chi(-\nu,\rho).\nonumber
  \end{align}
  Formula \eqref{eq:f(gh)2} can be rewritten as
  \begin{align}\label{eq:f(gh)3}
    f\star (g\star h) =& \sum_{\nu, \rho, \lambda} K^\theta_{\rho,\lambda,\nu} b_{\alpha,\beta,\rho,\lambda,\nu} \chi(\nu,\beta+\lambda-\tau(\lambda)-\rho)\cdot\\
    &\cdot\big[(\osigma(F_\nu F_\rho)\lact f)K_{\tau(\rho+\nu)}\big]
    \cdot \big[\big(\osigma(F_\lambda)\lact g \ract(S^{-1}(E_\rho)K_\rho)\big) K_{\tau(\lambda)} \big]\cdot\nonumber\\
    & \cdot\big[h \ract(S^{-1}(E_\nu E_\lambda)K_{\nu+\lambda}) \big]. \nonumber
  \end{align}
  Now the relation $f\star(g\star h)=(f\star g)\star h$ follows from comparison of the Equations \eqref{eq:(fg)h3} and \eqref{eq:f(gh)3} and the fact that
  \begin{align*}
   a_{\alpha,\beta,\rho,\lambda,\nu} \chi(-\nu,\tau(\rho))^2\, \chi(\nu,\tau(\lambda)) \, \chi(\beta{-}\tau(\lambda),\tau(\nu))
    &= b_{\alpha,\beta,\rho,\lambda,\nu} \chi(\nu,\beta{+}\lambda{-}\tau(\lambda){-}\rho)
  \end{align*}
  which in turn follows from \eqref{eq:a-def} and \eqref{eq:b-def} by direct calculation.
\end{proof}
%%%%%%%%%%%%%%%%%%%%%%%%%%%%%%%%%%%%%%%%%%%%%%%%%%%%%%%%%%
It is convenient to invert the formula \eqref{eq:fstarg}. In the following lemma we express the usual multiplication in $U^-$ in terms of the twist product $\star$ on $H_\theta \ltimes U^-$.
%%%%%%%%%%%%%%%%%%%%%%%%%%%%%%%%%%%%%%%%%%%%%%%%%%%%%%%%%%
\begin{lem}\label{lem:fg}
  For any $f,g\in U^-$ the relation
  \begin{align*}
     fg = \sum_{\mu\in \N^n}(-1)^{|\mu|}\Big(\big(\osigma(F_\mu)\lact f\big)K_\mu\Big)\star [g\ract E_\mu]
  \end{align*}
  holds in $(H_\theta\ltimes U^-,\star)$.
\end{lem}
%%%%%%%%%%%%%%%%%%%%%%%%%%%%%%%%%%%%%%%%%%%%%%%%%%%%%%%%%%
\begin{proof}
   Note first that \eqref{eq:kow-Theta2} implies that
  \begin{align}\label{eq:111}
     \sum_{\nu,\mu\in \N^n} (-1)^{|\mu|+|\nu|} \osigma(F_\nu F_\mu) \ot K_\nu K_\mu \ot E_\mu K_\mu^{-1} S^{-1}(E_\nu) K_\nu K_\mu = 1\ot 1 \ot 1.
  \end{align}
  By bilinearity we may assume that $f\in U^-_{-\alpha}$ for some $\alpha\in \N^n$. We obtain
  \begin{align*}
    &\sum_{\mu}\Big(\big(\osigma(F_\mu)\lact f\big)K_\mu\Big)\star [g\ract E_\mu]\\
    &\quad\stackrel{\eqref{eq:KfstarKg}}{=} \sum_\mu (-1)^{|\mu|} \chi(\alpha{-}\tau(\mu),\mu{-}\tau(\mu))\, K_{\mu-\tau(\mu)} \big[\big(\osigma(F_\mu)\lact f\big)K_{\tau(\mu)}\big] \star [g\ract E_\mu]\\
    &\quad\stackrel{\eqref{eq:fstarg}}{=}  \sum_{\mu,\nu} (-1)^{|\mu|+|\nu|} \chi(\alpha{-}\tau(\mu),\mu{-}\tau(\mu))\, K_{\mu-\tau(\mu)} \Big(\osigma(F_\nu)\lact\big[\big(\osigma(F_\mu)\lact f\big)K_{\tau(\mu)}\big]\Big) K_\nu \cdot\\
        &\qquad \qquad \qquad \qquad\qquad \qquad \qquad \qquad\qquad \cdot[g\ract E_\mu]\ract(S^{-1}(E_\nu)K_\nu)\\
    &\quad\stackrel{\phantom{\eqref{eq:111}}}{=}  \sum_{\mu,\nu} (-1)^{|\mu|+|\nu|} \Big(\osigma(F_\nu)\lact\big(\osigma(F_\mu)\lact f\big)\Big)K_{\mu+\nu}  \big[g\ract (E_\mu K_\mu^{-1} S^{-1}(E_\nu)K_{\nu+\mu})\big]\\
    &\quad\stackrel{\eqref{eq:111}}{=}fg  
  \end{align*}  
  which proves the lemma.  
\end{proof}  
%%%%%%%%%%%%%%%%%%%%%%%%%%%%%%%%%%%%%%%%%%%%%%%%%%%%%%%%%%
\subsection{A twisted coaction}
\label{sec:comod}
%%%%%%%%%%%%%%%%%%%%%%%%%%%%%%%%%%%%%%%%%%%%%%%%%%%%%%%%%%
Define a linear map $\kow_\star: H_\theta\ltimes U^- \rightarrow (H_\theta \ltimes U^-) \otimes U(\chi)_\mx$ by
\begin{align}\label{eq:kowstar-def}
  \kow_\star (K_\nu f_\alpha)= \sum_{\lambda,\mu} K_\nu\big(\osigma(F_\lambda) \lact f_\alpha \ract E_\mu\big) K_\lambda \ot K_\nu F_\mu K_{\mu-\alpha} E_\lambda
\end{align}
for all $\nu\in H_\theta$ and $f_\alpha\in U^-_{-\alpha}$. For any $f\in H_\theta\ot U^-$ and any $\nu\in H_\theta$ we have
\begin{align}
  \kow_\star(K_\nu f)&=(K_\nu\ot K_\nu) (\star \ot \cdot) \kow_\star(f) \label{eq:kow-Kstar},\\
  \kow_\star(f K_\nu)&= \kow_\star(f)(\star \ot \cdot)(K_\nu\ot K_\nu) \label{eq:kow-starK}
\end{align}
as $\chi(\nu,\lambda+\tau(\lambda))=1$ for all $\lambda\in \N^n$.
Moreover,
\begin{align}\label{eq:kowstarFi}
  \kow_\star(F_i) = F_i\ot K_i^{-1} + c_i K_i^{-1}K_{\tau(i)} \ot E_{\tau(i)} K_i^{-1} + 1 \ot F_i
\end{align}
for all $i\in I$. 
%%%%%%%%%%%%%%%%%%%%%%%%%%%%%%%%%%%%%%%%%%555
\begin{prop}\label{prop:kowstar}
  The map $\kow_\star$ endows $(H_\theta\ltimes U^-,\star)$ with the structure of a right $U(\chi)_\mx$-comodule algebra. 
\end{prop}  
%%%%%%%%%%%%%%%%%%%%%%%%%%%%%%%%%%%%%%%%%%%%%
\begin{proof}
  Let $f\in U^-_{-\alpha}$. It follows from Lemma \ref{lem:kowidTheta} that
  \begin{align*}
    (\id &\ot \kow)\circ\kow_\star(f)
   = \sum_{\lambda,\mu}\big(\osigma(F_\lambda) \lact f \ract E_\mu\big) K_\lambda \ot \kow(F_\mu K_{\mu-\alpha}  E_\lambda)\\
    &\stackrel{\phantom{\eqref{eq:kow-starK}}}{=}\sum_{\nu,\rho, \lambda,\mu}\big(\osigma(F_\lambda F_\mu) \lact f \ract (E_\nu E_\rho)\big) K_{\lambda+\mu} \ot F_\rho K_{\rho+\nu-\alpha} E_\lambda K_\mu \ot F_\nu K_{\nu-\alpha} E_\mu\\
    & \stackrel{\phantom{\eqref{eq:kow-starK}}}{=} \sum_{\nu,\rho, \lambda,\mu}\Big(\osigma(F_\lambda)\lact\big((\osigma(F_\mu) \lact f \ract E_\nu)K_{\tau(\mu)}\big)\ract E_\rho\Big) K_{\lambda+\mu-\tau(\mu)}\\
    &\qquad \qquad \qquad \qquad \qquad \qquad
    \ot F_\rho K_{\rho+\nu+\tau(\mu)-\alpha} E_\lambda K_{\mu-\tau(\mu)} \ot F_\nu K_{\nu-\alpha} E_\mu\\
    &\stackrel{\eqref{eq:kow-starK}}{=} \sum_{\mu,\nu} \kow_{\star}\big((\osigma(F_\mu)\lact f \ract E_\nu)\big) K_\mu\big) \ot F_\nu K_{\nu-\alpha} E_\mu\\
    &\stackrel{\phantom{\eqref{eq:kow-starK}}}{=} (\kow_\star \ot \id) \circ \kow_\star(f).
  \end{align*}
  Hence the map $\kow_\star$ is coassociative and $(H_\theta\ltimes U^-,\vep,\kow_\star)$ is a right $U(\chi)_\mx$-comodule. It remains to check that $\kow_\star$ is an algebra homomorphism. In view of \eqref{eq:kow-Kstar} and \eqref{eq:kow-starK} it suffices to show that
  \begin{align}\label{eq:kow-fstarg}
    \kow_\star(f\star g) = \kow_\star(f)(\star\ot \cdot)\kow_\star(g)
  \end{align}
  for all $f,g\in U^-$. Moreover, by the associativity of the twisted product $\star$ it suffices to verify relation \eqref{eq:kow-fstarg} for $f=F_i$ for all $i\in I$.

  Assume that $f\in U^-_{-\alpha}$ and $g\in U^+_{-\beta}$. From the definition of $\star$ and $\kow_\star$, using the fact that $(H\ltimes U^-)^\cop$ is a left and right $(H\ltimes U^+)$-module algebra via the actions \eqref{eq:lactract}, one obtains 
  \begin{align}
    \kow_\star(f\star g)=&\sum_{\kappa, \lambda, \mu,\nu, \rho} (-1)^{|\rho|} \chi(\kappa,\alpha-\mu-\rho)\, \chi(\nu,\beta-2\rho-\kappa) \label{eq:kow-star-alg-hom1}\\
    &\big(\osigma(F_\nu F_\rho)\lact f \ract E_\mu\big) K_{\nu+\rho}\big(\osigma(F_\lambda)\lact g \ract(S^{-1}(E_\rho)K_\rho E_\kappa)\big)K_\lambda \nonumber\\
    &\qquad\otimes F_\mu F_\kappa K_{\mu+\kappa+2\rho-\alpha-\beta} E_\nu E_\lambda,\nonumber\\
\kow_\star(f)(\star\ot \cdot)&\kow_\star(g) = \sum_{\kappa, \lambda, \mu,\nu, \rho} (-1)^{|\rho|}\chi(-\rho,\nu)\label{eq:kow-star-alg-hom2}\\
    &\big(\osigma(F_\rho F_\nu)\lact f \ract E_\mu\big) K_{\nu+\rho}\big(\osigma(F_\lambda)\lact g \ract(E_\kappa S^{-1}(E_\rho)K_\rho )\big)K_\lambda\nonumber\\
    &\qquad\otimes F_\mu K_{\mu-\alpha} E_\nu F_\kappa K_{\kappa-\lambda} E_\lambda. \nonumber
  \end{align}  
For $f=F_i$ the first factors in the second line of \eqref{eq:kow-star-alg-hom1} and \eqref{eq:kow-star-alg-hom2} are non-zero if and only if $\nu=\rho=\mu=0$ or two of $\nu, \rho, \mu$ vanish while the remaining one is one of $\nu=\alpha_{\tau(i)}, \rho=\alpha_{\tau(i)}$ or $\mu=\alpha_i$. Hence we get
\begin{align}\label{eq:kow-star-alg-hom3}
    \kow_\star(F_i\star g)&=\sum_{\kappa, \lambda}\chi(\kappa,\alpha_i) F_i (\osigma(F_\lambda)\lact g \ract E_\kappa) \ot F_\kappa K_{\kappa-\alpha_i-\beta}E_\lambda\\
&+(\osigma(F_{\tau(i)})\lact F_i)K_{\tau(i)}\big(\osigma(F_\lambda)\lact g \ract E_\kappa\big) K_\lambda \ot K_i^{-1}F_\kappa E_{\tau(i)} K_{\kappa-\beta}  E_\lambda\nonumber\\
    &+\chi(\kappa,-\alpha_{\tau(i)})(\osigma(F_{\tau(i)})\lact F_i)K_{\tau(i)}\big(\osigma(F_\lambda)\lact g \ract(E_{\tau(i)} E_\kappa)\big) K_\lambda \nonumber\\
    & \qquad \qquad \qquad \qquad \qquad \qquad \qquad \qquad \ot K_i^{-1} F_\kappa K_{\kappa+2\alpha_{\tau(i)}-\beta} E_\lambda\nonumber\\
    &+ (F_i\ract E_i) \big(\osigma(F_\lambda)\lact g \ract E_\kappa \big) K_\lambda \otimes F_i F_\kappa K_{\kappa-\beta} E_\lambda\nonumber.
\end{align}
Multiplying each summand in $U(\chi)_\mx\ot (U^+_\mx)_\kappa$ in Equation \eqref{eq:EFTheta1} for $j=\tau(i)$ from the right by $(-1)^{|\kappa|}K_\kappa\ot 1$, we obtain the relation
\begin{align*}
  \sum_\kappa E_{\tau(i)}F_\kappa K_\kappa \ot E_\kappa - & \chi(\kappa,-\alpha_{\tau(i)}) F_\kappa K_{\kappa+2\alpha_{\tau(i)}}\ot E_{\tau(i)} E_\kappa \\
  =& \sum_\kappa F_\kappa E_{\tau(i)}K_\kappa \ot E_\kappa - F_\kappa K_\kappa \ot E_\kappa E_{\tau(i)}.
\end{align*}  
  This relation can be applied to the second and third summand in \eqref{eq:kow-star-alg-hom3} to give
\begin{align*}
    \kow_\star(F_i\star g)&=\sum_{\kappa, \lambda}\chi(\kappa,\alpha_i) F_i (\osigma(F_\lambda)\lact g \ract E_\kappa) \ot F_\kappa K_{\kappa-\alpha_i-\beta}E_\lambda\\
&+(\osigma(F_{\tau(i)})\lact F_i)K_{\tau(i)}\big(\osigma(F_\lambda)\lact g \ract E_\kappa\big) K_\lambda \ot K_i^{-1}E_{\tau(i)}F_\kappa K_{\kappa-\beta} E_\lambda \\
    &+(\osigma(F_{\tau(i)})\lact F_i)K_{\tau(i)}\big(\osigma(F_\lambda)\lact g \ract(E_\kappa E_{\tau(i)})\big) K_\lambda \ot K_i^{-1} F_\kappa K_{\kappa-\beta}E_\lambda \\
    &+ (F_i\ract E_i) \big(\osigma(F_\lambda)\lact g \ract E_\kappa \big) K_\lambda \otimes F_i F_\kappa K_{\kappa-\beta} E_\lambda\\
    &= \kow_\star(F_i)(\star \ot \cdot) \kow_\star(g)\nonumber
\end{align*}
where the last equality follows from \eqref{eq:kow-star-alg-hom2} for $f=F_i$.
\end{proof}
%%%%%%%%%%%%%%%%%%%%%%%%%%%%%%%%%%%%%%%%%%%%%

%%%%%%%%%%%%%%%%%%%%%%%%%%%%%%%%%%%%%%%%%%%%%%%%%%%%%%%%%%%%%%%%%%%%%%%%%%%%%%%%%%%%
\section{Star products on partial bosonizations}\label{sec:star}
%%%%%%%%%%%%%%%%%%%%%%%%%%%%%%%%%%%%%%%%%%%%%%%%%%%%%%%%%%%%%%%%%%%%%%%%%%%%%%%%%%%%
In this section we introduce the notion of a star product on a graded algebra. We show that the twist product $\star$ on the partial bosonization $H_\theta\ltimes U^-$ from Section \ref{sec:twist-def-ass} is a star product which gives rise to an algebra isomorphic to the coideal subalgebra $B_\bc$. In Section \ref{sec:RelsRev} we employ the star product on $H_\theta\ltimes U^-$ to find a novel way to obtain defining relations for the algebra $B_\bc$.
%%%%%%%%%%%%%%%%%%%%%%%%%%%%%%%%%%%%%%%%%%%%%%%%%%%%%%%%%%%%%%%%%%%%%%%%
\subsection{General star products on $\N$-graded algebras}
%%%%%%%%%%%%%%%%%%%%%%%%%%%%%%%%%%%%%%%%%%%%%%%%%%%%%%%%%%%%%%%%
For any $\N$-graded $\field$-algebra $A=\bigoplus_{j\in \N} A_j$ and any $m\in \N$ we write $A_{<m}=\bigoplus_{j=0}^{m-1}A_j$ and $A_{\le m}=\bigoplus_{j=0}^{m}A_j$.
%%%%%%%%%%%%%%%%%%%%%%%%%%%%%%%%%%%%%%%%%%%%%%%%%%%%%%%%%%%%%%%%
\begin{defi}\label{def:star-prod} Let $A=\bigoplus_{j\in \N} A_j$ be a $\N$-graded $\field$-algebra. A star product on $A$ is an associative bilinear operation $* : A \times A \to A$, $(a,b)\mapsto a\ast b$
such that
\begin{align}\label{eq:a*b-ab}
  a * b - ab \in A_{< m+ n} \qquad \mbox{for all $a \in A_m$, $b \in A_n$.}
\end{align}
The star product $\ast$ on $A$ is called $0$-equivariant
if
\[
a * h = ah \quad \mbox{and} \quad h * a= h a \quad \mbox{for all} \quad h \in A_0, a \in A. 
\]
\end{defi}
%%%%%%%%%%%%%%%%%%%%%%%%%%%%%%%%%%%%%%%%%%%%%%%%%%%%%%%%%%%%55
If $*$ is a star product on an $\N$-graded algebra $A$ then $(A,\ast)$ is a filtered algebra with $\cF_m (A) := A_{\leq m}$. By condition \eqref{eq:a*b-ab} the associated graded algebra satisfies
\[
\gr (A, *) \cong A.
\]
If the graded algebra $A$ is generated in degrees 0 and 1, then every star product algebra structure $(A,*)$ is also generated in degrees 0 and 1.
%%%%%%%%%%%%%%%%%%%%%%%%%%%%%%%%%%%%%%%%%%%%%%%%%5
\begin{lem} 
\label{lem:sprod-isom}
Let $A$ be an $\N$-graded $\field$-algebra generated in degrees 0 and 1. 
\begin{enumerate}
\item[(i)] Any 0-equivariant star product on $A$ is uniquely determined by the $\field$-linear map $\mu^L:A_1\to \End_{\field}(A)$, $f\mapsto \mu^L_f$ defined by
\[
\mu^L_f(a) = f * a  - f a \qquad \mbox{for all $f \in A_1, a \in A$}.
\]
\item[(ii)] If $U$ is a graded subalgebra of $A$ such that $A_0 U = U A_0 = A$, then every 0-equivariant star product on $A$ is uniquely determined by the $\field$-linear map $\mu^L : U_1 \to \Hom_{\field}(U,A)$, $f\mapsto \mu^L_f$ defined by
\[
   \mu^L_f(b) = f*b - fb \qquad \mbox{for all $f \in U_1, b \in U$.}
\]
\end{enumerate}
\end{lem}
%%%%%%%%%%%%%%%%%%%%%%%%%%%%%%%%%%%%%%%%%%%%%%%%%%%%
\begin{proof} Let $*$ be a $0$-equivariant star product on $A$.

(i) Define a $\field$-linear map $M^L : A_{\leq 1} \to \End_{\field}(A)$ by
\[
M^L_f(a) := f * a = 
\begin{cases}
\mu^L_f(a) + f a, & \mbox{if} \; \; f \in A_1  \\
fa, & \mbox{if} \; \; f \in A_0
\end{cases} 
\]
where in the second case we use the assumption that $*$ is $0$-equivariant. The map $M^L$ is uniquely determined by the linear map $\mu^L : A_1 \to \End_{\field}(A)$. Since the algebra $(A, *)$ is generated in degrees $0$ and $1$, the vector space $A$ is the $\field$-span of elements 
of the form $a_1 * \cdots * a_j$ for $a_1, \ldots ,a_j \in A_{\leq 1}$. Since 
\[
(a_1 * \cdots * a_j) * a = M_{a_1}^L \ldots M_{a_j}^L (a)
\]
for all $a \in A$, $a_1, \ldots a_j \in A_{\leq 1}$, the bilinear operation $* : A \times A \to A$ 
is uniquely determined by the linear map $\mu^L : A_1 \to \End_{\field}(A)$.

(ii) Similarly to the first part, the assumption that $A_0 U = U A_0 = A$ and the $0$-equivariance of $*$
imply that the bilinear operation $* : A \times A \to A$ 
is uniquely determined by its restriction to $U_1 \times U$. This 
restriction is 
\[
f * b = f b + \mu_f^L(b)
\]
for $f \in U_1$ and $b \in U$, which completes the proof of the lemma.
\end{proof}
%%%%%%%%%%%%%%%%%%%%%%%%%
\subsection{The first star product on the partial bosonization $H_\theta \ltimes U^-$}
%%%%%%%%%%%%%%%%%%%%%%%%%%%%%%%%%%%%%55
\label{sec:first-star}
We work in the setting of Section \ref{sec:size}. Throughout Sections \ref{sec:first-star}, \ref{sec:second-star} and \ref{sec:RelsRev} we assume that $\bc \in \field^n$ satisfies condition \eqref{assume:parameters} in Section \ref{sec:prop-c}.

Recall from Section \ref{sec:Heis} that $U(\chi)^\poly$ denotes the subalgebra of $U(\chi)$ generated by $F_i$, $\Etil_i= E_i K_i^{-1}$, $K_i^{-1}$, $K_i K_{\tau(i)}^{-1}$ for all $i\in I$. Consider the triangular decomposition \eqref{eq:U+H-} of $U(\chi)$ written in reverse order
\begin{align*}
  U(\chi)\cong U^- \rtimes H \ltimes G^+
\end{align*}  
where  $G^+$ denotes the subalgebra of $U(\chi)$ generated by $\{\Etil_i\,|\,i\in I\}$. The restriction of this triangular decomposition to the subalgebra $\Upoly$ give rise to a linear isomorphism
\begin{align}\label{eq:Upoly-decomp}
\Upoly &\cong \big( H_\theta \ltimes U^- \big) \oplus  \big( \Upoly \mathrm{span}_\field \{\Etil_i, K_i^{-1} \mid i \in I\} \big)
\end{align}
where as before $\mathrm{span}_\field$ denotes the $\field$-linear span. Let
\begin{align}\label{eq:psi-poly}
\psi : \Upoly \twoheadrightarrow H_\theta \ltimes U^-
\end{align}
denote the $\field$-linear projection with respect to the direct sum decomposition \eqref{eq:Upoly-decomp}. Since the kernel of $\psi$ is a left ideal we have that 
\begin{equation}
\label{psipsi}
\psi(a b) = \psi(a \psi(b)) \qquad \mbox{for all $a,b\in \Upoly$.} 
\end{equation}
Recall that $B_\bc$ is a subalgebra of $\Upoly$. For quantized enveloping algebras the following Lemma recently appeared in \cite[Corollary 4.4]{a-Letzter17p}.
%%%%%%%%%%%%%%%%%%%%%%%%%%%%%%%%%%%%%%%%
\begin{lem}\label{lem:psi-iso1}
  The restriction of the map \eqref{eq:psi-poly} to $B_\bc$ is a $\field$-linear isomorphism
  \begin{align}\label{eq:psi-B}
    \psi:B_\bc\rightarrow H_\theta \ltimes U^-.    
  \end{align}
  \end{lem}  
%%%%%%%%%%%%%%%%%%%%%%%%%%%%%%%%%%%%%%%%
\begin{proof}
  For any multi-index $J$ and any $a\in H_\theta$ we have the relation
  \begin{align}\label{eq:psiBJ}
     \psi(aB_J)- a F_J\in H_\theta U^-_{\le |J|-1}.
  \end{align}
  This shows that the restriction \eqref{eq:psi-B} is surjective. On the other hand Corollary \ref{cor:basis-Bc} of Theorem \ref{thm:c-cond} implies that the restriction \eqref{eq:psi-B} is also injective. 
\end{proof}  
%%%%%%%%%%%%%%%%%%%%%%%%%%%%%%%%%%%%%%  
\begin{rema}\label{rem:phi-psi}
  The statement that the map $\psi$ in \eqref{eq:psi-B} is a linear isomorphism is equivalent to any of the statements in Theorem \ref{thm:c-cond} or Remark \ref{rem:Iwasawa}. Indeed, if say condition \eqref{assume:parameters} in Section \ref{sec:prop-c} does not hold, then the second part of Theorem \ref{thm:c-cond} implies that $B_\bc$ intersects nontrivially with the second summand of the decomposition \eqref{eq:Upoly-decomp}.
\end{rema}
%%%%%%%%%%%%%%%%%%%%%%%%%%%%%%%%%%%%%%%%
We use the isomorphism \eqref{eq:psi-B} to define an algebra structure $\ast$ on $H_\theta \ltimes U^-$ by
\begin{align}\label{eq:ast-def}
a  * b = \psi ( \psi^{-1}(a) \psi^{-1}(b)) \qquad \mbox{for all $a,b \in H_\theta \ltimes U^-$.}
\end{align}
Relation \eqref{eq:psiBJ} and Corollary \ref{cor:basis-Bc} imply that $*$ is a star product on the partial bosonization $H_\theta \ltimes U^-$ with the $\N$-grading defined by setting $\deg(h) =0$ and $\deg(F_i) =1$ for all $h \in H_\theta$, $i \in I$. Moreover, this star product is $0$-equivariant because $\psi$ is a left and right $H_\theta$-module homomorphism. The subalgebra $U^-\subset H_\theta\ltimes U^-$ satisfies the assumption of Lemma \ref{lem:sprod-isom}(ii). Hence, in view of $U^-_1=V^-(\chi)$, the $0$-equivariant star product is uniquely determined by a $\field$-linear map $\mu^L:V^-(\chi)\rightarrow \Hom_k(U^-,H_\theta\ltimes U^-)$. We summarize the situation in the following theorem.
%%%%%%%%%%%%%%%%%%%%%%%%%%%%%%%%%%%%%%%
\begin{thm}\label{thm:first-star}
  Let $U^+$ be a pre-Nichols algebra of diagonal type and assume that the parameters  $\bc \in \field^n$ satisfy condition \eqref{assume:parameters} in Section \ref{sec:prop-c}. Then the algebra structure $*$ on $H_\theta\ltimes U^-$ defined by \eqref{eq:ast-def} is a $0$-equivariant star product and the associated $\field$-linear map
  \begin{align*}
    \mu^L:V^-(\chi) \rightarrow \Hom_\field(U^-,H_\theta\ltimes U^-), \quad f\mapsto \mu^L_f
  \end{align*}
  from Lemma \ref{lem:sprod-isom}(ii) is given by
  \begin{align}\label{eq:mu-Nichols}
     \mu^L_{F_i}(u) = c_i q_{i\tau(i)}(K_{\tau(i)} K_i^{-1}) \partial_{\tau(i)}^L(u) 
  \end{align}
for all $i\in I$, $u \in U^-$.
\end{thm}
%%%%%%%%%%%%%%%%%%%%%%%%%%%%%%%%%%%%%%%%
\begin{proof}
  It remains to compute the map $\mu^L$. For any $b\in B_\bc$ and any $i\in I$ relation \eqref{psipsi} implies that
  \begin{align*}
    F_i\ast \psi(b)&=\psi(B_ib)\\
    &=\psi((F_i+c_i E_{\tau(i)}K_i^{-1})\psi(b))\\
    &=F_i \psi(b) + c_iq_{i\tau(i)} \psi(K_i^{-1}E_{\tau(i)}\psi(b))\\
    &=F_i \psi(b) + c_iq_{i\tau(i)} \psi(K_i^{-1} [E_{\tau(i)},\psi(b)]).
  \end{align*}
  Hence we get for any $u\in U^-$ the relation
  \begin{align*}
      F_i \ast u = F_i u +  c_iq_{i\tau(i)} \psi(K_i^{-1} [E_{\tau(i)},u])
  \end{align*}
  which by Equation \eqref{eq:partial-def} and the definition of $\psi$ implies that
  \begin{align}\label{eq:Fiastu}
       F_i \ast u = F_i u +  c_iq_{i\tau(i)} K_i^{-1}K_{\tau(i)} \partial^L_{\tau(i)}(u).
  \end{align}
  Hence $\mu^L_i$ is given by \eqref{eq:mu-Nichols}.
\end{proof}
%%%%%%%%%%%%%%%%%%%%%%%%%%%%%%%%%%%%%%%%%%%%%%%
\subsection{The second star product on the partial bosonization $H_\theta \ltimes U^-$}
\label{sec:second-star}
%%%%%%%%%%%%%%%%%%%%%%%%%%%
Next we interpret the associative product $\star$ from Section \ref{twist} in terms of star products on partial bosonizations. It follows from \eqref{eq:fstarg} and \eqref{eq:KfstarKg} that $\star$ is a $0$-equivariant star product on $H_\theta \ltimes U^-$. By Lemma \ref{lem:mu-map-star} the corresponding $\field$-linear map $\mu^L$ is also given by \eqref{eq:mu-Nichols}. We summarize these observations in the following proposition.
%%%%%%%%%%%%%%%%%%%%%%%%%%%%%%%%%%%%%%%%%%%%%%%%%%%%%%%%%%%%%%%%%%%%%%%%%%%%%%% 
\begin{prop} 
\label{prop:second-star}
For all pre-Nichols algebras of diagonal type $U^+$, the binary operation $\star$ on $H_\theta \ltimes U^-$ given by 
\eqref{eq:KfstarKg} is a 0-equivariant star product for which the map $\mu^L:
U^- \to \Hom_{\field}( U^- , H_\theta \ltimes U^-)$ from Lemma \ref{lem:sprod-isom}(ii) is given by
\[
\mu^L_{F_i}(u) = c_i q_{i \tau(i)} (K_{\tau(i)} K_i^{-1}) \partial_{\tau(i)}^L(u) 
\]
for all $i\in I$, $u \in U^-$.
\end{prop}
%%%%%%%%%%%%%%%%%%%%%%%%%%%%%%%%%%%%%%%%%%%%%%%%%%%%%%%%%%%
Combining the above proposition with Theorem \ref{thm:first-star} and using Lemma \ref{lem:sprod-isom}(ii) we obtain the following corollary.
%%%%%%%%%%%%%%%%%%%%%%%%%%%%%%%%%%%%%%%%%%%%%%%%%%%%%%%%%%%
\begin{cor}\label{cor:star=ast}
  For all pre-Nichols algebras of diagonal type $U^+$, the associative products $*$ and $\star$ on  $H_\theta \ltimes U^-$ coincide.
\end{cor}
%%%%%%%%%%%%%%%%%%%%%%%%%%%%%%%%%%%%%%%%%%%%%%%%%%%%%%%%%%%
Recall from Proposition \ref{prop:kowstar} that $(H_\theta \ltimes U^-,\star)$ is a right $U(\chi)_\mx$-comodule algebra with coaction $\kow_\star$. Composing the coproduct $\kow:B_\bc \rightarrow B_\bc \ot U(\chi)$ on $B_\bc$ with the projection $U(\chi)\rightarrow U(\chi)_\mx$ on the second tensor factor, one also obtains a $U(\chi)_\mx$-comodule algebra structure on $B_\bc$.
%%%%%%%%%%%%%%%%%%%%%%%%%%%%%%%%%%%%%%%%%%%%%%%%%%%%%%%%%%%%
\begin{cor}   \label{cor:psi-iso}
  For all pre-Nichols algebras of diagonal type $U^+$, the map 
\begin{align} \label{eq:psi-iso-star}
\psi : B_\bc \to (H_\theta \ltimes U^-, \star)
\end{align}
is an isomorphism of right $U(\chi)_\mx$-comodule algebras.
\end{cor}
%%%%%%%%%%%%%%%%%%%%%%%%%%%%%%%%%%%%%%%%%%%%%%%
\begin{proof}
  It follows from Lemma \ref{lem:psi-iso1} and the definition of the star product $\ast$ that $\psi:B_\bc\to (H_\theta \ltimes U^-, \ast)$ is an isomorphism of algebras. By Corollary \ref{cor:star=ast} the map \eqref{eq:psi-iso-star} is also an isomorphism of algebras. Moreover, by \eqref{eq:kowstarFi} the map \eqref{eq:psi-iso-star} respects the right $U(\chi)_\mx$-coaction.
\end{proof}  
%%%%%%%%%%%%%%%%%%%%%%%%%%%%%%%%%%%%%%%%%%%%%%%
\subsection{Generators and relations for $B_\bc$, revisited}\label{sec:RelsRev}
%%%%%%%%%%%%%%%%%%%%%%%%%%%%%%%%%%%%%%%%%%%%%%5
We can apply the constructions of Sections \ref{sec:first-star} and \ref{sec:second-star} in particular in the case where the biideal $\cI$ which defines $U^+, U^-$ and $U(\chi)$ is trivial, that is $\cI=\{0\}$. In this case we have $H_\theta\ltimes U^-=H_\theta\ltimes T(V^-(\chi))$. We write $\circledast$
to denote the star product $\ast$ on $H_\theta\ltimes T(V^-(\chi))$, and we write $\Btil_\bc$, $\Util(\chi)^\poly$, and $\wt{\psi}$ to denote $B_\bc$, $\Upoly$ and $\psi$, respectively, in the case $\cI=\{0\}$. For a general biideal $\cI\subset T(V^+(\chi))$ and parameters $\bc \in \field^n$ satisfying 
condition \eqref{assume:parameters} in Section \ref{sec:prop-c} we hence obtain a commutative diagram
\[
\begin{tikzpicture}[scale=1.7]
\node (A) at (0,1) {$\Btil_\bc$};
\node (B) at (1,1) {$\Util(\chi)^\poly$};
\node (C) at (3,1) {$(H_\theta \ltimes T(V^-(\chi)), \circledast)$};
\node (D) at (0,0) {$B_\bc$};
\node (E) at (1,0) {$\Upoly$};
\node (F) at (3,0) {$(H_\theta \ltimes U^-, *)$};
%\path[->>,font=\scriptsize,>=angle 90]
\draw
(A) edge[right hook ->,font=\scriptsize,>=angle 90] (B)
(B) edge[->>,font=\scriptsize,>=angle 90] node[above]{$\wt{\psi}$} (C)
(D) edge[right hook ->,font=\scriptsize,>=angle 90] (E)
(E) edge[->>,font=\scriptsize,>=angle 90] node[above]{$\psi$} (F)
(A) edge[->>,font=\scriptsize,>=angle 90] node[right]{$\eta$}(D)
(B) edge[->>,font=\scriptsize,>=angle 90] node[right]{$\eta$} (E)
(C) edge[->>,font=\scriptsize,>=angle 90] node[right]{$\eta$} (F);
\draw[->] (A) to [bend left] node[scale=.7] (b) [above] {$\wt{b}$} (C);
\draw[->] (D) to [bend right] node[scale=.7] (f) [below] {$b$} (F);
\end{tikzpicture}
\]
where $b=\psi|_{B_\bc}$ and $\wt{b}=\wt{\psi}|_{\Btil_\bc}$. In the above diagram the vertical arrows are surjective algebra homomorphisms.  The rightmost vertical arrow is a homomorphism both 
of the undeformed partial bosonizations $H_\theta \ltimes T(V^-(\chi)) \to H_\theta \ltimes U^-$ and of the transferred algebra structures
$(H_\theta \ltimes T(V^-(\chi)), \circledast) \to (H_\theta \ltimes U^-, *)$. The maps $\psi$ and $\wt{\psi}$ are $\field$-linear maps, while the 
other two horizontal maps are algebra embeddings. The maps $b$ and $\wt{b}$ are algebra isomorphisms. 

The following proposition provides a procedure to determine the defining relations of $(H_\theta \ltimes U^-,\ast)$ from the defining relations of $U^-$.
%%%%%%%%%%%%%%%%%%%%%%%%%%%%%%%%%%%%%%%%%%%%%%
\begin{prop} \label{prop:def-rels}
Let $U^+$ be a  pre-Nichols algebras of diagonal type and assume that the parameters $\bc \in \field^n$ satisfy condition \eqref{assume:parameters} in Section \ref{sec:prop-c}. If $S$ is a generating set for the kernel of the homomorphism $\eta: T(V^-(\chi)) \to U^-$ for the undeformed algebra structures, then it is a generating set also for the kernel of the homomorphism 
\[
\eta : (H_\theta \ltimes T(V^-(\chi)), \circledast) \to (H_\theta \ltimes U^-, *)
\]
with respect to the transferred algebra structures.
\end{prop}
%%%%%%%%%%%%%%%%%%%%%%%%%%%%%%%%%%%%%%%%%%%%%%%%%%%%%%%%
\begin{proof}
Consider the projection $\eta : H_\theta \ltimes T(V^-(\chi)) \to H_\theta \ltimes U^-$. By the definition of $S$ we have 
\[
\ker(\eta) = (H_\theta \ltimes T(V^-(\chi))) \cdot S \cdot (H_\theta \ltimes T(V^-(\chi))).
\]
We need to prove that 
\begin{align}\label{eq:rels-goal}
\ker(\eta)= (H_\theta \ltimes T(V^-(\chi))) \circledast S \circledast (H_\theta \ltimes T(V^-(\chi))).
\end{align}
As $\eta:(H_\theta\ltimes T(V^-(\chi),\circledast)\rightarrow (H_\theta\ltimes U^-,\ast)$ is an algebra homomorphism, the right hand side of \eqref{eq:rels-goal} is contained in $\ker(\eta)$. The map $\eta$ is graded and we show by induction on $j \in \N$ that 
\[
\ker(\eta)_j  \subseteq (H_\theta \ltimes T(V^-(\chi))) \circledast S \circledast (H_\theta \ltimes T(V^-(\chi))).
\]
Indeed, for $a \in \ker(\eta)_{j+1}$ there exist homogeneous elements $b'_l, b''_l \in H_\theta \ltimes T(V^-(\chi))$, $s_l \in \cS$ such that
$a = \sum_l a'_l s_l b''_l$. Property \eqref{eq:a*b-ab} of the star product implies that
\[
a - \sum_l a'_l \circledast s_l \circledast b''_l \in \ker(\eta)_{\leq j}, 
\]
and by induction hypothesis we have
\[
\ker(\eta)_{\leq j}  \subseteq (H_\theta \ltimes T(V^-(\chi))) \circledast  S \circledast (H_\theta \ltimes T(V^-(\chi))).
\]
This shows that  $\ker(\eta)_{j+1}\subseteq (H_\theta \ltimes T(V^-(\chi))) \circledast  S \circledast (H_\theta \ltimes T(V^-(\chi)))$ and hence completes the proof of \eqref{eq:rels-goal}.
\end{proof}
%%%%%%%%%%%%%%%%%%%%%%%%%%%%%%%%%%%%%%%%%%%%%%%%%%%%%%%%%%%%%%%%5
For any noncommutative polynomial $r(x_1,\dots,x_n)=\sum_J a_J x_{j_1}\dots x_{j_n}$ in $n$ variables with coefficients $a_J\in H_\theta$ and any elements $u_1,\dots, u_n$ in $H_\theta\ltimes T(V^-(\chi))$ we write
\begin{align*}
  r(u_1 \stackrel{\circledast}{,}\dots,  \stackrel{\circledast}{,}u_n)=\sum_J a_J u_{j_1}\circledast \dots \circledast u_{j_n}. 
\end{align*}  
Proposition \ref{prop:def-rels} has the following immediate corollary giving an effective way to determine the relations of the coideal subalgebra 
$B_\bc$ of $U(\chi)$.
%(for an arbitrary pre-Nichols algebra  $U^+$ of diagonal type and $\bc \in \field^n$ satisfying
%condition \eqref{assume:parameters} in \S \ref{sec:prop-c}).

\medskip

\noindent{\bf Procedure for determining the relations of $B_\bc$ }:
\begin{enumerate}
\item Let $S=\{ p_m(x_1, \ldots, x_m) \mid m \in \mathcal{S} \}$  be a set of homogeneous noncommutative polynomials such that $\{p_m(\uE)\mid m\in \cS\}$ generates the kernel of the projection $\eta : T(V^+(\chi)) \to U^+$. In other words, $S$ provides the defining relations of $U^+$. Let $d_m$ denote the degree of the polynomial $p_m$ for all $m\in \cS$.
\item Let 
\[
r_m(x_1, \ldots, x_n) = \sum_J a_J x_{j_1} \ldots x_{j_l}  
\]
be the noncommutative polynomials with coefficients in $a_J \in H_\theta$ such that 
\[
p_m(F_1, \ldots, F_n) = r(F_1 \stackrel{\circledast}{,}\dots,  \stackrel{\circledast}{,}F_n) 
\]
where the left hand side uses the undeformed product in $T(V^-(\chi))$.
It follows from \eqref{eq:a*b-ab} that $r_m$ has degree $d_m$ and leading term $p_m$. 
\item The algebra $B_\bc$ is generated by $H_\theta$ and $B_i$ for $i \in I$ subject to the relations
\begin{align}
&&K_\lambda B_i &= \chi(\lambda, \alpha_i)^{-1} B_i K_ \lambda
& &\mbox{for all $\lambda \in \Z_\theta^n$, $i \in I$,}\label{eq:KlBBKl}\\
&&r_m(\uB)&=0 &&\mbox{for all $m \in \mathcal{S}$.} \nonumber
\end{align}
\end{enumerate}
%%%%%%%%%%%%%%%%%%%%%%%%%%%%%%%%%%%%%%%%%%%%%%%%%%%
\begin{eg}\label{eg:Uzsl3}
  Consider the quantized universal enveloping algebra $U_\zeta(\slfrak_3)$ for $\zeta\in \field^\times$ as described in Section \ref{sec:large}. It has generators $E_i, F_i, K_i^{\pm 1}$ for $i\in I=\{1,2\}$ and relations given by \eqref{eq:uqg-rels}. We apply the above procedure to the coideal subalgebra $B_\bc$ of $U_\zeta(\slfrak_3)$ corresponding to the bijection $\tau:I\rightarrow I$ given by $\tau(1)=2$, $\tau(2)=1$. The quantum Serre relations are given by $p_{12}(F_1,F_2)=p_{21}(F_1,F_2)=0$ where
  \begin{align*}
     p_{12}(x,y)= x^2y - (\zeta+\zeta^{-1}) xyx + y x^2, \qquad p_{21}(x,y)=p_{12}(y,x).
  \end{align*}
  Using relation \eqref{eq:Fiastu} one obtains
  \begin{align*}
    F_1 \circledast F_2&= F_1 F_2 + c_1 \zeta^{-1} K_2 K_1^{-1},&
    F_1 \circledast (F_1 F_2) &= F_1^2 F_2 + c_1 \zeta  F_1 K_2 K_1^{-1}
  \end{align*}
  and hence
  \begin{align}\label{eq:1*1*2}
      F_1^2 F_2 = F_1 \circledast F_1 \circledast F_2 - c_1 (\zeta+\zeta^{-1})F_1 K_2 K_1^{-1}.
  \end{align}
  Similarly one calculates
  \begin{align}
    F_1 F_2 F_1 &= F_1 \circledast F_2 \circledast F_1 - c_1 \zeta^2 F_1 K_2 K_1^{-1} - c_2 \zeta^{-1} F_1 K_1 K_2^{-1}\label{eq:1*2*1}\\
    F_2 F_1^2 &= F_2 \circledast F_1 \circledast F_1 - c_2(\zeta+\zeta^{-1}) \zeta^{-3} F_1 K_1 K_2^{-1}. \label{eq:2*1*1}
  \end{align}
  Combining \eqref{eq:1*1*2}, \eqref{eq:1*2*1} and \eqref{eq:2*1*1} one obtains
  \begin{align*}
     p_{12}(F_1,F_2) = p_{12}(F_1 \stackrel{\circledast}{,} F_2) + (\zeta^2-\zeta^{-2})F_1[c_1 \zeta K_2 K_1^{-1} + c_2 \zeta^{-2} K_1 K_2^{-1}].
  \end{align*}
  Hence the noncommutative polynomial $r_{12}(x,y)$ describing the corresponding defining relation of the coideal subalgebra $B_\bc$ is given by
  \begin{align*}
    r_{12}(x,y)=p_{12}(x,y)+(\zeta^2-\zeta^{-2})[c_1 \zeta^{-2}a K_2 K_1^{-1} + c_2 \zeta K_1 K_2^{-1}]x.
  \end{align*}
  Similarly one obtains
  \begin{align*}
    r_{21}(x,y)=p_{21}(x,y)+(\zeta^2-\zeta^{-2})[c_2 \zeta^{-2}a K_1 K_2^{-1} + c_1 \zeta K_2 K_1^{-1}]y.
  \end{align*}
  By the above procedure the algebra $B_\bc$ is generated by $B_1, B_2$ and $H_\theta$ subject to the relations \eqref{eq:KlBBKl} and $r_{12}(B_1,B_2)=r_{21}(B_1,B_2)=0$. The latter two relations coincide with the relations given in \cite[Theorem 7.1 (iv)]{a-Letzter03}.
\end{eg}  
%%%%%%%%%%%%%%%%%%%%%%%%%%%%%%%%%%%%%%%%%%%%%%%%%%%
\begin{rema}\label{rem:Letzter-relations}
  For quantum symmetric pair coideal subalgebras a different method to determine defining relations was devised by G.~Letzter in \cite[Theorem 7.1]{a-Letzter03}, see also \cite[Section 7]{a-Kolb14}. This method also works in the general setting of the present paper. Letzter's method relies on relation \eqref{eq:kowZ} which holds with $Z=0$ by choice of parameters. With Letzter's method individual monomials in the quantum Serre relations lead to completely different lower order terms in the relations for $B_\bc$ than with the procedure described above. This shows that the procedure described above is not a mere reformulation of Letzter's method.  
\end{rema}  
%%%%%%%%%%%%%%%%%%%%%%%%%%%%%%%%%%%%%%%%%%%%%%%%%%%
\begin{eg} As a second example we consider the coideal subalgebra $B_\bc$ of the Drinfeld double of the distinguished pre-Nichols algebra of type $\ufofrak(8)$ from Section \ref{sec:ufo}.
The algebra  $B_\bc$ has generators $K^{\pm 1}, B_1, B_2$ where
\[
K = K_1 K_2^{-1}, \quad B_1=  F_1+c_1 E_2 K_1^{-1}, \quad
B_2=F_2+c_2 E_1 K_2^{-1}.
\]
Calculating recursively as in Example \ref{eg:Uzsl3} on obtains that
\[
F_i^m = F_i^{\circledast m} \quad \mbox{for all} \quad \mbox{ $i=1,2$ and $m \in \N$},
\]
and that for the polynomial $p(x_1,x_2)$ from \eqref{p-poly} one has $p(F_1, F_2) = r(F_1 \stackrel{\circledast}{,} F_2)$ where 
\begin{align*}
r(x_1, x_2) =& p(x_1, x_2) - (3 \zeta + 2) (c_1 K^{-1} x_1 x_2 + c_2 K x_2 x_1) 
\\
&+ \zeta^{-1/2}( 2 \zeta+3) (c_1 K^{-1} x_2 x_1 + c_2 K x_1 x_2) 
\\
&+ \zeta^{-1/2} (\zeta^2 + \zeta + 1) (c_1^2 K^{-2} + c_2^2 K^2) - 2 ( \zeta+1) c_1 c_2.
\end{align*}
Assume that the parameters $c_1, c_2 \in \field$ satisfy the relation in Proposition \ref{prop:ufo8}(ii). By the above procedure, the algebra $B_\bc$ has generators $K^{\pm 1}, B_1, B_2$ and relations  
\begin{equation*}
\begin{aligned}
  K K^{-1} &= 1,\\
  K B_1 = - \zeta^{ - 3/2} B_1 K,& \qquad K B_2 = - \zeta^{ 3/2} B_2 K,\\
  B_1^3 = B_2^3 =0,& \qquad r(B_1, B_2) =0.
\end{aligned}
\end{equation*}
We have checked the above relations also with Letzter's method referred to in Remark \ref{rem:Letzter-relations}, and this produces the same relation $r(B_1,B_2)=0$.
\end{eg}
%%%%%%%%%%%%%%%%%%%%%%%%%%%%%%%%%%%%%%%%%%%%%%%%%%%%%%%%%%%%%%%%%%%%%%%%%%%%%%%%%%%%
\section{The quasi $K$-matrix for $B_\bc$}\label{sec:quasiK}
%%%%%%%%%%%%%%%%%%%%%%%%%%%%%%%%%%%%%%%%%%%%%%%%%%%%%%%%%
%%%%%%%%%%%%%%%%%%%%%%%%%%%%%%%%%%%%%%%%%%%%%%%%%%%%%%%%%
From now on we restrict to the case where the graded biideal is maximal $\cI=\cI_\mx$ and hence $U^\pm=U^\pm_\mx$ are Nichols algebras. We also retain the assumption that $\bc\in \field^n$ satisfies condition  \eqref{assume:parameters} in Section \ref{sec:prop-c}.
Recall the isomorphism of $U(\chi)_\mx$-comodule algebras $\psi:B_\bc\rightarrow (H_\theta \ltimes U_\mx^-, \star, \kow_\star)$ from Corollary \ref{cor:psi-iso} 
and the quasi $R$-matrix $\Theta=\sum_\mu (-1)^{|\mu|} F_\mu\ot E_\mu$ from Section \ref{sec:quasiR}. We call the formal sum
\begin{align}\label{eq:quasiK}
  \Theta^\theta=(\psi^{-1}\ot \id)(\Theta)= \sum_\mu (-1)^{|\mu|} \psi^{-1}(F_\mu) \ot E_\mu \in \prod_\mu B_\bc \ot (U^+_\mx)_\mu
\end{align}  
the quasi $K$-matrix for $B_\bc$. Here we consider the infinite product $\prod_\mu B_\bc \ot (U^+_\mx)_\mu$ as a subalgebra of the completion $U(\chi)_\mx\widehat{\ot}U(\chi)_\mx$ from Section \ref{sec:quasiR}. We multiply elements in $\prod_\mu B_\bc \ot (U^+_\mx)_\mu$ as infinite sums.
%%%%%%%%%%%%%%%%%%%%%%%%%%%%%%%%%%%%%%%%%%%%%%%%%%%%%%%%%%
\subsection{The coproducts of the quasi $K$-matrix}
%%%%%%%%%%%%%%%%%%%%%%%%%%%%%%%%%%%%%%%%%%%%%%%%%%%%%%%%%%
Similarly to Lemma \ref{lem:kowidTheta} we are interested in the behavior of $\Theta^\theta$ under the coproduct of $U(\chi)_\mx$ in each tensor factor. To this end we introduce elements
\begin{align*}
  \Theta^\theta_{12}&=\Theta^\theta\ot 1,\qquad  \Theta_{23}=1 \ot \Theta,\\
  \Theta^\theta_{1K3}&= \sum_\mu (-1)^{|\mu|}\psi^{-1}(F_\mu)\ot K_\mu \ot E_\mu,\\
  \Theta^{\theta-}_{1K3}&=\sum_\mu (-1)^{|\mu|}\psi^{-1}(F_\mu)\ot K_\mu^{-1} \ot E_\mu,\\
  \Theta^{\osigma}_{K23}&= \sum_{\mu} (-1)^{|\mu|} K_{\mu-\tau(\mu)}\ot \osigma(F_\mu) \ot E_\mu,\\
  \Theta^{\osigma K}_{K32}  &=\sum_{\mu} (-1)^{|\mu|} K_{\mu-\tau(\mu)} \ot E_\mu K_{\tau(\mu)}^{-1} \ot K_\mu^{-1}\osigma(F_\mu)
\end{align*}  
in $\prod_{\mu,\nu} B_\bc\ot H(U^+_\mx)_\mu \ot (U^+_\mx)_\nu$. As before, we multiply elements in $\prod_{\mu,\nu} B_\bc\ot H(U^+_\mx)_\mu \ot (U^+_\mx)_\nu$ infinite sums. A formal completion of $B_\bc\ot U(\chi)_\mx \ot U(\chi)_\mx$ containing the above product will be given in Section \ref{sec:w-quasi-Hopf}.   With the above notation we can express the desired analog of Lemma \ref{lem:kowidTheta}.
%%%%%%%%%%%%%%%%%%%%%%%%%%%%%%%%%%%%%%%%%%%%%%%%%%%%%
\begin{prop}\label{prop:kowidThetatheta}
  The quasi $K$-matrix $\Theta^\theta$ satisfies the relation
  \begin{align}\label{eq:kowidThetatheta}
    (\id \ot \kow)(\Theta^\theta) &= \Theta^\theta_{12}\cdot \Theta^{\osigma}_{K23}\cdot \Theta^\theta_{1K3}
  \end{align}
   in $\prod_{\mu,\nu} B_\bc\ot H(U^+_\mx)_\mu \ot (U^+_\mx)_\nu$, and the relation
  \begin{align}  
     (\kow \ot \id)(\Theta^\theta) &= \Theta_{23} \cdot \Theta^{\theta-}_{1K3} \cdot \Theta^{\osigma K}_{K32}\label{eq:idkowThetatheta}
  \end{align}
 in  $\prod_{\mu} B_\bc\ot U(\chi)_\mx \ot (U^+_\mx)_\mu$.
\end{prop}
%%%%%%%%%%%%%%%%%%%%%%%%%%%%%%%%%%%%%%%%%%%%%%%%%%%%%%%
\begin{proof}
  To prove Equation \eqref{eq:kowidThetatheta} first observe that \eqref{eq:lactract} and \eqref{eq:kow-Theta1} imply that
  \begin{align}\label{eq:idkow1}
    \sum_{\mu,\nu} (-1)^{|\nu|}F_\mu \ot (&F_\nu\ract E_\mu)\ot K_\nu \ot E_\nu\\
    &=\sum_{\mu,\lambda} (-1)^{|\mu|+|\lambda|} F_\mu\ot F_\lambda \ot K_{\lambda+\mu} \ot E_\mu E_\lambda.\nonumber
  \end{align}
  Similarly, also taking into account \eqref{eq:amu-def}, one obtains
  \begin{align}\label{eq:idkow2}
    \sum_{\mu,\lambda} (-1)^{|\mu|+|\lambda|}(\osigma(F_\mu)&\lact F_\lambda)K_\mu \ot E_\mu \ot E_\lambda\\
    &=\sum_{\mu,\kappa}(-1)^{|\kappa|} F_\kappa K_{\mu-\tau(\mu)}\ot E_\mu \ot E_\kappa \osigma(F_\mu)K_\mu^{-1}.\nonumber
  \end{align}
  With this preparation we use  Equation \eqref{eq:kow-Theta2}, Lemma \ref{lem:fg}, and the fact that $\psi$ is an isomorphism of algebras, to calculate
  \begin{align*}
    (\id &\ot \kow)(\Theta^\theta) \stackrel{\phantom{\eqref{eq:idkow2}}}{=}\sum_{\lambda,\nu}   (-1)^{|\lambda|+|\nu|}\psi^{-1}(F_\lambda F_\nu) \ot E_\lambda K_\nu\ot E_\nu\\
    &\stackrel{\phantom{\eqref{eq:idkow2}}}{=}\sum_{\lambda,\mu,\nu} (-1)^{|\lambda| + |\mu|+|\nu|} \psi^{-1}\big((\osigma(F_\mu)\lact F_\lambda)K_\mu\big) \psi^{-1}\big( F_\nu \ract E_\mu\big)\ot E_\lambda K_\nu\ot E_\nu\\
    &\stackrel{\eqref{eq:idkow2}}{=}\sum_{\mu,\nu,\kappa} (-1)^{|\nu|+|\kappa|}
    \psi^{-1}\big(F_\kappa K_{\mu-\tau(\mu)}\big) \psi^{-1}\big(F_\nu\ract E_\mu\big) \ot E_\kappa \osigma(F_\mu)K_\mu^{-1}K_\nu \ot E_\nu\\
    &\stackrel{\eqref{eq:idkow1}}{=}\sum_{\mu,\kappa,\lambda}(-1)^{|\lambda| + |\mu|+|\kappa|}
    \psi^{-1}(F_\kappa) K_{\mu-\tau(\mu)} \psi^{-1}(F_\lambda)\ot E_\kappa \osigma(F_\mu)K_\lambda \ot E_\nu E_\lambda\\
    &\stackrel{\phantom{\eqref{eq:idkow2}}}{=} \Theta^\theta_{12}\cdot \Theta^{\osigma}_{K23}\cdot \Theta^\theta_{1K3}
  \end{align*}
  which proves Equation \eqref{eq:kowidThetatheta}. Equations \eqref{eq:kowB} and \eqref{eq:kowstarFi}, the fact that $B_\bc$ is a coideal subalgebra 
  of $U(\chi)_\mx$ and Proposition \ref{prop:kowstar} imply that $\psi$ is an isomorphism of $U(\chi)_\mx$-comodules. Therefore 
  \begin{align*}
    (\kow &\ot \id)(\Theta^\theta) \stackrel{\phantom{\eqref{eq:idkow1}}}{=} \sum_{\nu}(-1)^{|\nu|}(\psi^{-1}\ot \id\ot \id)(\kow_\star(F_\nu)\ot E_\nu)\\
    &\stackrel{\eqref{eq:kowstar-def}}{=} \sum_{\lambda,\mu,\nu}(-1)^{|\nu|}\psi^{-1}\big((\osigma(F_\lambda)\lact F_\nu \ract E_\mu)K_\lambda\big) \ot F_\mu K_{\mu-\nu} E_\lambda \ot E_\nu\\
    &\stackrel{\eqref{eq:idkow1}}{=} \sum_{\lambda,\rho,\mu}(-1)^{|\mu|+|\rho|} \psi^{-1}\big((\osigma(F_\lambda)\lact F_\rho)K_\lambda)\ot F_\mu K_\rho E_\lambda \ot E_\mu E_\rho\\
    &\stackrel{\eqref{eq:idkow2}}{=} \sum_{\lambda,\kappa,\mu}(-1)^{|\mu|+|\kappa|+|\lambda|} \psi^{-1}\big(F_\kappa)K_{\lambda-\tau(\lambda)}\ot F_\mu K_{-\kappa-\tau(\lambda)} E_\lambda \ot E_\mu E_\kappa \osigma(F_\lambda)K_\lambda^{-1}\\
    &\stackrel{\phantom{\eqref{eq:idkow1}}}{=}
    \Theta_{23} \cdot \Theta^{\theta-}_{1K3} \cdot \Theta^{\osigma K}_{K32}
  \end{align*}
  which proves Equation \eqref{eq:idkowThetatheta}.
\end{proof}
%%%%%%%%%%%%%%%%%%%%%%%%%%%%%%%%%%%%%%%%%%%%%%%%%%%%%%%%%%%%%
\subsection{The intertwiner property of the quasi $K$-matrix}
%%%%%%%%%%%%%%%%%%%%%%%%%%%%%%%%%%%%%%%%%%%%%%%%%%%%%%%%%%%%%
The quasi $K$-matrix $\Theta^\theta$ also satisfies an analog of Corollary \ref{cor:EFTheta}.
%%%%%%%%%%%%%%%%%%%%%%%%%%%%%%%%%%%%%%%%5
\begin{prop}\label{prop:quasiK-intertwiner}
  The element $\Theta^\theta$ satisfies the relations
  \begin{align}
    \kow(B_i) \cdot \Theta^\theta &= \Theta^\theta \cdot \big(B_i\ot K_i + c_{\tau(i)} q_{i\tau(i)} K_{\tau(i)}^{-1}K_i\ot E_{\tau(i)}K_i + 1\ot F_i\big),\label{eq:quasiK-comm1}\\
    \kow(K_\lambda)\cdot \Theta^\theta &= \Theta^\theta\cdot \kow(K_\lambda)\label{eq:quasiK-comm2}
  \end{align}
  for all $i\in I$, $\lambda\in \Z^n_\theta$.
\end{prop}  
%%%%%%%%%%%%%%%%%%%%%%%%%%%%%%%%%%%%%%%%%
\begin{proof}
  We rewrite $(F_i\ot K_i^{-1})\cdot \Theta$ in terms of the twisted product
  \begin{align*}
    &(F_i\ot K_i^{-1})\cdot \Theta\\
    &\stackrel{\eqref{eq:Fistarg}}{=} \sum_\mu(-1^{|\mu|})(F_i\star F_\mu)\ot K_i^{-1}E_\mu - c_i q_{i\tau(i)}  \sum_\mu(-1)^{|\mu|} K_{\tau(i)} K_i^{-1}\partial^L_{\tau(i)}(F_\mu)\ot K_i^{-1} E_\mu\\
    &\stackrel{\eqref{eq:Theta-partial}}{=}\sum_\mu(-1)^{|\mu|}(F_i\star F_\mu)\ot K_i^{-1}E_\mu + c_i  (K_{\tau(i)} K_i^{-1}\ot E_{\tau(i)}K_i^{-1}) \cdot \Theta.
  \end{align*}
  Similarly we rewrite $\Theta \cdot (F_i\ot K_i)$ in terms of the twisted product
  \begin{align*}
    &\Theta\cdot (F_i\ot K_i)\\
    &\stackrel{\eqref{eq:fstarFi}}{=} \sum_\mu(-1^{|\mu|})(F_\mu\star F_i)\ot E_\mu K_i - c_{\tau(i)}q_{i\tau(i)}\sum_{\mu}(-1)^{|\mu|}\partial^R_{\tau(i)}(F_\mu)K_iK_{\tau(i)}^{-1} \ot E_\mu  K_i)\\
    &\stackrel{\eqref{eq:Theta-partial}}{=}\sum_\mu(-1)^{|\mu|}(F_\mu\star F_i)\ot E_\mu K_i + c_{\tau(i)}q_{i\tau(i)} \Theta \cdot(K_i K_{\tau(i)}^{-1}\ot E_{\tau(i)}K_i).
  \end{align*}
  Now Equation \eqref{eq:quasiK-comm1} follows from the above two relations, and the fact that $\psi^{-1}:(H_\theta \ltimes U^-,\star)\rightarrow B_\bc$ is an algebra isomorphism, by application of $\psi^{-1}$ to the first tensor factor of Equation \eqref{eq:EFTheta2}. Similarly, Equation \eqref{eq:quasiK-comm2} follows from the relation $\kow(K_\lambda)\cdot \Theta=\Theta \cdot \kow(K_\lambda)$ by application of $\psi^{-1}$ to the first tensor factor. 
\end{proof}
%%%%%%%%%%%%%%%%%%%%%%%%%%%%%%%%%%%%%%%%%%%%%%%%%%%%%%%%%%%%%%%%%%%%%%%%%%%
\begin{rema}
  The statement of Proposition \ref{prop:quasiK-intertwiner} is known in the theory of quantum symmetric pairs as the intertwiner property for the quasi $K$-matrix (called quasi $R$-matrix in \cite[Section 3]{a-BaoWang18a}). In \cite[Proposition 3.2]{a-BaoWang18a} and \cite[Proposition 3.5]{a-Kolb17p} this property is formulated in terms of the bar-involution for quantum symmetric pair coideal subalgebras. For general Nichols algebras and their coideal subalgebras there is no bar-involution. Proposition \ref{prop:quasiK-intertwiner} achieves a bar-involution free formulation of the intertwiner property in the same way as Corollary \ref{cor:EFTheta} provides a bar-involution free formulation of the intertwiner property for the quasi $R$-matrix. 
\end{rema}  
%%%%%%%%%%%%%%%%%%%%%%%%%%%%%%%%%%%%%%%%%%%%%%%%%%%%%%%%%%%%%%%%%%%%%%%%%%%%%%%%
\subsection{Weakly quasitriangular Hopf algebras}\label{sec:w-quasi-Hopf}
%%%%%%%%%%%%%%%%%%%%%%%%%%%%%%%%%%%%%%%%%%%%%%%%%%%%%%%%%%%%%
We now want to show that the quasi $K$-matrix \eqref{eq:quasiK} gives rise to a universal $K$-matrix for the coideal subalgebra $B_\bc$ of $U(\chi)_\mx$. In \cite{a-BalaKolb15p} and \cite{a-Kolb17p} universal $K$-matrices are constructed on suitable categories of representations. Due to the generality of our setting we do not know much about the representation theory of $U(\chi)_\mx$. Instead we follow an approach used in \cite{a-Tanisaki92}, \cite{a-Reshetikhin95}, \cite{a-Gav97} and consider a weak notion of quasitriangularity. In the present section we recall this approach. In Section \ref{sec:wq-coid} we introduce the corresponding notion of weakly quasitriangular coideal subagebras and show that $B_\bc$ is weakly quasitriangular up to completion.
%%%%%%%%%%%%%%%%%%%%%%%%%%%%%%%%%%%%%%%%%%%%%%%%%%%%%%55
\begin{defi}\label{def:w-quasi-Hopf} {\upshape(\cite[Definition 3]{a-Reshetikhin95}, \cite[Definition 1.2]{a-Gav97})}
  A weakly quasitriangular Hopf algebra is a pair $(U,\cR)$ consisting of a Hopf algebra $U$ and an algebra automorphism $\cR\in \Aut(U^{\ot 2})$ satisfying the relations
\begin{gather}
\cR \circ \Delta = \Delta^{\mathrm{op}} \quad \mbox{on} \quad U \label{eq:wq1},\\
(\Delta \otimes \id) \circ \cR = \cR_{13} \circ \cR_{23} \circ (\Delta \otimes \id) \quad \mbox{on $U^{\otimes 2}$}, \label{eq:wq2} \\
(\id \otimes \Delta) \circ \cR = \cR_{13} \circ \cR_{12} \circ (\id \otimes \Delta) \quad \mbox{on $U^{\ot 2}$.} \label{eq:wq3}
\end{gather}
Here we use the usual leg notation where $\cR_{ij}$ denotes the operation of $\cR$ on the $i$-th and $j$-th tensor factor.
\end{defi}
%%%%%%%%%%%%%%%%%%%%%%%%%%%%%%%%%%%%%%%%%%%%%%%%%%%%%%%%
For any invertible element $u$ of a unital algebra $A$ let $\Ad(u)$ denote the inner automorphism of $A$ defined by
\[
  \Ad(u)(a) = u a u^{-1} \qquad \mbox{for all $a\in A$.}
  \]
%%%%%%%%%%%%%%%%%%%%%%%%%%%%%%%%%%%%%%%%%%%%%%%%%%%%%%%%%%  
\begin{rema}\label{rem:qwq}
  Recall the notion of a quasitriangular Hopf algebra from \cite{inp-Drinfeld1}. If $U$ is a quasitriangular Hopf algebra with universal $R$-matrix $R$, then $U$ is weakly quasitriangular with the automorphism $\cR$ defined by conjugation $\cR=\Ad(R)$.
\end{rema}
%%%%%%%%%%%%%%%%%%%%%%%%%%%%%%%%%%%%%%%%%%%%%%%%%%%%%%%%
\begin{rema}\label{rem:qYB}
  By \cite[(7)]{a-Reshetikhin95} the automorphism $\cR$ of a weakly quasitriangular Hopf algebra satisfies the quantum Yang-Baxter equation
  \begin{align}\label{eq:qYB}
     \cR_{12}\circ \cR_{13}\circ \cR_{23} = \cR_{23}\circ \cR_{13}\circ \cR_{12}.
  \end{align}
  Indeed, \eqref{eq:wq1} and \eqref{eq:wq2} imply that both sides of \eqref{eq:qYB} coincide on the image of $\kow\ot \id$, while  \eqref{eq:wq1} and \eqref{eq:wq3} imply that both sides of \eqref{eq:qYB} coincide on the image of $\id \ot\kow$. Now the quantum Yang-Baxter equation \eqref{eq:qYB} follows from the fact that if $U$ is a Hopf algebra then $\mathrm{Im}(\kow\ot \id)+\mathrm{Im}(\id \ot \kow)$ generates $U^{\ot 3}$ as an algebra.
\end{rema}
%%%%%%%%%%%%%%%%%%%%%%%%%%%%%%%%%%%%%%%%%%%%%%%%%%%%%%%%
\begin{rema}
  In \cite{a-Reshetikhin95} a weakly quasitriangular Hopf algebra is called a \textit{braided} Hopf algebra, see also \cite[Definition 1.2]{a-Gav97}. We avoid this terminology because it is often used for Hopf algebras in a braided category. In \cite[4.3]{a-Tanisaki92} weakly quasitriangular Hopf algebras are realized under the name \textit{pre-triangular} Hopf algebras via a construction similar to the following lemma. 
  \end{rema}  
%%%%%%%%%%%%%%%%%%%%%%%%%%%%%%%%%%%%%%%%%%%%%%%%%%%%%%%%

%%%%%%%%%%%%%%%%%%%%%%%%%%%%%%%%%%%%%%%%%%%%%%%%%%%%%
\begin{lem} \label{lem:weak-quasi}
  {\upshape (\cite[Definition 3]{a-Reshetikhin95}, \cite[Definition 1.3]{a-Gav97})}
Let  $U$ be a Hopf algebra, $\cR^{(0)}\in \Aut(U\ot U)$ an algebra automorphism, and $R^{(1)}\in U\ot U$ an invertible element such that the following relations hold 
\begin{gather}
\big( \Ad(R^{(1)}) \circ \cR^{(0)} \big) \circ \kow = \kow^{\mathrm{op}}, \label{eq:RR1} \\
(\kow \otimes \id) \circ \cR^{(0)} = \cR^{(0)}_{13} \circ \cR^{(0)}_{23} \circ (\kow \otimes \id), \label{eq:RR2}\\ 
(\id \otimes \kow) \circ \cR^{(0)} = \cR^{(0)}_{13} \circ \cR^{(0)}_{12} \circ (\kow \otimes \id), \label{eq:RR3}\\
(\kow \otimes \id) (R^{(1)}) = R^{(1)}_{13} \cdot \cR^{(0)}_{13} (R^{(1)}_{23}), \label{eq:RR4}\\ 
(\id \otimes \kow) (R^{(1)}) = R^{(1)}_{13} \cdot \cR^{(0)}_{13} (R^{(1)}_{12}).\label{eq:RR5}
\end{gather}
Then $(U,\Ad(R^{(1)})\circ \cR^{(0)})$ is a weakly quasitriangular Hopf algebra.
\end{lem}
%%%%%%%%%%%%%%%%%%%%%%%%%%%%%%%%%%%%%%%%%%%%%%%%%%%%%%
The construction of weakly quasitriangular Hopf algebras in the theory of quantum groups involve completions, see \cite[4.3]{a-Tanisaki92}, \cite[1.3]{a-Reshetikhin95}. We set up these completions in a way which also works for the weakly quasitriangular coideal subalgebras in Section \ref{sec:wq-coid}. Recall that $U(\chi)=U(\chi)_\mx$ is the Drinfeld double of a Nichols algebra of diagonal type $U^+=U^+_\mx$. Let $B$ be an arbitrary algebra and consider a finite sequence of signs $s_1,\dots,s_m\in \{+,-\}$ for some $m\in \N$. For any $j\in \N$ define
\begin{align*}
\big( B \otimes U(\chi)U^{s_1} \otimes \cdots \otimes U(\chi)U^{s_m} \big)_j = 
\bigoplus_{\beta_1, \ldots, \beta_m \in \N^n, \atop \sum_{i=1}^m |\beta_i| \geq  j }  U(\chi) U^{s_1}_{s_1\beta_1} \otimes \cdots \otimes U(\chi) U^{s_m}_{s_m\beta_m}.
\end{align*}
Then the inverse limit
\begin{align*}
  \wh{(B \otimes U(\chi)^{\otimes m})}_{s_1 \ldots s_m} : = \varprojlim_{j\in \N} \Big(\big( B \otimes U(\chi)^{\otimes m} \big) 
/ \big( B \otimes U(\chi)U^{s_1} \otimes \cdots \otimes U(\chi)U^{s_m} \big)_j\Big)
\end{align*}  
is an algebra which contains $B\otimes U(\chi)^{\otimes m}$ as a subalgebra.
If the algebra $B$ coincides with the field $\field$ then we write
$\wh{( U(\chi)^{\otimes m})}_{s_1 \ldots s_m}$ instead of  $\wh{(\field \otimes U(\chi)^{\otimes m})}_{s_1 \ldots s_m}$. The coproduct $\kow$ extends to the inverse limits. For example, we have algebra homomorphisms
\begin{align*}
   (\kow \ot \id)\colon &\wh{( U(\chi)^{\otimes 2})}_{s_1 s_2}\rightarrow \wh{( U(\chi)^{\otimes 3})}_{s_1 s_1 s_2},\\
  (\id \ot \kow)\colon &\wh{( U(\chi)^{\otimes 2})}_{s_1 s_2}\rightarrow \wh{( U(\chi)^{\otimes 3})}_{s_1 s_2 s_2}
\end{align*}
which canonically extend $\kow \ot \id, \id\ot \kow: U(\chi)^{\ot 2}\rightarrow U(\chi)^{\ot 3}$.
Recall that $\Theta$ denotes the quasi $R$-matrix defined by \eqref{eq:Theta}. For $s_1s_2\in\{++, +-, --\}$ we may consider $\Theta_{21}=\sum_\mu(-1)^{|\mu|} E_\mu\ot F_\mu$ as an invertible element of $\wh{( U(\chi)^{\otimes 2})}_{s_1 s_2}$. Moreover, there is a well defined  algebra automorphism $\cR^{(0)}\in \Aut(U(\chi)^{\ot 2})$ such that
\begin{align}\label{eq:cR0-def}
  \cR^{(0)} |_{U(\chi)_{\beta} \otimes U(\chi)_\gamma } = 
\chi(\beta,\gamma) (K_{- \gamma} \cdot) \otimes (K_{- \beta} \cdot)
\end{align}  
for all $\beta, \gamma \in \Z^n$. Here $K_{-\gamma}\cdot$ and $K_{-\beta}\cdot$ denote the operators of left multiplication by $K_{-\gamma}$ and $K_{-\beta}$, respectively. In terms of generators of the algebra $U(\chi)^{\ot 2}$ the algebra automorphism $\cR^{(0)}$ is given by $\cR^{(0)}|_{H\ot H}=\id_{H\ot H}$ and
\begin{align*}
  \cR^{(0)}(E_i\ot 1)&=E_i\ot K_i^{-1}, & \cR^{(0)}(1\ot E_i)&= K_i^{-1}\ot E_i,\\
  \cR^{(0)}(F_i\ot 1)&=F_i\ot K_i, & \cR^{(0)}(1\ot F_i)&= K_i\ot F_i
\end{align*}
for all $i\in I$.
The automorphism $\cR^{(0)}$ extends canonically to an automorphism of the completion $\wh{( U(\chi)^{\otimes 2})}_{s_1 s_2}$. We can also make use of the leg notation to obtain algebra automorphism $\cR^{(0)}_{ij}$ of $\wh{( U(\chi)^{\otimes m})}_{s_1 \ldots s_m}$.

The following theorem states that the Drinfeld double $U(\chi)_\mx$ is weakly quasitriangular up to completion. The theorem hence extends \cite[Proposition 1.3.1]{a-Reshetikhin95}, \cite[Theorem 3.1]{a-Gav97} from the setting of quantum groups to Drinfeld doubles of general Nichols algebras of diagonal type. To simplify notation, we mostly drop the subscript $_\mx$.

%%%%%%%%%%%%%%%%%%%%%%%%%%%%%%%%%%%%%%%%%%%%%%%%%%%%%
\begin{thm}\label{thm:wq-tri}
  Let $s_1s_2\in \{++,+-,--\}$ and let $U^+$ be a Nichols algebra of diagonal type with Drinfeld double $U(\chi)$.
  \begin{enumerate}
  \item The element $R^{(1)}=\Theta_{21}$ and the automorphism $\cR^{(0)}\in \Aut(\wh{( U(\chi)^{\otimes 2})}_{s_1 s_2})$ defined by \eqref{eq:cR0-def} satisfy the relations \eqref{eq:RR1} -- \eqref{eq:RR5}.
  \item Define an algebra automorphism $\cR^{s_1s_2}\in \Aut(\wh{( U(\chi)^{\otimes 2})}_{s_1 s_2})$ by $\cR^{s_1s_2}=\Ad(\Theta_{21})\circ \cR^{(0)}$. Then $\cR^{s_1s_2}$ satisfies relations \eqref{eq:wq1} -- \eqref{eq:wq3}.
  \end{enumerate}  
\end{thm}
%%%%%%%%%%%%%%%%%%%%%%%%%%%%%%%%%%%%%%%%%%%%%%%%%%%%%
\begin{proof}
  (1) It suffices to check \eqref{eq:RR1} on the generators $E_i, F_i, K_i$. Hence property \eqref{eq:RR1} follows from Corollary \ref{cor:EFTheta}. Properties \eqref{eq:RR2} and \eqref{eq:RR3} hold because the coproduct preserves weights. Finally, properties \eqref{eq:RR4} and \eqref{eq:RR5} hold by Lemma \ref{lem:kowidTheta}.

  (2) This follows from (1) analogously to the proof of Lemma \ref{lem:weak-quasi}.
\end{proof}  
%%%%%%%%%%%%%%%%%%%%%%%%%%%%%%%%%%%%%%%%%%%%%%%%%%%%%
Analogously to Remark \ref{rem:qYB}, the second part of Theorem \ref{thm:wq-tri} implies that for $s_1s_2 s_3 \in \{ +++, ++-, +--, ---\}$ the quantum Yang-Baxter equation
\[
\cR_{12}^{s_1 s_2} \circ \cR_{13}^{s_1 s_3} \circ \cR_{23}^{s_2 s_3} = \cR_{23}^{s_2 s_3} \circ \cR_{13}^{s_1 s_3} \circ \cR_{12}^{s_1 s_2}
\]
holds on $\wh{(U(\chi)^{\otimes 3})}_{s_1s_2s_3}$.
%%%%%%%%%%%%%%%%%%%%%%%%%%%%%%%%%%%%%%%%%%%%%%%%%%%%%%%%%%%%%%%%%%%%%%%%%%%
\subsection{Extending $\osigma$ to an algebra automorphism}\label{sec:extending-osigma}
%%%%%%%%%%%%%%%%%%%%%%%%%%%%%%%%%%%%%%%%%%%%%%%%%%%%%%%%%%%%%%%%%%%%%%%%%%%
From now on we assume that the parameters satisfy $c_i\neq 0$ for all $i\in I$. Under this assumption the algebra homomorphism $\osigma:U^-\rightarrow U^+\rtimes H$ from Section \ref{sec:osigma1} can be extended to an algebra automorphism of $U(\chi)$. Indeed, it follows from the defining relations \eqref{eq:def-rel} and from Lemma \ref{lem:EKF} that there is a well-defined algebra automorphism  
$\osigma:U(\chi)\rightarrow U(\chi)$ such that
\begin{align}\label{eq:osigma-extend}
  \osigma(E_i)&=c_{\tau(i)}^{-1} F_{\tau(i)}K_i^{-1}, &
  \osigma(F_i)&=c_{\tau(i)} K_i E_{\tau(i)}, & \osigma(K_i)&=K_{\tau(i)}^{-1} 
\end{align}
for all $i\in I$. We are interested in the compatibility between $\osigma$ and the coproduct. In the following lemma $\cR^{(0)}$ denotes the algebra automorphism of $U(\chi)^{\ot 2}$ given by \eqref{eq:cR0-def} and $\Theta$ denotes the quasi $R$-matrix for $U(\chi)$.
%%%%%%%%%%%%%%%%%%%%%%%%%%%%%%%%%%%%%%%%%%%%%%%%%%%%%%%%%%%%%%%%%%%%%%%%%%%%
\begin{lem}\label{lem:osigma-kow}
  Let $U^+$ be a Nichols algebra of diagonal type with Drinfeld double $U(\chi)$. Assume that $\bc\in (\field^\times)^n$. The algebra automorphism $\osigma$ satisfies the relation
  \begin{align}\label{eq:osigma-kow}
     \kow \circ \osigma =(\osigma\ot \id)\circ \cR^{(0)}_{21} \circ (\id \ot \osigma) \circ \Ad(\Theta_{21})\circ \cR^{(0)}\circ \kow.
  \end{align}  
\end{lem}  
%%%%%%%%%%%%%%%%%%%%%%%%%%%%%%%%%%%%%%%%%%%%%%%%%%%%%%%%%%%%%%%%%%%%%%%%%%%%
\begin{proof}
  It suffices to show that both sides of \eqref{eq:osigma-kow} coincide when evaluated on $K_i, E_i$ and $F_i$. Evaluated on $K_i$ both sides give $K_{\tau(i)}^{-1}\ot K_{\tau(i)}^{-1}$. By Equation \eqref{eq:EFTheta1} we have
  \begin{align*}
    (\osigma\ot \id)\circ \cR^{(0)}_{21} \circ (\id \ot \osigma)& \circ \Ad(\Theta_{21})\circ \cR^{(0)}\circ \kow (E_i)\\
    &=(\osigma\ot \id)\circ \cR^{(0)}_{21} \circ (\id \ot \osigma)\big(E_i\ot K_i + 1 \ot E_i\big)\\
    &=\osigma(E_i)\ot K_i K_{\tau(i)}^{-1} + K_i^{-1}\ot \osigma(E_i)\\
    &=\kow \circ \osigma(E_i).
  \end{align*}
  The calculation for $F_i$ is similar.
\end{proof}  
%%%%%%%%%%%%%%%%%%%%%%%%%%%%%%%%%%%%%%%%%%%%%%%%%%%%%%%%%%%%%%%%%%%%%%%%%%%%

%%%%%%%%%%%%%%%%%%%%%%%%%%%%%%%%%%%%%%%%%%%%%%%%%%%%%%%%%%%%%%%%%%%%%%%%%%%%%%%%
\subsection{Weakly quasitriangular comodule algebras}\label{sec:wq-coid}
%%%%%%%%%%%%%%%%%%%%%%%%%%%%%%%%%%%%%%%%%%%%%%%%%%%%%%%%%%%%%
We now introduce a weak version of quasitriangularity for comodule algebras over weakly quasitriangular Hopf algebras. 
%%%%%%%%%%%%%%%%%%%%%%%%%%%%%%%%%%%%%%%%%%%%%%%%%%%%%%%%%%%%
\begin{defi}\label{def:w-quasi-coid}
Let $(U,\cR)$ be a weakly quasitriangular Hopf algebra. A weakly quasitriangular right comodule algebra over $(U,\cR)$ is a triple $(B,\kow_B,\cK)$ where $B$ is a right $U$-comodule algebra with coaction $\kow_B:B\rightarrow B\ot U$ and $\cK$ is an algebra automorphism of $B\ot U$ which satisfies the following properties
\begin{gather}
\cK \circ \Delta_B = \Delta_B \qquad \mbox{on $B$,}\label{eq:wq-comod1}\\
(\Delta_B \otimes \id) \circ \cK = \cR_{32} \circ \cK_{13} \circ \cR_{23} \circ (\Delta_B \otimes \id) \qquad \mbox{on $B\ot U$,}\label{eq:wq-comod2} \\
(\id \otimes \Delta) \circ \cK = \cR_{32} \circ \cK_{13} \circ \cR_{23} \circ \cK_{12} \circ (\id \otimes \Delta)\qquad \mbox{on $B\ot U$.}\label{eq:wq-comod3}
\end{gather}
We say that the comodule algebra $B$ is weakly quasitriangular if the coaction $\kow_B$ and the automorphism $\cK$ are understood.
\end{defi}
%%%%%%%%%%%%%%%%%%%%%%%%%%%%%%%%%%%%%%%%%%%%%%%
Remarks \ref{rem:qwq} and \ref{rem:qYB} have analogs for comodule algebras over a Hopf algebra.
%%%%%%%%%%%%%%%%%%%%%%%%%%%%%%%%%%%%%%%%%%%%%%%
\begin{rema}
  Let $U$ be a quasitriangular Hopf algebra with universal $R$-matrix $R$. By Remark \ref{rem:qwq} the pair $(U,\Ad(R))$ is a weakly quasitriangular Hopf algebra. Recall the definition of a quasitriangular comodule algebra $B$ over $U$ with universal $K$-matrix $K\in B\ot U$ from \cite[Definition 2.7]{a-Kolb17p}. If the $U$-comodule algebra $B$ is quasitriangular then $B$ is weakly quasitriangular with the automorphism $\cK=\Ad(K)$ of $B\ot U$. 
\end{rema}
%%%%%%%%%%%%%%%%%%%%%%%%%%%%%%%%%%%%%%%%%%%%%%%
\begin{rema}\label{rem:reflection}
  If $(B,\cK)$ is a weakly quasitriangular comodule algebra over a weakly quasitriangular Hopf algebra $(U,\cR)$ then the automorphisms $\cK$ and $\cR$ satisfy the reflection equation
  \begin{align}\label{eq:refl}
    \cK_{12} \circ \cR_{32}\circ \cK_{13}\circ \cR_{23}=\cR_{32}\circ \cK_{13}\circ \cR_{23}\circ \cK_{12}
  \end{align}
  on $B\ot U\ot U$. Indeed, \eqref{eq:wq-comod1} and \eqref{eq:wq-comod2} imply that
  \begin{align*}
    \cK_{12} \circ \cR_{32}\circ \cK_{13}\circ \cR_{23}\circ (\kow_B\ot \id)=\cR_{32}\circ \cK_{13}\circ \cR_{23}\circ \cK_{12}\circ (\kow_B\ot \id)
  \end{align*}
  on $B\ot U$ while \eqref{eq:wq-comod1} and \eqref{eq:wq-comod3} imply that
  \begin{align*}
    \cK_{12} \circ \cR_{32}\circ \cK_{13}\circ \cR_{23}\circ (\id \ot\kow)=\cR_{32}\circ \cK_{13}\circ \cR_{23}\circ \cK_{12}\circ (\id \ot\kow)
  \end{align*}
  on $B\ot U$. Now the reflection equation \eqref{eq:refl} follows from the fact that if $U$ is a Hopf algebra then $\mathrm{Im}(\kow_B\ot \id)+ \mathrm{Im}(\id\ot \kow)$ generates $B\ot U^{\ot 2}$ as an algebra.
\end{rema}
%%%%%%%%%%%%%%%%%%%%%%%%%%%%%%%%%%%%%%%%%%%%%%%
We have the following analog of Lemma \ref{lem:weak-quasi} for comodule algebras.
%%%%%%%%%%%%%%%%%%%%%%%%%%%%%%%%%%%%%%%%%%%%%%%
\begin{lem}\label{lem:wq-KK}
  Let $(U,\cR^{(0)}, R^{(1)})$ be as in Lemma \ref{lem:weak-quasi} and let $B$ be a right $U$-comodule algebra with coaction $\kow_B:B\rightarrow B\ot U$. Let $\cK^{(0)}$ be an algebra automorphism of $B\ot U$ and let $K^{(1)}\in B\ot U$ be an invertible element satisfying the following relations
  \begin{gather}
    \Ad(K^{(1)})\circ \cK^{(0)}\circ \kow_B=\kow_B,\label{eq:wqKK1}\\
    (\kow_B\ot \id)\circ \cK^{(0)} = \cR^{(0)}_{32}\circ \cK^{(0)}_{13}\circ \cR^{(0)}_{23}\circ (\kow_B\ot \id),\label{eq:wqKK2}\\
    (\id \ot \kow)\circ \cK^{(0)}= \cK^{(0)}_{12}\circ \cR^{(0)}_{32}\circ \cK^{(0)}_{13}\circ \Ad(R^{(1)}_{23})\circ \cR^{(0)}_{23}\circ (\id \ot \kow), \label{eq:wqKK3}\\
    (\kow_B\ot \id)(K^{(1)})= R^{(1)}_{32} \cdot \cR^{(0)}_{32}(K^{(1)}_{13})\cdot \cR^{(0)}_{32}\cK^{(0)}_{13}(R^{(1)}_{23}),\label{eq:wqKK4}\\
    (\id \ot\kow)(K^{(1)})=K^{(1)}_{12}\cdot \cK^{(0)}_{12}(R^{(1)}_{32})\cdot \cK^{(0)}_{12}\cR^{(0)}_{32}(K^{(1)}_{13}).\label{eq:wqKK5}
  \end{gather}
Then $(B,\kow_B,\Ad(K^{(1)})\circ \cK^{(0)})$ is a weakly quasitriangular right comodule algebra over the weakly quasitriangular Hopf algebra $(U,\Ad(R^{(1)})\circ \cR^{(0)})$. 
\end{lem}
%%%%%%%%%%%%%%%%%%%%%%%%%%%%%%%%%%%%%%%%%%%%%%%
\begin{proof}
  Set $\cK=\Ad(K^{(1)})\circ \cK^{(0)}\in \Aut(B\ot U)$ and $\cR=\Ad(R^{(1)})\circ \cR^{(0)}\in \Aut(U^{\ot 2})$. Then Equation \eqref{eq:wq-comod1} follows from Equation \eqref{eq:wqKK1}. Equation \eqref{eq:wq-comod2} follows from Equations \eqref{eq:wqKK2} and \eqref{eq:wqKK4}, and  Equation \eqref{eq:wq-comod3} follows from Equations \eqref{eq:wqKK3} and \eqref{eq:wqKK5}. 
\end{proof}
%%%%%%%%%%%%%%%%%%%%%%%%%%%%%%%%%%%%%%%%%%%%%%%
We return to the concrete example of the coideal subalgebra $B_\bc$ of the Drinfeld double $U(\chi)=U(\chi)_\mx$ of a Nichols algebra $U^+=U^+_\mx$ of diagonal type.
There is a well defined algebra automorphism $\cK^{(0),\tau}\in \Aut(U(\chi)\ot U(\chi))$ such that  
\begin{align}\label{eq:K0tau}
  \cK^{(0),\tau}|_{U(\chi)_\beta\ot U(\chi)_\gamma}=\chi(\beta,\gamma{-}\tau(\gamma)) \, (K_{-\gamma+\tau(\gamma)}\cdot)\ot (K_{-\beta+\tau(\beta)}\cdot)
\end{align}
for all $\beta, \gamma\in \Z^n$. More explicitly, the algebra automorphism $\cK^{(0),\tau}$ is defined by $\cK^{(0),\tau}|_{H\ot H}=\id_{H\ot H}$ and
\begin{gather*}
  \cK^{(0),\tau}(1\ot E_i)=K_i^{-1}K_{\tau(i)}\ot E_i, \qquad
  \cK^{(0),\tau}(1\ot F_i)=K_i K_{\tau(i)}^{-1}\ot F_i, \\
   \cK^{(0),\tau}(E_i\ot 1)=E_i \ot K_i^{-1}K_{\tau(i)}, \qquad
  \cK^{(0),\tau}(F_i\ot 1)=F_i\ot K_i K_{\tau(i)}^{-1} 
\end{gather*}
for all $i\in I$. Similarly to the proof of Equations \eqref{eq:RR2}, \eqref{eq:RR3} for the automorphism $\cR^{(0)}$ given by \eqref{eq:cR0-def}, one sees that
\begin{align}
  (\kow\ot \id)\circ \cK^{(0),\tau} &=  \cK^{(0),\tau}_{23} \circ \cK_{13}^{(0),\tau}\circ (\kow \ot \id),\label{eq:KK0tau1}\\
   (\id \ot \kow)\circ \cK^{(0),\tau} &=  \cK^{(0),\tau}_{12} \circ \cK_{13}^{(0),\tau}\circ (\id \ot \kow).\label{eq:KK0tau2}
\end{align}  
The algebra automorphism $\cK^{(0),\tau}$ restricts to an automorphism of the subalgebra $B_\bc\ot U(\chi)$ such that
\begin{gather*}  
  \cK^{(0),\tau}(B_i\ot 1)=B_i \ot K_i K_{\tau(i)}^{-1} 
\end{gather*}  
for all $i\in I$. Recall the algebra automorphism $\osigma$ from Section \ref{sec:extending-osigma}. Define an algebra automorphism $\cK^{(0)}$ of $B\ot U(\chi)$ by
\begin{align*}
   \cK^{(0)}=\cK^{(0),\tau}\circ (\id \ot \osigma).
\end{align*}
By construction $\cK^{(0)}$ extends to algebra isomorphisms
\begin{align*}
  \cK^{(0)}_- \colon \wh{(B\ot U(\chi))}_- \rightarrow  \wh{(B\ot U(\chi))}_+,\qquad
  \cK^{(0)}_+ \colon \wh{(B\ot U(\chi))}_+ \rightarrow  \wh{(B\ot U(\chi))}_-.
\end{align*}
The isomorphism $\cK^{(0)}_-$ will provide us with the desired completed version of the automorphism $\cK^{(0)}$ in Lemma \ref{lem:wq-KK}. To obtain a completed version of $K^{(1)}$, we may consider the element $\Theta^\theta=\sum_\mu (-1)^{|\mu|} \psi^{-1}(F_\mu)\ot E_\mu$ from \eqref{eq:quasiK} as an invertible element in $\wh{(B\ot U(\chi))}_+$. By the following theorem the coideal subalgebra $B_\bc$ of $U(\chi)$ is weakly quasitriangular up to completion.
%%%%%%%%%%%%%%%%%%%%%%%%%%%%%%%%%%%%%%%%%%%%%%%%%%%%%%%%%%%%%%%
\begin{thm}\label{thm:wq-tri2}
  Let $U^+$ be a Nichols algebra of diagonal type with Drinfeld double $U(\chi)$. Let $B_\bc$ be the coideal subalgebra defined in Section \ref{sec:Bc} and assume that the parameters $\bc\in (\field^\times)^n$ satisfy condition \eqref{assume:parameters} in Section \ref{sec:prop-c}.
  Then the following hold:
  \begin{enumerate}
    \item The element $K^{(1)}=\Theta^\theta\in \wh{(B\ot U(\chi))}_+$ and the isomorphism $\cK^{(0)}=\cK^{(0)}_-=\cK^{(0),\tau}\circ (\id \ot \osigma)\colon \wh{(B\ot U(\chi))}_- \rightarrow  \wh{(B\ot U(\chi))}_+$ defined by \eqref{eq:K0tau} and \eqref{eq:osigma-extend} satisfy relations \eqref{eq:wqKK1} -- \eqref{eq:wqKK5}.
    \item Define an isomorphism of algebras $\cK^-\colon \wh{(B\ot U(\chi))}_- \rightarrow  \wh{(B\ot U(\chi))}_+$ by $\cK^-=\Ad(\Theta^\theta)\circ \cK^{(0)}_-$. Then $\cK^-$ satisfies relations \eqref{eq:wq-comod1} -- \eqref{eq:wq-comod3} with the operators $\cR^{s_1s_2}$ from Theorem \ref{thm:wq-tri}.
  \end{enumerate}  
\end{thm}
%%%%%%%%%%%%%%%%%%%%%%%%%%%%%%%%%%%%%%%%%%%%%%%%%%%%%%%%%%%%%%%%%%%%%%%
\begin{proof}
 (1)  We first verify \eqref{eq:wqKK1}. It suffices to check \eqref{eq:wqKK1} on the generators $B_i$ for $i\in I$ and $K_\lambda$ for $\lambda\in \Z^n_\theta$. We calculate
  \begin{align*}
    \Ad(\Theta^\theta)&\circ \cK^{(0),\tau}\circ (\id \ot \osigma)\circ \kow(B_i)\\
    &= \Ad(\Theta^\theta)\circ \cK^{(0),\tau}\big( B_i\ot K_{\tau(i)} + K_{\tau(i)} K_i^{-1}\ot F_i + 1 \ot c_{\tau(i)} K_i E_{\tau(i)}\big)\\
    &=\Ad(\Theta^\theta)\big(B_i\ot K_i + 1 \ot F_i + c_{\tau(i)} q_{i \tau(i)} K_{\tau(i)}^{-1} K_i \ot E_{\tau(i)} K_i\big)\\
    &=\kow(B_i)
  \end{align*}
  where the last equality follows from the intertwiner property \eqref{eq:quasiK-comm1}. The relation
  \begin{align*}
     \Ad(\Theta^\theta)&\circ \cK^{(0),\tau}\circ (\id \ot \osigma)\circ \kow(K_\lambda)=\kow(K_\lambda) \qquad \mbox{for $\lambda\in \Z^n_\theta$}
  \end{align*}
  holds as $\osigma(K_\lambda)=K_\lambda$ for all $\lambda\in \Z^n_\theta$. This completes the proof of \eqref{eq:wqKK1}.

  Property \eqref{eq:wqKK2} follows from the fact that
  \begin{align*}
      \cK^{(0),\tau}\circ (\id \ot \osigma) = \cR^{(0)}_{32}\circ (\id \ot \osigma) \circ \cR^{(0)}_{23}
  \end{align*}
  and from Equation \eqref{eq:KK0tau1}.  Property \eqref{eq:wqKK3} follows from Equation \eqref{eq:KK0tau2} and from Lemma \ref{lem:osigma-kow}. Finally, Equations  \eqref{eq:wqKK4} and  \eqref{eq:wqKK5} hold by Proposition \ref{prop:kowidThetatheta}.

 (2) This follows from (1) analogously to the proof of Lemma \ref{lem:wq-KK}.
\end{proof}  
%%%%%%%%%%%%%%%%%%%%%%%%%%%%%%%%%%%%%%%%%%%%%%%%%%%%%%%%%%%%%%%%%%%%%%%
Analogously to Remark \ref{rem:reflection}, the second part of Theorem \ref{thm:wq-tri2} implies that $\cK^-$ satisfies the reflection equation
\begin{align*}
   \cK_{12}^- \circ \cR_{32}^{+-}\circ \cK_{13}^{-}\circ \cR_{23}^{--}=\cR_{32}^{++}\circ \cK_{13}^{-}\circ \cR_{23}^{+-}\circ \cK_{12}^-
\end{align*}
as an  equality of algebra isomorphisms $\wh{(B\ot U(\chi)^{\ot 2})}_{--} \rightarrow \wh{(B\ot U(\chi)^{\ot 2})}_{++}$.
%%%%%%%%%%%%%%%%%%%%%%%%%%%%%%%%%%%%%%%%%%%%%%%%%%%%%%%%%%%%%%%%%%%%%%%
%\bibliographystyle{amsalpha}
%\bibliography{litbank2}

\providecommand{\bysame}{\leavevmode\hbox to3em{\hrulefill}\thinspace}
\providecommand{\MR}{\relax\ifhmode\unskip\space\fi MR }
\providecommand{\MRhref}[2]{%
  \href{http://www.ams.org/mathscinet-getitem?mr=#1}{#2}
}
\providecommand{\href}[2]{#2}

%%%%%%%%%%%%%%%%%%%%%%%%%%%%%%%%%%%%%%%%%%%%%%%%%%%%%%%%%%
\end{document}